\documentclass[11pt, leqno]{amsart}
\usepackage[dvipsnames]{xcolor}
\usepackage{amsfonts,amssymb, amscd}
\usepackage{amsmath}
\usepackage{lmodern}
\usepackage{mathrsfs}
\usepackage{amsthm}
\usepackage{qsymbols}
\usepackage{mathtools}
\mathtoolsset{showonlyrefs}
\usepackage{dsfont}
\usepackage{graphicx}
\usepackage{latexsym}
\usepackage{chngcntr}
\usepackage[noadjust]{cite}
\usepackage{paralist}
\usepackage[parfill]{parskip}
\usepackage{bm}
\usepackage{tikz,pgfplots,subcaption}
\usetikzlibrary{calc}
\usepackage{nicefrac}
\usepackage{float}
\usepackage{csquotes}
\usepackage{esint}
\usepackage{multirow}
\usepackage{todonotes}
\usepackage[shortlabels]{enumitem}
\setlist[enumerate,1]{label = \normalfont(\arabic*), ref = (\arabic*)}
\usepackage{cancel}
\usepackage[markup=underlined]{changes}
\usepackage{comment}
\newcommand{\keepcomment}{0}%
\usepackage{euflag}

\AtBeginDocument{\ifnum\keepcomment=1
	\excludecomment{comment}
	\else
	\includecomment{comment}
	\fi}

\newtheorem{theorem}{Theorem}[section]
\newtheorem{lemma}[theorem]{Lemma}
\newtheorem{proposition}[theorem]{Proposition}
\newtheorem{corollary}[theorem]{Corollary}

\theoremstyle{definition}
\newtheorem{definition}[theorem]{Definition}
\newtheorem{remark}[theorem]{Remark}
\newtheorem{example}[theorem]{Example}
\newtheorem*{theorem-nonum}{Theorem}

\usepackage[backref=page]{hyperref}
\hypersetup{
	plainpages=false,
	pdfpagelabels,
	colorlinks,
	linkcolor={cyan!90!black},
	citecolor={magenta},
	urlcolor={green!40!black},
	bookmarksdepth=2,
} %

\newcommand{\R}{\mathbb{R}}

\newcommand{\IZ}{\mathbb{Z}}
\renewcommand{\L}{\mathrm{L}}
\newcommand{\W}{\mathrm{W}}

\newcommand{\X}{\mathrm{X}}
\newcommand{\Y}{\mathrm{Y}}
\newcommand{\Z}{\mathrm{Z}}

\newcommand{\Zv}{{}_\mathrm{v}\mathrm{Z}}

\newcommand{\Lv}{{}_\mathrm{v}\mathrm{L}}
\newcommand{\cL}{\mathcal{L}}
\newcommand{\cD}{\mathcal{D}}
\newcommand{\cF}{\mathcal{F}}
\newcommand{\cM}{\mathcal{M}}

\renewcommand{\S}{\mathcal{S}}
\newcommand{\Soo}{\mathcal{S}_\infty}
\newcommand{\Ext}{\mathcal{E}}
\newcommand{\Sol}{\mathcal{R}}
\newcommand{\Soll}{\mathcal{L}}
\newcommand{\Cont}{\mathrm{C}}

\newcommand{\ann}{\mathrm{C}}
\newcommand{\WB}{\mathrm{W}}
\newcommand{\e}{\mathrm{e}}
\renewcommand{\d}{\,\mathrm{d}}
\newcommand{\ddt}{\,\frac{\mathrm{d}t}{t}}
\newcommand{\ddtau}{\,\frac{\mathrm{d}\tau}{\tau}}
\newcommand{\dds}{\,\frac{\mathrm{d}s}{s}}
\newcommand{\cc}{\mathrm{c}}

\renewcommand{\r}{\mathrm{r}}

\newcommand{\eps}{\varepsilon}
\newcommand{\B}{\mathrm{B}}

\newcommand{\vv}{{}_\mathrm{v}}

\renewcommand\Re{\operatorname{Re}}

\newcommand{\ind}{\mathbf{1}}

\DeclareMathOperator{\supp}{supp}
\DeclareMathOperator{\dist}{d}
\DeclareMathOperator{\diam}{diam}
\DeclareMathOperator{\dom}{D}
\DeclareMathOperator{\Div}{div}

\newcommand{\wt}{\widetilde}
\newcommand{\loc}{\mathrm{loc}}

\newcommand{\Id}{\mathrm{Id}}

\newcommand{\RD}{\mathrm{RD}}

\newcommand{\SobSet}{\mathcal{I}}

\def\Yint#1{\mathchoice
	{\YYint\displaystyle\textstyle{#1}}%
	{\YYint\textstyle\scriptstyle{#1}}%
	{\YYint\scriptstyle\scriptscriptstyle{#1}}%
	{\YYint\scriptscriptstyle\scriptscriptstyle{#1}}%
	\!\iint}
\def\YYint#1#2#3{{\setbox0=\hbox{$#1{#2#3}{\iint}$}
		\vcenter{\hbox{$#2#3$}}\kern-.51\wd0}}
\def\longdash{{-}\mkern-3.5mu{-}}
\def\tiltlongdash{\rotatebox[origin=c]{15}{$\longdash$}}
\def\tiltfiint{\Yint\longdash}
\def\tiltfiint{\Yint\tiltlongdash}

\makeatletter
\@namedef{subjclassname@2020}{%
	\textup{2020} Mathematics Subject Classification}
\makeatother

\title[Non-linear parabolic PDEs with rough coefficients and critical data]{Non-linear parabolic PDEs with rough coefficients and critical data: existence, uniqueness and regularity of weak solutions}

\author{Pascal Auscher}
\address{Université Paris-Saclay
\\ France--Australia Mathematical Sciences and Interactions ANU -- CNRS International Research Laboratory, Canberra, ACT 2601, Australia.}
\email{pascal.auscher@universite-paris-saclay.fr}
\author{Sebastian Bechtel}
\address{Université Paris-Saclay, CNRS \\ Laboratoire de Mathématiques d’Orsay \\ 91405 Orsay \\ France}
\email{sebastian.bechtel@universite-paris-saclay.fr}
\date{\today}
\dedicatory{}
\thanks{This project has received funding from the European Union’s Horizon 2020 research and innovation programme under the Marie Skłodowska-Curie grant agreement No 101034255 \euflag{}.
A CC-BY 4.0 \url{https://creativecommons.org/licenses/by/4.0/} public copyright license has been applied by the authors to the present document and will be applied to all subsequent versions up to the Author Accepted Manuscript arising from this submission.}
\begin{document}
	\begin{abstract}
		This article investigates the well-posedness of weak solutions to non-linear parabolic PDEs driven by rough coefficients with rough initial data in critical homogeneous Besov spaces. Well-posedness is understood in the sense of existence and uniqueness of maximal weak solutions in suitable weighted $\Z$-spaces in the absence of smallness conditions. We showcase our theory with an application to rough reaction--diffusion equations. Subsequent articles will treat further classes of equations, including equations of Burgers-type and quasi-linear problems, using the same approach. Our toolkit includes a novel theory of hypercontractive singular integral operators (SIOs) on weighted $\mathrm{Z}$-spaces and a self-improving property for super-linear reverse Hölder inequalities.
	\end{abstract}
	\maketitle

	\allowdisplaybreaks

	\setcounter{tocdepth}{1}
	\tableofcontents

	\section{Introduction}
	\label{sec:intro}

	\subsection{Motivation}
	\label{subsec:intro_motivation}

	This article aims to contribute to the broad study of nonlinear evolution equations such as the nonlinear heat equation and other reaction--diffusion systems, the Navier--Stokes equations and other fluid models, moving interface problems and geometric flows~\cite{Amann,Fujita,AW,KatoNS,KT,Ericksen_Leslie,Matioc,Fitzhugh,Nagumo,Koch-Lamm,DaPG}.

	We concentrate on second-order equations on $(0,\infty) \times \R^n$ of the abstract type
	\begin{align}
		\label{intro:nl}
		\partial_t u - \Div(a(t,x,u) \nabla u) = F(u,\nabla u), \quad u(0) = u_0.
	\end{align}
	Systems and higher-order equations will be outlined in Section~\ref{subsec:intro_outlook}.
	Our goal is to establish existence, regularity and uniqueness of local and global weak solutions when the diffusion matrix is rough, typically only measurable, and initial data are taken from scaling-critical smoothness spaces of negative regularity.
	Such a result is out of reach of the current state of the art. Hence, we develop a novel framework including new harmonic-analysis and PDE tools, which employs weighted $\Z$-spaces as solution spaces and homogeneous Besov spaces as initial data spaces. Then, we showcase our framework by obtaining existence, uniqueness and regularity of a maximal weak solution to the rough reaction--diffusion equation
	\begin{align}
		\tag{RD}
		\label{eq:rd}
		\partial_t u -\Div(A\nabla u) = \phi(u), \quad u(0) = u_0.
	\end{align}
	Here, the non-linearity $\phi(u)$ has a local Lipschitz property and polynomial growth of order $1+\rho$, and $A \colon \R^n \to \R^{n\times n}$ is a merely bounded, measurable, elliptic coefficient matrix.
	We state a precise well-posedness result for~\eqref{eq:rd} in Theorem~\ref{thm:rd_new} and restrict ourselves here to an informal version illustrating the flavor of our theory.
	\begin{theorem-nonum}[Well-posedness of~\eqref{eq:rd}]
		Fix the growth parameter $\rho > \nicefrac{2}{n}$ of the non-linearity $\phi$. Let $u_0 \in \dot\B^\alpha_{p,p}$, a scaling-critical Besov space with respect to the non-linearity. Then, there exists a weak solution $u$ of~\eqref{eq:rd} which is  unique and maximal among the solutions satisfying $$u \in \Z^{r,q}_{\frac{n}{2r} - \frac{1}{\rho}} \quad \& \quad \nabla u \in \Z^{p,2}_{\frac{n}{2p} - \frac{1}{\rho} - \frac{1}{2}}$$ for admissible choices of $r$ and $q$.
	\end{theorem-nonum}
	The exponent $q$ of the space $\Z^{r,q}_\beta$ expresses local integrability, and some $\L^q$-averages parameterized by the upper half-space enjoy a global $\L^r$-integrability at a scale determined by the weight parameter $\beta$. In other words, the scale of weighted $\Z$-spaces captures a local-global control on functions with possible different exponents $q$ and $r$.
	When $r = q$, the space $\Z^{r,q}_\beta$ coincides with the weighted Lebesgue space $\L^r_\beta$.
	For weak solutions it is crucial that $\nabla u$ belongs to a weighted $\Z$-space of locally square-integrable functions, whereas the local integrability parameter $q$ for $u$ enjoys greater flexibility below or above $2$.

	Note that~\eqref{eq:rd} recovers the classical non-linear heat equation describing reaction--diffusion processes in a homogeneous, isotropic medium when $A \equiv \Id$. To study reaction--diffusion processes in media such as biological tissue, compound materials or porous materials, the diffusivity of the respective medium has to be encoded by the matrix $A$. As an illustrative example, consider the spreading of a tumor in the brain~\cite{Tumor}. Then, the diffusivity jumps between areas of gray and white matter, so that $A$ is indeed a real bounded elliptic matrix which is only measurable, lacking any form of continuity.

	Other models for our theory are, for instance, problems of Burgers-type, quasi-linear problems, or the generalized harmonic map heat flow. These and others are outlined in Section~\ref{subsec:intro_outlook}. To keep the length of the present article reasonable, we develop our tools and approach in full generality, but postpone applications to other models than~\eqref{eq:rd} to subsequent work.

	Next, we discuss some central challenges appearing in our well-posedness theory and indicate how they are resolved within the framework of weighted $\Z$-spaces and homogeneous Besov spaces.

	\subsection{Hypercontractivity as a main tool}
	\label{subsec:intro_hyper}

	We ignore the initial condition $u(0) = u_0$ for the moment and take a look at the non-linear term $\phi(u)$ instead. Assume that $u$ belongs to the solution space $X \coloneqq \Z^{r,q}_\beta$. Then $\phi(u)$ does \emph{not} belong again to $X$ because $\phi$ is of super-linear growth. Consequently, the parabolic solution operator $(\partial_t - \Div(A\nabla))^{-1}$ has to compensate the loss in the integrability and weight parameters caused by the non-linearity in order to be able to map back into $X$. We call this regularizing effect \emph{hypercontractivity} throughout this article.
	To obtain it, one can rely on one of the two items:
	\begin{itemize}
		\item[(a)] parabolic Sobolev embeddings,
		\item[(b)] hypercontractivity of the heat semigroup on Lebesgue spaces.
	\end{itemize}
	In the presence of rough coefficients, (a) is not exploitable to obtain sufficient hypercontractivity, see Remark~\ref{rem:hyper_sobolev} below. Instead, we exploit (b) in our framework. To do so, we introduce a class of hypercontractive singular integral operators (SIOs) in Definition~\ref{def:SIO} which generalizes classical notions of SIOs.
	Our SIOs are bounded operators $\L^q$ to weighted $\L^r$ on the upper half-space, and they enjoy some decay properties. However, they do not, in general, possess a representation by a pointwise kernel (the \enquote{beyond Calderón--Zygmund}-regime), including solution operators with rough coefficients.
	By virtue of the hypercontractivity of heat semigroups (generated by $L \coloneqq - \Div(A\nabla)$), solution operators to parabolic problems with rough coefficients are hypercontractive SIOs, see Section~\ref{subsec:duhamel_cauchy}.
	Hypercontractivity of the solution operators follows from the following extrapolation result, see Theorem~\ref{thm:sio}.

	\begin{theorem-nonum}[Hypercontractive SIOs on weighted $\Z$-spaces]
		Let $S$ be an SIO of type $(p_0,q,r,\kappa,\infty)$, where $p_0,q,r\in (1,\infty]$ and $\kappa \in [0,1]$ satisfying $q \geq p_0$ and $r \geq q$. Then, for $p \in [p_0, \infty]$ and $\beta > -1$, the operator $S$ extends to a bounded operator $\Z^{p,q}_\beta \to \Z^{p,r}_{\beta + \kappa}$.
	\end{theorem-nonum}

	Usual SIOs correspond to the case $q = r$, and this is a special instance of our result.
	In applications to parabolic PDEs, the integrability parameter $r$ is restricted in terms of parabolic Sobolev conjugates of $q$.
	Owing to the theorem, the losses in the local integrability parameter $q$ and in the weight parameter $\beta$ caused by the non-linearity $\phi$ can indeed be compensated. To gain in the global integrability parameter $p$, we use the \enquote{Hardy--Sobolev}-type embedding in Proposition~\ref{prop:Z_space_embedding}.
	Hence, non-linear terms can be effectively controlled in weighted $\Z$-spaces.

	\begin{remark}[Limitations of parabolic Sobolev embeddings]
		\label{rem:hyper_sobolev}
		The parabolic Sobolev embeddings in~(a) are the foundation of the maximal regularity approach~\cite{HNVW3,DHP,MovingInterface,CL94,PSW,PW_Addendum}, in which one proves mixed weighted $\L^p_t(\L^q_x)$ estimates. This approach allows one to obtain good results for the Laplace operator (that is, when $A \equiv \Id$) or, by perturbation techniques, for $-\Div(A\nabla)$ when $A$ possesses some regularity~\cite{Dong-Kim-ARMA,Krylov}. However, when $A$ is only measurable, the parameter $q$ cannot be taken large for the operator $\nabla (\partial_t - \Div(A\nabla))^{-1}$ used in the maximal regularity approach.
		Therefore, since Sobolev embeddings become powerful only when the integrability parameter $q$ can be taken large,~(a) does not provide sufficient hypercontractivity in the case of rough coefficients. For instance, important cases such as the Allen--Cahn non-linearity $\phi(u) = -u^3 + u$ cannot be considered by the maximal regularity approach in the physical dimension $n = 3$ when the coefficients are rough, see for instance~\cite{AV_var}.
	\end{remark}

	\subsection{Rough initial data}
	\label{subsec:intro_rough_data}

	By previous work of the authors together with Haardt~\cite{Luca}, homogeneous Besov spaces of negative regularity admit a caloric characterization in weighted $\Z$-spaces. More precisely, one has $u_0 \in \dot \B^\alpha_{p,p}$ if and only if $\e^{t\Delta} u_0 \in \Z^{p,q}_{\nicefrac{\alpha}{2}}$ for all $q \in (1,\infty]$. When $-\Delta$ is replaced by $L \coloneqq -\Div(A\nabla)$, mild assumptions on $\alpha \in  (-1,0)$ and $p \in (1,\infty]$ guarantee that the free evolution $\Ext_L(u_0)$ of $u_0$ still belongs to $\Z^{p,q}_{\nicefrac{\alpha}{2}}$, see~\cite{AH3} for the case $q = 2$ and Section~\ref{sec:caloric} for $q\in (1,\infty]$.

	Moreover, the free evolution $\Ext_L(u_0)$ solves the initial value problem in the weak sense and $u_0$ is attained in the sense of Definition~\ref{def:initial_condition}.
	Consequently, homogeneous Besov spaces are the natural initial data spaces within our framework.

	\subsection{Solutions and their regularity}
	\label{subsec:intro_regularity}

	The principal notion of solutions for us is the one of weak solutions in the sense of PDEs. It is a flexible and operator-free notion, physically meaningful, and the initial data are attained in a natural sense. By definition, weak solutions have a locally square-integrable gradient. As said, this is captured by the condition $\nabla u \in \Z^{p,2}_{\nicefrac{n}{2p}-\nicefrac{1}{\rho} - \nicefrac{1}{2}}$ in our well-posedness result for~\eqref{eq:rd}.

	Besides weak solutions, we also work with \enquote{mild solutions} in weighted $\Z$-spaces as a first step towards existence, see Section~\ref{sec:mild}. They are well-suited to capitalize on the hypercontractivity results presented above.
	In particular, hypercontractivity of the solution operators allows us to show existence and uniqueness of mild solutions to non-linear PDEs, in particular to~\eqref{eq:rd}.
	An important difference between mild and weak solutions is that mild solutions do not demand any knowledge on their gradients.

	To bridge between mild and weak solutions, regularity theory for both types of solution plays a major role in our approach. Our results can be summarized as follows:
	\begin{itemize}
		\item If $u \in \Z^{r_0,q_0}_{\beta_0}$ is a mild solution to a non-linear PDE, then $u \in \Z^{r,q}_\beta$ for a wide range of parameters $r$, $q$ and $\beta$ dictated by the equation.
		\item If $u \in \Z^{r,q_0}_\beta$ is a weak solution to a non-linear PDE, then $u \in \Z^{r,q}_\beta$ for an improved value of $q$.
		\item Building on these regularity results, a mild--weak transference principle, identifying mild solutions as weak solutions and vice versa, holds.
	\end{itemize}

	We postpone the discussion of weak solutions to the next subsection. Mild solutions to~\eqref{eq:rd} satisfy by definition the identity $u = \Soll(\phi(u)) + \Ext(u_0)$, where $\Soll$ is a hypercontractive solution operator applied to the non-linearity $\phi(u)$, and $\Ext(u_0)$ is the free evolution of the initial datum $u_0$. Since the identity is implicit, an iteration argument using hypercontractivity of $\Soll$ yields self-improvement of $u$ to various weighted $\Z$-spaces (Proposition~\ref{prop:RD_bootstrapping}).

	\subsection{Weak solutions to non-linear PDEs}
	\label{subsec:intro_weak}

	Instead of developing a well-posedness theory for weak solutions from scratch, we transfer existence, regularity and uniqueness of mild solutions to weak solutions. This is the mild--weak transference principle just mentioned. To obtain existence of weak solutions, it suffices to identify mild solutions as weak solutions. As for uniqueness, we have to reverse the order of things and identify a weak solution as a mild solution. Then, uniqueness of mild solutions is transferred to weak solutions.

	For both directions, establishing $\nabla u \in \Z^{p,2}_\beta$ is the main obstacle, where $\beta \coloneqq \nicefrac{n}{2p} - \nicefrac{1}{\rho} - \nicefrac{1}{2}$. We need $\phi(u) \in \Z^{p,2_*}_\beta$ in order to leverage our SIO-theory. Here, $2_*$ is the lower parabolic Sobolev conjugate of $2$. The regularity theory for mild solutions ensures this property, and allows the passage from mild solutions to weak solutions. If $u$ is a weak solution, the local integrability of $u$ is in general not sufficient to guarantee this property for $\phi(u)$.

	To resolve this problem, we first establish reverse Hölder inequalities of \enquote{super-linear type} for weak solutions to~\eqref{eq:rd} in Lemma~\ref{lem:RH_rd}. They read
	\begin{align}
		\label{eq:intro_RH}
		\tag{RH}
		\left( \tiltfiint_{B} |s^{\nicefrac{1}{\rho}} u|^{2^*} \right)^\frac{1}{2^*} \lesssim \tiltfiint_{2B} |s^{\nicefrac{1}{\rho}} u| + \left( \tiltfiint_{2B} |s^{\nicefrac{1}{\rho}} u|^{2_*(1+\rho)} \right)^\frac{1}{2_*},
	\end{align}
	holding for suitable parabolic balls $B$ and $2^*$ is the upper parabolic Sobolev conjugate of $2$.
	But, the local exponent $2_*(1+\rho)$ might be too large to a priori control the last term.
	Hence, we need to lower it, which is the purpose of the following result, see Theorem~\ref{thm:rh_improvement}.

	\begin{theorem-nonum}[Self-improvement of super-linear RH-inequalities]
		The RH-inequality of super-linear type~\eqref{eq:intro_RH} self-improves: there exists $\theta > 1$ such that, for $\lambda > 1$ small enough, there holds
		\begin{align}
			\left( \tiltfiint_{B} |s^{\nicefrac{1}{\rho}} u|^{2^*} \right)^\frac{1}{2^*} \lesssim \left( \tiltfiint_{\lambda B} |s^{\nicefrac{1}{\rho}} u|^{q} \right)^\frac{\theta}{q} + \left( \tiltfiint_{\lambda B} |s^{\nicefrac{1}{\rho}} u|^{q} \right)^\frac{1}{q},
		\end{align}
		where $q \in (\RD_{-}(n,\rho), 2_* (1+\rho))$.
	\end{theorem-nonum}

	Given the non-linearity, the number $\RD_{-}(n,\rho) > 1$ is the minimal number for which we can implement hypercontractivity for~\eqref{eq:rd}. Note that in contrast with the linear case, $q$ cannot go down to zero here.
	At this stage, we can exploit the above inequality to improve the local integrability exponent for $u$ in weighted $\Z$-spaces. This puts us in a good position to prove uniqueness of $u$ as a weak solution.

	\subsection{Outlook: further non-linear models}
	\label{subsec:intro_outlook}

	Besides the reaction term $\phi(u)$, our framework applies to Burgers-type non-linearities $\Div(\Phi(u))$, where $\Phi(u)$ is a vector field of polynomial growth enjoying a local Lipschitz condition.
	Such non-linear terms in the case of quadratic growth were investigated in~\cite{Auscher_Frey}.
	Compared to the reaction--diffusion case, the ranges of parameters have to be adjusted to the different types of non-linearities. This has the effect that the end-point Besov space with $p = \infty$ can be reached, which is not the case for the reaction--diffusion equation presented here.
	For this, we  need to introduce separable subspaces $\vv\dot\B^\alpha_{\infty,\infty}$ and $\Zv^{\infty,q}_\beta$ of the respective Besov and weighted $\Z$-spaces, see Section~\ref{subsec:prelim_function_spaces}, and all linear results in Sections~\ref{subsec:SIO} to~\ref{sec:linear_wp} are formulated for them likewise.

	Moreover, using the linear theory for $\Zv^{\infty,q}_\beta$, quasi-linear equations of the type
	\begin{align}
		\partial_t u -\Div(a(u)\nabla u) = 0, \quad u(0) = u_0
	\end{align}
	can be treated, where $u_0$ is taken from a suitable subspace of $\L^\infty$. Note that $\L^\infty$ is a scaling-critical space for the non-linearity. Likewise, models such as the generalized harmonic map heat flow
	\begin{align}
		\partial_t u -\Div(A\nabla u) = u|\nabla u|^2, \quad u(0) = u_0,
	\end{align}
	fall into the scope of our framework, provided $A$ has minimal smoothness in the sense of Definition~\ref{def:regular_A}.
	The case $A = \Id$ was investigated in~\cite{Koch-Lamm} by using weighted tent-space norms.

	Beyond~\eqref{eq:rd}, it would be interesting to study concrete systems of reaction--diffusion equations, for instance of Fitzhugh--Nagumo~\cite{Fitzhugh,Nagumo} type. Our Theorem~\ref{thm:rd_new} does not directly apply to them as it only treats equations of reaction--diffusion type. However, due to the specific structure of such systems, we expect that our approach generalizes to them.

	Finally, it would be interesting to study fourth-order problems in future work. Our SIO theory is general enough to apply to them, and well-posedness of weak solutions to linear higher-order equations in divergence form was investigated for instance in~\cite{Zaton}.

	\subsection*{Acknowledgments}

	The authors are grateful to Luca Haardt and Hedong Hou for enriching discussions on the subject, and to Moritz Egert and Pierre Portal for valuable feedback.

	\subsection*{Notation}

	Put $\xi \cdot \eta \coloneqq \sum_{i=1}^n \xi_i \eta_i$ for two vectors $\xi, \eta \in \R^n$. The inner product of $\R^n$ can hence be written as $\langle \xi, \eta \rangle = \xi \cdot \eta$.
	Write $\R^{1+n}_+ \coloneqq (0,\infty) \times \R^n$ for the upper half-space.
	For $x\in \R^n$ and $t > 0$ define the \emph{parabolic annuli} by $\ann_1(x,t) \coloneqq \B(x,\sqrt{4t})$ and $\ann_j(x,t) \coloneqq \B(x, \sqrt{2^{j+1} t}) \setminus \B(x, \sqrt{2^j t})$ for $j \geq 2$.
	Write $\WB(t,x) \coloneqq (\nicefrac{t}{2}, t) \times \B(x, \sqrt{t})$ for the \emph{parabolic Whitney box} centered at $x$ of height $t$.
	For a function $f$ defined on $\R^{1+n}_+$, write $\pi(f) \coloneqq \pi_1(\supp(f)) \times \pi_2(\supp(f))$, where $\pi_1$ and $\pi_2$ are the projections of $\R^{1+n} = \R \times \R^n$ to $\R$ and $\R^n$, respectively.

	For $q \in [1,\infty]$ define the sets $\SobSet^*(q)$ and $\SobSet^{**}(q)$ through
	\begin{align}
		\SobSet^*(q) \coloneqq \begin{cases}
			[q, q^*], \quad &\text{if } q < n+2, \\
			[q, \infty), \quad &\text{if } q = n+2, \\
			[q,\infty], \quad &\text{if } q > n+2,
		\end{cases}
		\quad \text{and} \quad
		\SobSet^{**}(q) \coloneqq \begin{cases}
			[q, q^{**}], \quad &\text{if } q < 1 + \nicefrac{n}{2}, \\
			[q, \infty), \quad &\text{if } q = 1 + \nicefrac{n}{2}, \\
			[q,\infty], \quad &\text{if } q > 1 + \nicefrac{n}{2}.
		\end{cases}
	\end{align}
	Here $q^*$ is the upper parabolic Sobolev index of $q$ defined by $\nicefrac{1}{q^*} \coloneqq \nicefrac{1}{q} - \nicefrac{1}{n+2}$ for $q < n+2$ and we use also the lower parabolic Sobolev index of $q$ defined by $\nicefrac{1}{q_{*}} \coloneqq \nicefrac{1}{q} + \nicefrac{1}{n+2}$.

	The space $\Soo$ is the subspace of the Schwartz functions $f\in \S$ satisfying $\langle f, P \rangle = 0$ for any polynomial $P$. Equip $\Soo$ with the topology inherited from $\S$. Then, write $\Soo'$ for its continuous dual space. $\Soo'$ can be identified with $\S'/\mathcal{P}$, where $\S'$ denotes the tempered distributions and $\mathcal{P}$ is the space of all polynomials.

	For standard function spaces on $\R^n$ like $\L^p(\R^n)$ or $\W^{1,p}(\R^n)$ we omit the underlying set $\R^n$ from the notation but write an index $x$ instead, for instance $\L^p_x$ or $\W^{1,p}_x$. A shorthand-notation for spaces on (subsets of) the half-space $\R^{1+n}_+$ will be introduced in Section~\ref{subsec:prelim_function_spaces}.

	\section{Preliminaries}
	\label{sec:prelim}

	\subsection{Function spaces}
	\label{subsec:prelim_function_spaces}

	We introduce the relevant function spaces for this article. First, we introduce solution spaces on $\R^{1+n}_+$ or on strips of the form $(a,b) \times \R^n$. Second, we introduce spaces of initial data.

	\subsubsection*{Solution spaces}

	Our primary solution spaces are weighted $\Z$-spaces. They go back to~\cite{Barton_Mayboroda} and were discussed in more depth in~\cite{Amenta_Interpolation,Luca}.

	\begin{definition}[$\Z$-spaces]
		\label{def:Z_space}
		Let $p,q \in (1,\infty]$ and $\beta \in \R$. The space $\Z^{p,q}_\beta$ consists of all measurable functions $u \colon \R^{1+n}_+ \to \R$ such that
		\begin{align}
			\| u \|_{\Z^{p,q}_\beta} \coloneqq \Bigl( \iint_{\R^{1+n}_+} \Bigl( \fint_{t/2}^t \fint_{\B(x, \sqrt{t})} |s^{-\beta} u(s,y)|^q \d y \d s \Bigr)^\frac{p}{q} \d x \ddt \Bigr)^\frac{1}{p} < \infty,
		\end{align}
		with the obvious modifications if $p=\infty$ and/or $q=\infty$.
		For $T \in (0,\infty]$ write $\Z^{p,q}_\beta(T)$ for the space of functions consisting of $u \colon (0,T) \times \R^n \to \R$ whose extension by zero $E_0 u$ to a function on $\R^{1+n}_+$ belongs to $\Z^{p,q}_\beta$. Set $\| u \|_{\Z^{p,q}_\beta(T)} \coloneqq \| E_0 u \|_{\Z^{p,q}_\beta}$. If $\beta = 0$, we remove it from the notation, that is, we put $\Z^{p,q} \coloneqq \Z^{p,q}_0$.
	\end{definition}

	\begin{remark}
		\label{rem:Z_spaces}
		The range for the parameters $p$ and $q$ can be enlarged to $(0,\infty]$. However, this extension is not relevant for the present article.
		In~\cite{AH3}, weighted $\Z$-spaces are defined for $q = 2$.
		As we use the average integral $\fint_{t/2}^t$ in our definition, our space $\Z^{p,2}_\beta$ here corresponds to the space $\Z^{p,2}_{\beta + 1/2}$ there, as it was defined with the integral $\int_{t/2}^t$ instead.
	\end{remark}

	If $p = q$, Fubini's theorem reveals that weighted $\Z$-spaces coincide with weighted $\L^p$-spaces, where the latter are properly defined in the following definition. More precisely, there holds the identity $\Z^{p,p}_\beta = \L^p_\beta$.

	\begin{definition}[Weighted $\L^p$-spaces]
		Let $p \in (1,\infty]$, $\beta \in \R$ and $T \in (0,\infty]$. The space $\L^p_\beta$ consists of measurable functions $u \in \R^{1+n}_+ \to \R$ such that $(t,x) \mapsto t^{-\beta} u(t,x)$ belongs to $\L^p(\R^{1+n}_+, \d x \ddt)$, with norm $$\| u \|_{\L^p_\beta} \coloneqq \| (t,x) \mapsto t^{-\beta} u(t,x) \|_{\L^p(\R^{1+n}_+, \d x \ddt)}.$$
		Using either extension by zero or restriction of the range of integration,
		the space $\L^p_\beta(T)$ is defined analogously, likewise for $\L^p_\beta(\tau,T)$. If $\beta = 0$, we remove the index from the notation.
	\end{definition}

	By Jensen's inequality, there holds the contractive inclusion $$\Z^{p,q_0}_\beta \subseteq \Z^{p,q_1}_\beta, \quad q_0 \geq q_1.$$
	{This \emph{nesting property} will be useful in the treatment of non-linearities, as they force us to work with different values of $q$.}

	Let $T \in (0,\infty)$. %
	Observe the norm equivalences%
	\begin{align}
		\| u \|_{\Z^{p,q}_\beta(T)} \approx \Bigl\| (t,x) \mapsto \Bigl( \fint_{t/2}^t \fint_{\B(x, \sqrt{t})} |s^{-\beta} u(s,y)|^q \d y \d s \Bigr)^\frac{1}{q} \Bigr\|_{\L^p_\beta(T)}.
	\end{align}
	Indeed, if $t \geq 2T$, the (open) Whitney box at height $t$ is disjoint to the support of the zero extension $E_0 u$ of $u$ used in the definition of $\Z^{\infty,q}_\beta(T)$, whereas the terms for $T<t <2T$ can be absorbed into the ones for $t\le T$ by a covering argument, compare with~\cite[Prop.~3.7]{Luca}.

	From~\cite{Luca} we collect some further important properties of weighted $\Z$-spaces.

	The spaces $\Z^{p,q}_\beta$ for $p,q \in (1,\infty]$ and $\beta \in \R$ are Banach spaces. There hold the continuous inclusions
	\begin{align}
		\label{eq:Z_loc_embedding}
		\L^q_\cc(\R^{1+n}_+) \subseteq \Z^{p,q}_\beta \subseteq \L^q_\loc(\R^{1+n}_+).
	\end{align}
	Here, the index $\cc$ indicates functions with compact support in $\R^{1+n}_+$. If $p$ and $q$ are both finite, then the inclusion $\L^q_\cc(\R^{1+n}_+) \subseteq \Z^{p,q}_\beta$ is dense. If $p = \infty$ but $q$ is finite, one can show weak-*-density, see~\cite[Lem.~3.1]{AH2}, in the duality with $\Z^{1,q'}_{-\beta}$ for the duality form $\iint_{\R^{1+n}_+} f g \, \d x \ddt$.

	Furthermore, there holds a \enquote{Hardy--Sobolev}-type embedding for weighted $\Z$-spaces, which we recall in the following proposition. For a proof, see~\cite[Thm.~3.11]{Luca}.

	\begin{proposition}[\enquote{Hardy--Sobolev}-type embedding]
		\label{prop:Z_space_embedding}
		Let $p_0, p_1, q \in (1,\infty]$ and $\beta_0, \beta_1 \in \R$ satisfying $p_0 < p_1$ and $\beta_0 > \beta_1$. Suppose that $2\beta_0 - \nicefrac{n}{p_0} = 2\beta_1 - \nicefrac{n}{p_1}$. Then, there holds the continuous inclusion
		\begin{align}
			\Z^{p_0,q}_{\beta_0} \subseteq \Z^{p_1,q}_{\beta_1}.
		\end{align}
		The statement can be localized to a finite time interval, with inclusion constant independent of the interval length.
	\end{proposition}

	An important tool for the SIO-theory in Section~\ref{subsec:SIO} will be the following change of angle formula, whose proof is given in~\cite[Lem.~3.6]{Luca}.

	\begin{lemma}[Change of angle]
		\label{lem:change_of_angle}
		Let $p,q \in (1,\infty]$ and $\beta \in \R$. Put $\tau \coloneqq \min(p,q)$. Then, for $\lambda \geq 1$ and $u \in \Z^{p,q}_\beta$, there holds the estimate
		\begin{align}
			\Bigl( \iint_{\R^{1+n}_+} \Bigl( \fint_{t/2}^t t^{-\frac{n}{2}} \int_{\B(x, \lambda \sqrt{t})} |s^{-\beta} u(s,y)|^q \d y \d s \Bigr)^\frac{p}{q} \d x \ddt \Bigr)^\frac{1}{p} \lesssim \lambda^\frac{n}{\tau} \| u \|_{\Z^{p,q}_\beta},
		\end{align}
		with the obvious modifications if $p=\infty$ and/or $q=\infty$.
	\end{lemma}

	A first consequence is local $\L^q_{\beta}$ integrability for $\Z^{\infty,q}_\beta$ up to the boundary.

	\begin{lemma}
		\label{lem:Z_local_Lq}
		Fix $q\in (1, \infty]$, $\beta\in \R$ and $F \in \Z^{\infty,q}_{\beta}$. Let $x\in \R^n$ and $t > 0$. Put $Q \coloneqq (0, t) \times \B(x, \sqrt{t})$. Then, $F \in \L^q_{\beta}(Q)$.
	\end{lemma}

	\begin{proof}
		It is enough to deal with $\beta=0$.
		The case $q = \infty$ is trivial, thus we suppose $q < \infty$.
		Decompose the interval $(0,t)$ into dyadic chunks and use the change of angle formula to give
		\begin{align}
			\Bigl( \fint_0^t \fint_{\B(x,\sqrt{t})} |F|^q \Bigr)^\frac{1}{q} &\lesssim \sum_{\ell \geq 0} 2^{-\frac{\ell}{q}} 2^{-\frac{\ell n}{2q}} \Bigl( \fint_{2^{-\ell-1}t}^{2^{-\ell}t} (2^{-\ell} t)^{-\frac{n}{2}} \int_{\B(x,\sqrt{2^\ell} \sqrt{2^{-\ell} t})} |F|^q \Bigr)^\frac{1}{q} \\
			&\lesssim \sum_{\ell \geq 0} 2^{-\frac{\ell}{q}} \| F \|_{\Z^{\infty,q}}.
		\end{align}
		The sum in $\ell$ converges due to $q < \infty$. Multiplying the resulting bound by $t^\frac{n+2}{2q}$ gives the claim.
	\end{proof}

	\subsubsection*{Solution spaces vanishing near the initial time}

	In the case $p = \infty$, we introduce a subclass of functions that \enquote{vanishes near the initial time}. They will be important for non-linear applications beyond~\eqref{eq:rd} with initial data that are not small.

	\begin{definition}[$\Zv$-spaces]
		\label{def:Zv-spaces}
		Let $q\in (1,\infty]$ and $\beta \in \R$. The subspace $\Zv^{\infty,q}_\beta$ of $\Z^{\infty,q}_\beta$ consists of all $u \in \Z^{\infty,q}_\beta$ with the property:
		\begin{align}
			\label{eq:Zv_defining_property}
			\forall \eps > 0 \;\; \exists \tau = \tau(\eps,u) \colon \quad \| u \|_{\Z^{\infty,q}_\beta(\tau)} \leq \eps.
		\end{align}
		For $T \in (0,\infty)$, define the space $\Zv^{\infty,q}_\beta(T)$ analogously. Similarly, define the spaces $\Lv^\infty_\beta$ and $\Lv^\infty_\beta(T)$.
	\end{definition}

	\begin{remark}
		\label{rem:Zv}
		Observe that the defining property for weighted $\Zv$-spaces is always true for $\Z^{p,q}_\beta$ when $p < \infty$ and $q$ is arbitrary.
	\end{remark}

	From {the remark and} the localization properties stated in Proposition~\ref{prop:Z_space_embedding}, we derive the following embedding into $\Zv$-spaces.

	\begin{corollary}[{Weighted $\Zv$-space embedding}]
		\label{cor:embedding_into_Zv}
		In the situation of Proposition~\ref{prop:Z_space_embedding}, when $p_1 = \infty$ and hence $\beta_1 = \beta_0 - \nicefrac{n}{2p_0}$, then there holds the continuous inclusion $\Z^{p_0,q}_{\beta_0} \subseteq \Zv^{\infty,q}_{\beta_1}$.
	\end{corollary}

	\subsubsection*{Initial data spaces}

	We use homogeneous Besov spaces $\dot\B^\alpha_{p,p}$ as spaces of initial data. Here, we simply recall their definition. Their free evolution will be discussed in Section~\ref{sec:caloric}.

	To introduce homogeneous Besov spaces, we follow the approach presented in~\cite{Sawano}. Recall $\Soo$ and $\Soo'$ from the notation section. Denote by $C_0$ the annulus $\{ \xi \in \R^n \colon \nicefrac{1}{2} \leq |\xi| \leq 4 \}$ and write $2^j C_0$ for the annulus $\{ \xi \in \R^n \colon 2^{j-1} \leq |\xi| \leq 2^{j+2} \}$ for any integer $j$. Let $\chi \in \Cont^\infty_\cc(\R^n)$ such that $\supp(\chi) \subseteq C_0$ and for any $\xi \neq 0$ there holds the identity $$\sum_{j = -\infty}^\infty \chi(2^{-j} \xi) = 1.$$
	Let $\Delta_j$ be the $j$th \emph{Littlewood--Paley operator} associated with $\chi$, given by
	\begin{align}
		\Delta_j f \coloneqq \cF^{-1}(\chi(2^{-j} \cdot) \cF(f)) \qquad (f \in \Soo').
	\end{align}
	Suppose\footnote{One can also treat the case $\alpha \in \R$ and $p \in (1,\infty]$, see~\cite{Sawano}. We restrict ourselves to the cases relevant for this article to avoid unnecessary technical difficulties.} that $\alpha < 0$ and $p \in (1,\infty]$. Define the homogeneous Besov space $\dot\B^\alpha_{p,p}$ as the collection of all $f \in \Soo'$ such that the norm
	\begin{align}
		\| f \|_{\dot\B^\alpha_{p,p}} \coloneqq \| (\Delta_j f)_j \|_{\ell^p(\IZ; \L^p(\R^n))} < \infty.
	\end{align}
	The space $\dot\B^\alpha_{p,p}$ is a Banach space~\cite[Cor.~2.3, p.~282]{Sawano}.
	For $f \in \dot\B^\alpha_{p,p}$, the series $\sum_j \Delta_j f$ converges to $f$ in $\Soo'$.
	By an observation of Peetre~\cite{Peetre}, owing to $\alpha < 0$, the series $\sum_j \Delta_j f$ even converges to some $g \in \S'$. The \emph{filter} operation\footnote{Here, we follow the terminology of~\cite[p.~286]{Sawano}.}
	\begin{align}
		\Phi \colon \dot\B^\alpha_{p,p} \to \S', \quad f \mapsto \Phi(f) \coloneqq g
	\end{align}
	is well-defined and injective. By construction, $f$ is the equivalence class of $\Phi(f)$ in $\Soo'$. From now on, we will identify $f$ and $\Phi(f)$. In other words, we consider an element $f \in \dot\B^\alpha_{p,p}$, which is a priori a tempered distribution modulo all polynomials, as the tempered distribution $\Phi(f)$.

	{
	The following definition is in the spirit of Definition~\ref{def:Zv-spaces}.

	\begin{definition}[$\vv\dot\B$-spaces]
		\label{def:vvB}
		Let $\alpha < 0$. Define the closed subspace $\vv\dot\B^\alpha_{\infty,\infty}$ of $\dot\B^\alpha_{\infty,\infty}$ as the closure of $\Soo$ in $\dot\B^\alpha_{\infty,\infty}$ equipped with the subspace topology.
	\end{definition}

	\begin{remark}
		\label{rem:def_vvB}
		By definition, $\Soo$ is dense in $\vv\dot\B^\alpha_{\infty,\infty}$.
		If $p < \infty$, then $\Soo$ is dense in $\dot\B^\alpha_{p,p}$, see~\cite[Thm.~2.30, p.~282]{Sawano}.
	\end{remark}
	}

	\subsection{Weak solutions I}
	\label{subsec:prelim_weak_solutionI}

	We introduce the notion of weak solutions to linear parabolic equations and collect some properties of them. %
	Later, after a short interlude on elliptic operators in Section~\ref{subsec:prelim_elliptic}, we will come back to weak solutions in Section~\ref{subsec:prelim_weak_solutionII} and recite some well-posedness results from the literature for them.

	Consider a measurable, real coefficient function $A \colon \R^n \to \R^{n\times n}$ which satisfies for some $\Lambda{(A)}, \lambda{(A)} > 0$ the ellipticity and boundedness properties
	\begin{align}
		\label{eq:A_elliptic}
		\tag{A}
		\Re A(x)\xi \cdot \overline{\xi} \geq \lambda{(A)} |\xi|^2, \quad |A(x)\xi| \leq \Lambda{(A)} |\xi|, \qquad x \in \R^n, \xi \in \R^n.
	\end{align}
	We define weak solutions to a linear parabolic equation in divergence form associated with the coefficient function $A$ as follows.

	\begin{definition}[Weak solution]
		\label{def:weak_solution}
		Let $0 \leq a < b \leq \infty$ and put $Q \coloneqq (a,b) \times \R^n$. Let $f$ and $F$ be in $\cD'(Q)$. A function $u \in \L^2_\loc(a,b; \W^{1,2}_\loc(\R^n))$ is called a \emph{weak solution} to the equation
		\begin{align}
			\label{eq:def_weak}
			\tag{PDE}
			\partial_t u -\Div(A\nabla u) = f + \Div(F)
		\end{align}
		on $Q$,
		if for any $\phi \in \Cont^\infty_\cc(Q)$ there holds the integral identity
		\begin{align}
			\label{eq:def_weak_integral}
			\iint_Q -u \partial_t \phi + A\nabla u \cdot \nabla \phi \d x \d t = \langle f, \phi \rangle + \langle F, \nabla \phi \rangle.
		\end{align}
		The pairings on the right-hand side are the usual duality pairings between a space and its (topological) dual space.
	\end{definition}

	Since $\phi \in \Cont^\infty_\cc(Q)$, it suffices to have $u, \nabla u \in \L^1_\loc(Q)$ instead of $u, \nabla u \in \L^2_\loc(Q)$ to make the integral identity~\eqref{eq:def_weak_integral} meaningful.
	We will treat this case in the subsequent definition.
	To employ energy techniques or to investigate questions such as uniqueness of weak solutions, it is more convenient to directly include local square-integrability of $u$ and $\nabla u$ into the definition of a weak solution.

	\begin{definition}[Very weak solution]
		\label{def:very_weak_solution}
		In the situation of Definition~\ref{def:weak_solution}, if a function $u\in \L^1_\loc(a,b; \W^{1,1}_\loc(\R^n))$ satisfies~\eqref{eq:def_weak_integral}, then call $u$ a \emph{very weak solution} of~\eqref{eq:def_weak}.
	\end{definition}

	We introduce a notion of distributional trace. This enables us to formulate initial value problems in a first step, and Cauchy problems in a second step.

	\begin{definition}[Initial condition]
		\label{def:initial_condition}
		Let $0 < b \leq \infty$, $u \in \Cont({(0,b)}; \cD'(\R^n))$, and $u_0 \in \cD'(\R^n)$. Say that $u$ satisfies the \emph{initial condition} $u(0) = u_0$ if $u(t)$ converges in $\cD'(\R^n)$ to $u_0$ as $t\to 0$.
	\end{definition}

	Using local energy estimates and Lions type embeddings, a weak solution $u$ of~\eqref{eq:def_weak} in the sense of Definition~\ref{def:weak_solution} with appropriate source terms is continuous in $\L^2_\loc$. This will be the case under our assumptions. In particular, $u(t)$ is a well-defined distribution for any $0 < t < b$ when using its continuous representative.

	\subsection{Elliptic operators}
	\label{subsec:prelim_elliptic}

	We recall some background on elliptic operators in divergence form with real, measurable, bounded, elliptic coefficients. Such operators already appeared in the weak formulation of linear parabolic equations in Section~\ref{subsec:prelim_weak_solutionI}.
	A reader who is not yet familiar with notions like maximal accretivity, sectoriality and so on can consult, for instance, the monograph~\cite{Haase} in addition.

	Consider a measurable, real coefficient function $A \colon \R^n \to \R^{n\times n}$ which satisfies the ellipticity and boundedness condition~\eqref{eq:A_elliptic} for some $\Lambda{(A)}, \lambda{(A)} > 0$ as before.
	The elliptic operator $L = -\Div(A\nabla \cdot)$ can be properly defined as a maximal accretive operator on $\L^2_x$ using Kato's form method~\cite{Kato}. In a nutshell, this approach goes as follows: define the elliptic and accretive sesquilinear form
	\begin{align}
		a \colon \W^{1,2}_x \times \W^{1,2}_x \to \R, \quad (u,v) \mapsto \int_{\R^n} A\nabla u \cdot \overline{\nabla v} \d x.
	\end{align}
	Then, there exists an unbounded, maximal accretive operator $L$ on $\L^2_x$ with $\dom(L) \subseteq \W^{1,2}_x$ that satisfies the relation
	\begin{align}
		(Lu, v)_2 = a(u,v), \qquad u\in \dom(L), v\in \W^{1,2}_x.
	\end{align}
	In particular, $-L$ is the generator of a contraction semigroup on $\L^2_x$. Moreover, $L$ is a sectorial operator of angle $\phi \in [0,\nicefrac{\pi}{2})$ on $\L^2_x$. The properties mentioned so far do not yet rely on the fact that $A$ is real. This is going to change now. Indeed, using that the coefficients of $L$ are real, the semigroup $\e^{-tL}$ for $t>0$ can be represented by an integral kernel satisfying Gaussian bounds, see for instance~\cite[Thm.~6.11]{Ouhabaz} (note that this result applies with $\eps = 0 = \alpha$ in our situation). More precisely, there is a measurable kernel $K \colon (0,\infty) \times \R^n \times \R^n \to \R$ satisfying for some $C,c > 0$ the \emph{Gaussian bounds}
	\begin{align}
		|K(t,x,y)| \leq C t^{-\frac{n}{2}} \e^{-c\frac{|x-y|^2}{t}}
	\end{align}
	for all $t > 0$ and almost every $x,y \in \R^n$, such that for all $t>0$ and almost every $x\in \R^n$ one has the representation
	\begin{align}
		(\e^{-tL} f)(x) = \int_{\R^n} K(t,x,y) f(y) \d y, \qquad f \in \L^1_x \cap \L^2_x.
	\end{align}
	Next, we relate kernel bounds to the notion of \enquote{off-diagonal estimates}. For a thorough background on this notion, the reader can consult~\cite{Memoirs}.
	\begin{definition}[Off-diagonal estimates]
		Let $1\leq p \leq q \leq \infty$. Given a family of operators $\mathcal{T} = \{ T_t \}_{t \in I}$ on $\L^2_x$, where $I \subseteq (0,\infty)$ is an interval, say that $\mathcal{T}$ satisfies $\L^p_x \to \L^q_x$ \emph{off-diagonal estimates} of exponential type if there exist $C,c > 0$ such that,
		for all measurable sets $E,F \subseteq \R^n$ and all $t\in I$, there holds the estimate
		\begin{align}
			\label{eq:def_ODE}
			\| \ind_F T_t (f\ind_E) \|_{\L^q_x} \leq C t^{-\frac{n}{2}\bigl( \frac{1}{p} - \frac{1}{q} \bigr)} \e^{-c\frac{\dist(E,F)^2}{t}} \| f \ind_E \|_{\L^p_x}, \qquad f \in \L^p_x \cap \L^2_x.
		\end{align}
		If $p=q$, simply say that $\mathcal{T}$ satisfies $\L^p_x$ off-diagonal estimates. If~\eqref{eq:def_ODE} is only valid for $E = F = \R^n$, say that $\mathcal{T}$ is $\L^p_x \to \L^q_x$ bounded. In the case $p=q$, simply say that $\mathcal{T}$ is $\L^p_x$-bounded.
	\end{definition}
	Using Gaussian bounds, it readily follows that $\{ \e^{-tL} \}_{t>0}$ satisfies $\L^1_x \to \L^\infty_x$ off-diagonal estimates. Then, it is also well-known that $\{ \e^{-tL} \}_{t>0}$ satisfies $\L^p_x \to \L^q_x$ off-diagonal estimates for any $1 \leq p \leq q \leq \infty$, see~\cite{Memoirs}.

	We will also need the gradient family $\{ \sqrt{t} \nabla \e^{-tL} \}_{t>0}$. {There exists some $q(A) \in (2,\infty]$ such that the gradient family $\{ \sqrt{t} \nabla \e^{-tL} \}_{t>0}$ satisfies $\L^q_x$ off-diagonal estimates whenever $q \in [2, q(A))$, see~\cite[Prop.~2.1 \& Cor.~3.10]{Memoirs}\footnote{In this reference, the number $q(A)$ is called $q_{+}(L)$, where $L = -\Div(A\nabla \cdot)$}.}
	Using standard arguments presented in~\cite{Memoirs} and that the coefficients are real, this fact can be extended to $\L^p_x \to \L^q_x$ off-diagonal estimates for $p,q \in [1,q(A))$ satisfying $p \leq q$.
	{Eventually, using duality, we also obtain $\L^q_x \to \L^r_x$ off-diagonal estimates for $\{ \sqrt{t} \e^{-tL} \Div \}_{t>0}$ if $q,r \in (q(A^*)', \infty]$ satisfying $q \leq r$, still using that $L$ has real coefficients.}

	For our non-linear application to a reaction--diffusion equation without any regularity of the coefficients, the just described setting is enough. In future work, for instance in applications to quasi-linear equations, we will assume some regularity of the coefficients. Thus we include this regular case in order to develop our theory in the necessary generality.
	More precisely, we are interested in the availability of $\L^p_x \to \L^q_x$ off-diagonal estimates for the gradient family with $q \in (2,\infty)$. In fact, we directly turn them into the definition of \enquote{regular coefficients}. Afterwards, we illustrate in Example~\ref{ex:regular_A} how a smoothness condition on $A$ can be used to ensure them.

	\begin{definition}[Regular coefficients]
		\label{def:regular_A}
		Given an elliptic coefficient function $A$, say that the coefficients are \emph{regular} if the associated gradient families $\{ \sqrt{t} \nabla \e^{-t(L+1)} \}_{t>0}$ and $\{ \sqrt{t} \nabla \e^{-t(L^*+1)} \}_{t>0}$ are $\L^q_x$-bounded for all $q\in (1,\infty)$.
	\end{definition}

	The definition is formulated using the shifted operator $L+1$. For the \enquote{homogeneous} part $L$, we conclude the following.

	\begin{lemma}[Gradient family for $L$]
		\label{lem:regular_coefficients}
		If $L$ has regular coefficients, then for any $T \in (0,\infty)$, the family $\{ \sqrt{t} \nabla \e^{-tL} \}_{0<t \leq T}$ is $\L^q_x$-bounded for all $q\in (1,\infty)$, where the bound potentially depends on $T$.
	\end{lemma}

	When we use off-diagonal estimates later on in Section~\ref{subsec:duhamel_cauchy} to study solution operators to linear PDEs, these solution operators will (potentially) have degenerate constants when $T \to \infty$ as a consequence of Lemma~\ref{lem:regular_coefficients}.

	\begin{example}
		\label{ex:regular_A}
		If $A$ is a uniformly continuous, elliptic coefficient function, then $A$ is regular. Indeed, this follows by freezing the coefficients. More precisely, we can combine~\cite{Dong-Kim-TAMS} with~\cite[Thm.~12.13]{Block} in the same way as was done in~\cite[Sec.~3]{B_Lions}.
	\end{example}

	Using again duality (in conjunction with composition), the families $\{ \sqrt{t} \e^{-t(L+1)} \Div \}_{t>0}$ and $\{ t \nabla \e^{-t(L+1)} \Div \}_{t>0}$ , as well as $\{ \sqrt{t} \e^{-tL} \Div \}_{0<t \leq T}$ and $\{ t \nabla\e^{-tL} \Div \}_{0<t \leq T}$ enjoy $\L^q_x$ off-diagonal estimates for any $q\in (1, \infty)$ and $T \in (0,\infty)$.

	\subsection{Weak solutions II}
	\label{subsec:prelim_weak_solutionII}

	We return to the investigation of weak solutions to linear parabolic equations initiated in Section~\ref{subsec:prelim_weak_solutionI}. Using some terminology from Section~\ref{subsec:prelim_elliptic}, we present well-posedness results from the literature for them.

	The following results
	are simplified versions of~\cite[Thm.~9.2]{AH3}. We extract them for the reader's convenience, taking also the shift of notation {discussed} in Remark~\ref{rem:Z_spaces} into account.
	To start with, recall from Section~\ref{subsec:prelim_elliptic} the number $q(A^*) \in (2,\infty]$ associated with the adjoint coefficient function $A^*$, describing regularity properties of the equation.
	Note that in~\cite{AH3}, this number is called $q_{+}(L^*)$.
	Put $Q \coloneqq q(A^*)' \in [1, 2)$. Using $Q$, define for a given $\alpha > -1$ the critical number $p(\alpha,A) \in (0,2)$ by
	\begin{align}
		\label{eq:def_pA}
		\frac{1}{p(\alpha,A)} = 1 + \alpha - \frac{\alpha}{Q}.
	\end{align}
	Note that $p(\cdot,A)$ is decreasing.
	The applicability of the linear theory in~\cite{AH3} stated below requires the condition $p > \tilde p_L(\beta)$, where $\beta>-1$ and the number $\tilde p_L(\beta)$ is defined in that article. For $p\in (1,\infty]$ and $\beta>-1$, this condition is implied by the simpler condition $p\ge p(2\beta+1,A)$. We systematically impose the latter.

	The first result on uniqueness is stated for global solutions, but the proof localizes to $[0,T]$ for any $0 < T < \infty$. %

	\begin{proposition}[Uniqueness of weak solutions]
		\label{prop:linear_uniqueness}
		Fix {$T \in (0,\infty]$}. {Let $\beta > -1$ and let $p \in (1,\infty]$ with $p \geq p(2\beta+1,A)$}. Then, there exists at most one weak solution $u$ of the initial value problem
		\begin{align}
			\partial_t u - \Div(A\nabla u) = 0, \quad u(0) = 0
		\end{align}
		on $[0,T]$ satisfying $\nabla u \in \Z^{p,2}_\beta(T)$.
	\end{proposition}

	The second one provides us with existence of a weak solution for the initial value problem with Besov data. %

	\begin{proposition}[Existence of the initial value problem]
		\label{prop:initial_value_problem_hedong}
		Let $\alpha \in (-1,0)$ and let $p\in (1,\infty]$ {with $p \geq p(\alpha, A)$}. Fix $u_0 \in \dot\B^\alpha_{p,p}$. Then, there exists a weak solution $u$ to the initial value problem
		\begin{align}
			\partial_t u - \Div(A\nabla u) = 0, \quad u(0) = u_0,
		\end{align}
		satisfying $u \in \Z^{p,2}_{\nicefrac{\alpha}{2}}$ and $\nabla u \in \Z^{p,2}_{\nicefrac{\alpha}{2} - \nicefrac{1}{2}}$, with
		\begin{align}
			\| u \|_{\Z^{p,2}_{\frac{\alpha}{2}}} + \| \nabla u \|_{\Z^{p,2}_{\frac{\alpha}{2} - \frac{1}{2}}} \lesssim \| u_0 \|_{\dot\B^\alpha_{p,p}}.
		\end{align}
		The initial value is attained in the sense of Definition~\ref{def:initial_condition}. If in addition $u_0 \in \L^2_x$, then $u(t) = \e^{-tL} u_0$.
	\end{proposition}

	\begin{remark}
		In Proposition~\ref{prop:caloric_extension_Besov}, we will provide additional regularity properties for the solution $u$ of the initial value problem. Moreover, we will discuss the case $u_0 \in \vv\dot\B^\alpha_{\infty,\infty}$, where the $\Z$-spaces can be replaced by the corresponding $\Zv$-spaces.
	\end{remark}

	Existence results were also obtained in~\cite{AH3}. We need a generalized version of them, which we will present in Theorem~\ref{thm:linear_wp} later on.

	\subsection{Scaling behavior}
	\label{subsec:scaling}

	We shortly comment on the scaling behavior of the non-linear equation~\eqref{eq:rd}.

	Let $u$ be a solution to~\eqref{eq:rd}. Put $v_\lambda(t,x) = \lambda^\sigma u(\lambda^2 t, \lambda x)$ for any $\lambda > 0$. Then $v_\lambda$ is again a solution\footnote{We assume additionally a natural homogeneity condition for the non-linearity to perform the calculation.} to~\eqref{eq:rd} if $\sigma = \nicefrac{2}{\rho}$.

	Indeed, simply writing $v$ instead of $v_\lambda$, the left-hand side of~\eqref{eq:rd} formally gives
	\begin{align}
		(\partial_t v - \Div(A \nabla v))(t,x) = \lambda^{2 + \sigma} (\partial_t u - \Div(A \nabla u))(\lambda^2 t,\lambda x),
	\end{align}
	whereas the right-hand side leads to
	\begin{align}
		\phi(v)(t,x) = \lambda^{(1+\rho) \sigma} \phi(u)(\lambda^2 t,\lambda x).
	\end{align}
	To use the fact that $u$ solves~\eqref{eq:rd}, we need to balance the powers of $\lambda$ on both sides of the equation. One has indeed $2 + \sigma = (1+\rho) \sigma$ if and only if $\sigma = \nicefrac{2}{\rho}$. Then, the right-hand sides of the last two displayed equations coincide, and we find that $v = v_\lambda$ is again a solution of~\eqref{eq:rd}. The calculation can be made rigorous by using the notion of weak solutions.

	Taking the limit $t \to 0$, the initial value of $v_\lambda$ is $\lambda^{\nicefrac{2}{\rho}} u_0(\lambda \,\cdot)$. This suggests that $u_0$ should be taken from a space whose norm is invariant under the transformation $u_0 \mapsto \lambda^{\nicefrac{2}{\rho}} u_0(\lambda \,\cdot)$. We will come back to this in Section~\ref{sec:caloric}.

	\section{Singular integral operators on $\Z$-spaces: hypercontractivity}
	\label{subsec:SIO}

	The goal of this section is to prepare the ground for the study of Duhamel operators in Section~\ref{subsec:duhamel_cauchy} with a focus on \emph{hypercontractive properties}, that is to say, obtaining higher integrability of the output compared to the input.
	Such a property is necessary in the treatment of non-linear equations.

	To this end, we study singular integral operators (SIOs) on weighted $\Z$-spaces. If the local integrability parameter of the weighted $\Z$-spaces are fixed to $q = 2$, such results were obtained in~\cite[Sec.~9]{AH3} using interpolation of corresponding results on weighted tent spaces. Our focus here is to allow the parameter $q$ to be different to $2$ on the one hand, and to study \enquote{hypercontractive} mapping properties on the other hand, which is a new aspect.

	The following definition is a hypercontractive extension of the definition introduced in~\cite{AH2} when $q=r=2$. For its formulation, recall $\pi(f)$ and $\ann_j(x,t)$ from the notation section.

	\begin{definition}[SIOs of type $(p_0,q,r,\kappa,M)$]
		\label{def:SIO}
		Let $p_0,q,r \in (1,\infty]$, $\kappa \in [0,1]$ and $M \in [0,\infty]$ satisfying $q \geq p_0$, $r \geq q$ and $\kappa \geq \nicefrac{n}{2}(\nicefrac{1}{q} - \nicefrac{1}{r})$.
		An operator $S$ is called \emph{singular integral operator (SIO)} of \emph{type $(p_0,q,r,\kappa,M)$} with \emph{integral kernel} $K(t,s)$, if
		\begin{enumerate}
			\item[(a)] $S$ is $\L^q \to \L^r_{\kappa - \nicefrac{n}{2}(\nicefrac{1}{q} - \nicefrac{1}{r})}$ bounded.
			\item[(b)] For $0 < s < t < \infty$, the kernel $K(t,s) \colon \L^{p_0}_x \to \L^r_x$ satisfies for $f\in \L^{p_0}_x$ and $j \geq 1$ the (weak) off-diagonal bound
			\begin{align}
				\| \ind_{\B(x,\sqrt{t})} K(t,s) (f \ind_{\ann_j(x,t)}) \|_{\L^r_x} \lesssim (t-s)^{-1 + \kappa - \frac{n}{2}\bigl( \frac{1}{p_0} - \frac{1}{r} \bigr)} \Bigl( 1 + \frac{d^2}{t-s} \Bigr)^{-M} \| \ind_{\ann_j(x,t)} f \|_{\L^{p_0}_x},
			\end{align}
			where $d \coloneqq \dist(\B(x,\sqrt{t}), \ann_j(x,t))$.
			\item[(c)] For $f \in \L^q$ with compact support and almost every $(t,x) \not\in \pi(f)$ there holds the representation formula
			\begin{align}
				S(f)(t,x) = \int_0^t \bigl(K(t,s) f(s) \bigr)(x) \d s.
			\end{align}
		\end{enumerate}
	\end{definition}

	\begin{remark}
		Strictly speaking, the terminology \enquote{singular} is only justified when $\kappa = \nicefrac{n}{2}(\nicefrac{1}{q} - \nicefrac{1}{r})$. Usually, this situation happens when $\kappa = 0$ and $q = r$. Our theory gives optimal estimates in the singular and non-singular case and we refrain from making the terminology unnecessarily complicated.
	\end{remark}

	\begin{remark}[Causality]
		\label{rem:SIO_causal}
		Owing to property (c), if $S$ is an SIO in the sense of Definition~\ref{def:SIO}, then $S$ is \emph{causal}, meaning that the value of $S$ at time $\tau$ only depends on the data in the time interval $[0,\tau]$. Indeed, fix $\tau > 0$ and split $f(s,y) = f(s,y)\ind_{(0,\tau]}(s) + f(s,y)\ind_{(\tau,\infty)}(s)=f_{1}(s,y)+f_{2}(s,y)$. Then, when $t < \tau$, property (c) applies to $S(f_2)$ and yields $S(f_2)(t,x) = 0$. Therefore, causality follows.
	\end{remark}

	The following is the abstract main result on SIOs of hypercontractive type.

	\begin{theorem}[Hypercontractivity of SIOs on $\Z$-spaces]
		\label{thm:sio}
		Let $p_0,q,r \in (1,\infty]$, $\kappa \in [0,1]$ and $M \in [0,\infty]$ satisfying $q \geq p_0$, $r \geq q$, $M > \nicefrac{n}{2p_0}$ and $M > \nicefrac{n}{2}(\nicefrac{1}{p_0} - \nicefrac{1}{r}) + \nicefrac{1}{p_0} - \kappa$. Fix a singular integral operator $S$ of type $(p_0,q,r,\kappa,M)$. Also, let $\beta > -1$ and $T \in (0,\infty]$.
		\begin{enumerate}
			\item[(i)] Let $p \in [p_0, \infty)$. Then, $S$ extends by density to a bounded operator $\Z^{p,q}_\beta(T) \to \Z^{p,r}_{\beta + \kappa}(T)$.
			\item[(ii)] There exists a bounded operator $\widetilde{S} \colon \Z^{\infty,q}_\beta(T) \to \Z^{\infty,r}_{\beta + \kappa}(T)$ satisfying the following properties:
			\begin{itemize}
				\item $\widetilde{S}(f) = S(f)$ almost everywhere for all $f \in \Z^{\infty,q}_\beta(T) \cap \L^q_\cc(T)$,
				\item there is a sequence of functions $f_m \in \L^q_\cc$ such that, for every compact subset $E \subseteq (0,T] \times \R^n$, $f_m$ converges in $\L^q(E)$ to $f$ as $m\to \infty$, there holds $u_m \coloneqq S(f_m) \in \L^r(E)$, $u_m$ converges in $\L^r(E)$ as $m\to \infty$, and its limit coincides with $\widetilde{S}(f)$ almost everywhere on $E$.
			\end{itemize}
		\end{enumerate}
		All operator norms are independent of $T$.
		In particular, the mapping property in (ii) remains true for $\Zv$-spaces.
	\end{theorem}

	\begin{remark}[Global versus local estimates in time]
		\label{rem:SIO_global_estimate}
		If an operator $S$ does not satisfy the hypotheses (a) and/or (b) of an SIO for $T = \infty$ but only for $T \in (0,\infty)$ with bounds depending on $T$, then the conclusion of Theorem~\ref{thm:sio} holds true for $T$ finite with operator norm depending on $T$ (through the norm of $S$ in condition (a) of Definition \ref{def:SIO} and/or the implicit constant in (b) for $0 < s < t \leq T$). Indeed, this follows by an inspection of the proof, or by applying Theorem~\ref{thm:sio} to a time-truncated version of $S$.
	\end{remark}

	For its proof we will need the following lemma, which gives additional decay for a restricted (non-singular) fractional integral. We postpone its proof to the end of this subsection. The unrestricted fractional integral will be treated in Lemma~\ref{lem:fractional_integral}.

	\begin{lemma}
		\label{lem:fractional_integral_decay}
		Let $\lambda \in \R$, $q,r \in (1,\infty]$ and $\gamma_0, \gamma_1 \in \R$ satisfying $r \geq q$ and
		\begin{align}
			\label{eq:fractional_integral_relation}
			\gamma_1 = \gamma_0 + \lambda.
		\end{align}
		For $k\geq 1$ consider the integral operator
		\begin{align}
			T_\lambda^k(f)(t) \coloneqq \int_{2^{-k-1} t}^{2^{-k} t} (t-s)^{-1+\lambda} f(s) \d s.
		\end{align}
		Then $T_\lambda^k \colon \L^q_{\gamma_0} \to \L^r_{\gamma_1}$ with $\| T_\lambda^k \|_{\cL(\L^q_{\gamma_0}, \L^r_{\gamma_1})} \lesssim 2^{-k(\gamma_0 + 1)}$.
	\end{lemma}

	Now, we can proceed with the proof of Theorem~\ref{thm:sio}.

	\begin{proof}[Proof of Theorem~\ref{thm:sio}]
		We are going to distinguish the cases $p < \infty$ and $p = \infty$. Due to the hypercontractive character of the result, we cannot obtain the case $p < \infty$ by interpolation with the case $p = \infty$ in a meaningful way. Therefore, we have to treat both cases by hand.

		To ease notation, summation over $j$ and $\ell$ refers to summation over $j \geq 1$ and $\ell \geq 1$ throughout the proof.

		\textbf{Case 1}: $p < \infty$. Fix $T \in (0,\infty]$. By the dense inclusion $\L^q_\cc(T) \subseteq \Z^{p,q}_\beta(T)$ (see Section~\ref{subsec:prelim_function_spaces}), it suffices to derive an estimate for $f\in \L^q_\cc(T)$. Let $x\in \R^n$ and $0 < t \leq T$. We start with a bound for $S(f)$ on the parabolic Whitney box $\WB(t,x) = (\nicefrac{t}{2},t) \times \B(x,\sqrt{t})$. To this end, decompose $f = \sum_{j \geq 1} \sum_{\ell \geq 1} f_{j,\ell}$ on $(0,t)\times \R^n$, where
		\begin{align}
			\label{eq:def_f_jl}
			f_{j,\ell}(s,z) \coloneqq
			\begin{cases}
				f(s,z) \ind_{(\nicefrac{t}{4}, t)}(s) \ind_{\ann_j(x,t)}(z), \qquad &j\geq 1, \ell = 1, \\
				f(s,z) \ind_{(2^{-\ell-1}t, 2^{-\ell}t)}(s) \ind_{\ann_j(x,t)}(z), \qquad &j \geq1, \ell \geq 2.
			\end{cases}
		\end{align}
		Note that the decomposition is adapted to the parabolic Whitney box $\WB(t,x)$.
		Also, observe that the sum in the decomposition of $f$ is finite owing to the compact support of $f$. Therefore, $S(f)(\tau) = \sum_{j,\ell} S(f_{j,\ell})(\tau)$ for every $0 \leq \tau \leq t$, where we also use causality of $S$ (see Remark~\ref{rem:SIO_causal}).
		To show the claim, it suffices to prove
		\begin{align}
			\label{eq:SIO_thm_key_estimate}
			\| S(f_{j,\ell}) \|_{\Z^{p,r}_{\beta + \kappa}(T)} \lesssim 2^{-\ell a} 2^{-j b} \| f \|_{\Z^{p,q}_{\beta}(T)}
		\end{align}
		for some numbers $a,b > 0$. For convenience, we focus on the case $r < \infty$. The case $r=\infty$ follows with the usual modifications.

		We start out with the \emph{local part}, that is, the term where $j = 1 = \ell$. Using $t \approx \tau$, property (a) of an SIO, and $\supp(f_{1,1}) \subseteq [\nicefrac{t}{4}, t] \times \overline{\B}(x,\sqrt{4t})$, we find
		\begin{align}
			&\Bigl( \fint_{\nicefrac{t}{2}}^t \fint_{\B(x,\sqrt{t})} |\tau^{-(\beta + \kappa)} S(f_{1,1})(\tau,y)|^r \d y \d \tau \Bigr)^\frac{1}{r} \\
			\approx{} &t^{-(\beta + \nicefrac{n}{2q})} \Bigl( \int_{\nicefrac{t}{2}}^t \int_{\B(x,\sqrt{t})} |\tau^{-(\kappa - \nicefrac{n}{2}(\nicefrac{1}{q} - \nicefrac{1}{r}))} S(f_{1,1})(\tau,y)|^r \d y \ddtau \Bigr)^\frac{1}{r} \\
			\lesssim{} &t^{-(\beta + \nicefrac{n}{2q})} \Bigl( \int_{\nicefrac{t}{4}}^t \int_{\B(x,\sqrt{4t})} |f(s,z)|^q \d z \dds \Bigr)^\frac{1}{q}
			\approx \Bigl( \fint_{\nicefrac{t}{4}}^t \fint_{\B(x,\sqrt{4t})} |s^{-\beta} f(s,z)|^q \d z \d s \Bigr)^\frac{1}{q}.
		\end{align}
		Applying the $\L^p_{x}$-norm to this bound gives
		$$\| S(f_{1,1}) \|_{\Z^{p,r}_{\beta + \kappa}(T)} \lesssim \| f \|_{\Z^{p,q}_\beta(T)},$$
		and hence concludes the treatment of the local part.

		We proceed with the \emph{global part}. First, we consider the cases $j = 1$ and $\ell \geq 2$. Since the Whitney box $\WB(t,x)$ is disjoint to $\pi(f_{1,\ell})$, we can use property (c), $t \approx \tau$, Minkowski's inequality and (b) to give
		\begin{align}
			&\Bigl( \fint_{\nicefrac{t}{2}}^t \fint_{\B(x,\sqrt{t})} |\tau^{-(\beta + \kappa)} S(f_{1,\ell})(\tau,y)|^r \d y \d \tau \Bigr)^\frac{1}{r} \\
			\lesssim{} &\Bigl( \int_{\nicefrac{t}{2}}^t \Bigl( \tau^{-(\beta + \kappa + \nicefrac{n}{2r})} \int_{2^{-\ell-1} t}^{2^{-\ell} t} \| K(\tau,s)f_{1,\ell}(s) \|_{\L^r_x} \d s \Bigr)^r \ddtau \Bigr)^\frac{1}{r} \\
			\lesssim &\Bigl( \int_{0}^\infty \Bigl( \tau^{-(\beta + \kappa + \nicefrac{n}{2r})} \int_{2^{-\ell-1} \tau}^{2^{-\ell+1} \tau} (\tau-s)^{-1 + \kappa -\nicefrac{n}{2}(\nicefrac{1}{p_0} - \nicefrac{1}{r})} \| f_{1,\ell}(s) \|_{\L^{p_0}_x} \d s \Bigr)^r \ddtau \Bigr)^\frac{1}{r}.
		\end{align}
		Apply Lemma~\ref{lem:fractional_integral_decay} with $k \coloneqq \ell - 1 \geq 1$, $\lambda \coloneqq \kappa - \nicefrac{n}{2}(\nicefrac{1}{p_0} - \nicefrac{1}{r})$, $\gamma_0 \coloneqq \beta + \nicefrac{n}{2p_0}$, $\gamma_1 \coloneqq \beta + \kappa + \nicefrac{n}{2r}$ to the last expression and recall $s \approx 2^{-\ell} t$ on the support of $f_{j,\ell}$, to find
		\begin{align}
			&\Bigl( \fint_{\nicefrac{t}{2}}^t \fint_{\B(x,\sqrt{t})} |\tau^{-(\beta + \kappa)} S(f_{1,\ell})(\tau,y)|^r \d y \d \tau \Bigr)^\frac{1}{r} \\
			\lesssim{} &2^{-\ell(\beta + \nicefrac{n}{2p_0} + 1)} \Bigl( \int_0^\infty \int_{\R^n} |s^{-(\beta + \nicefrac{n}{2p_0})} f_{1,\ell}(s,z)|^{p_0} \d z \dds \Bigr)^\frac{1}{p_0} \\
			\approx{} &2^{-\ell(\beta + \nicefrac{n}{2p_0} + 1)} \Bigl( \fint_{2^{-\ell-1} t}^{2^{-\ell} t} (2^{-\ell} t)^{-\frac{n}{2}} \int_{\B(x, \sqrt{4t})} |s^{-\beta} f(s,z)|^{p_0} \d z \d s \Bigr)^\frac{1}{p_0}.
		\end{align}
		Now, apply the $\L^p_{x}$-norm to the last bound, use the change of angle formula (Lemma~\ref{lem:change_of_angle}), which yields the growth factor $2^{\nicefrac{\ell n}{2p_0}}$ as $p \geq p_0$, and use the nesting property of weighted $\Z$-spaces (due to $q \geq p_0$) to give
		\begin{align}
			\| S(f_{1,\ell}) \|_{\Z^{p,r}_{\beta + \kappa}(T)} \lesssim 2^{-\ell(\beta + 1)} \| f \|_{\Z^{p,p_0}_\beta(T)} \leq 2^{-\ell(\beta + 1)} \| f \|_{\Z^{p,q}_\beta(T)}.
		\end{align}
		This establishes~\eqref{eq:SIO_thm_key_estimate} with $a \coloneqq \beta + 1$ in the current case as $a>0$ by the assumption $\beta > -1$.

		Next, consider the terms $j \geq 2$ and $\ell = 1$. Now, the spatial support of $f_{j,1}$ is disjoint to the Whitney box $\WB(t,x)$, so that we have (c) at our disposal. Thus, we can argue similarly as before to obtain
		\begin{align}
			&\Bigl( \fint_{\nicefrac{t}{2}}^t \fint_{\B(x,\sqrt{t})} |\tau^{-(\beta + \kappa)} S(f_{j,1})(\tau,y)|^r \d y \d \tau \Bigr)^\frac{1}{r} \\
			\lesssim{} &\Bigl( \int_{\nicefrac{t}{2}}^t \Bigl( \tau^{-(\beta + \kappa + \nicefrac{n}{2r})} \int_{\nicefrac{\tau}{4}}^\tau (\tau-s)^{-1 + \kappa - \nicefrac{n}{2}(\nicefrac{1}{p_0} - \nicefrac{1}{r})} \Bigl(\frac{2^j t}{\tau-s} \Bigr)^{-M} \| f_{j,1}(s) \|_{\L^{p_0}(\ann_j(x,t))} \d s \Bigr)^r \ddtau \Bigr)^\frac{1}{r}.
		\end{align}
		Owing to the term $(\tau-s)^M$ from the off-diagonal decay, a simple calculation using Hölder's inequality\footnote{The second constraint on $M$ in the theorem is used here to ensure integrability of the appearing integrals.} yields
		\begin{align}
			&\Bigl( \fint_{\nicefrac{t}{2}}^t \fint_{\B(x,\sqrt{t})} |\tau^{-(\beta + \kappa)} S(f_{j,1})(\tau,y)|^r \d y \d \tau \Bigr)^\frac{1}{r} \\
			\lesssim{} &2^{-jM} \Bigl( \int_{\nicefrac{t}{4}}^t \int_{\B(x, \sqrt{2^{j+1} t})} |s^{-(\beta + \nicefrac{n}{2p_0})} f(s,z)|^{p_0} \d z \dds \Bigr)^\frac{1}{p_0} \\
			\approx{} &2^{-jM} \Bigl( \fint_{\nicefrac{t}{4}}^t t^{-\frac{n}{2}} \int_{\B(x, \sqrt{2^{j+1} t})} |s^{-\beta} f(s,z)|^{p_0} \d z \d s \Bigr)^\frac{1}{p_0}.
		\end{align}
		Apply again the $\L^p_{x}$-norm and use a change of angle followed by the nesting property of weighted $\Z$-spaces, to give
		$$\| S(f_{j,1}) \|_{\Z^{p,r}_{\beta + \kappa}(T)} \lesssim 2^{-jM} 2^{\frac{jn}{2p_0}} \| f \|_{\Z^{p,p_0}_\beta(T)} \leq 2^{-jM} 2^{\frac{jn}{2p_0}} \| f \|_{\Z^{p,q}_\beta(T)}.$$
		Put $b \coloneqq M - \nicefrac{n}{2p_0}$ for~\eqref{eq:SIO_thm_key_estimate} and note $b > 0$ by assumption on $M$.

		Eventually, we consider the cases $j \geq 2$ and $\ell \geq 2$. Analogous arguments as before (using again Lemma~\ref{lem:fractional_integral_decay}) provide
		\begin{align}
			&\Bigl( \fint_{\nicefrac{t}{2}}^t \fint_{\B(x,\sqrt{t})} |\tau^{-(\beta + \kappa)} S(f_{j,\ell})(\tau,y)|^r \d y \d \tau \Bigr)^\frac{1}{r} \\
			\lesssim{} &2^{-jM} 2^{-\ell(\beta + \nicefrac{n}{2p_0} + 1)} \Bigl( \fint_{2^{-\ell-1} t}^{2^{-\ell}t} (2^{-\ell} t)^{-\frac{n}{2}} \int_{\B(x, \sqrt{2^{j+1} t})} |s^{-\beta} f|^{p_0} \d z \d s \Bigr)^\frac{1}{p_0}
		\end{align}
		Hence, applying the $\L^p$-norm and using a change of angle joint with the nesting property of weighted $\Z$-spaces give
		$$\| S(f_{j,\ell}) \|_{\Z^{p,r}_{\beta + \kappa}(T)} \lesssim 2^{-\ell a} 2^{-jb} \| f \|_{\Z^{p,p_0}_\beta(T)} \leq 2^{-\ell a} 2^{-jb} \| f \|_{\Z^{p,q}_\beta(T)}$$
		with the same choices of $a$ and $b$ as before. This concludes Case~1.

		\textbf{Case 2}: $p = \infty$. We cannot argue by density anymore. Instead, we give an explicit construction of the extension $\tilde S$ of $S$ to $\Z^{\infty,q}_\beta(T)$. The argument is an adaptation of~\cite[Thm.~3.1]{Auscher_Portal}.

		Fix $T \in (0,\infty]$ and let $f \in \Z^{\infty,q}_\beta(T)$. Let $x \in \R^n$ and $0 < t \leq T.$ As in Case~1, decompose $f = \sum_{j,\ell} f_{j,\ell}$ on $(0,t) \times \R^n$, where the functions $f_{j,\ell}$ are defined as in~\eqref{eq:def_f_jl}. Note that $f_{j,\ell} \in \L^q_\cc(t)$. Repeating the arguments from Case~1, we obtain the bound
		\begin{align}
			\sum_{j, \ell} \Bigl( \fint_{\nicefrac{t}{2}}^t \fint_{\B(x,\sqrt{t})} |\tau^{-(\beta + \kappa)} S(f_{j,\ell})(\tau,y)|^r \d y \d \tau \Bigr)^\frac{1}{r} \lesssim \| f \|_{\Z^{\infty,q}_\beta(T)}.
		\end{align}
		In fact, for a dilation parameter $a \geq 1$, we can even deduce the bound
		\begin{align}
			\label{eq:convergens_tilde_S}
			\sum_{j, \ell} \Bigl( \fint_{t/(2a)}^t \fint_{\B(x,\sqrt{a t})} |\tau^{-(\beta + \kappa)} S(f_{j,\ell})(\tau,y)|^r \d y \d \tau \Bigr)^\frac{1}{r} \lesssim_a \| f \|_{\Z^{\infty,q}_\beta(T)}.
		\end{align}
		Indeed, without loss of generality, it suffices to treat the case $a = 2^k$. Then, the terms with $\ell \leq 1 + k$ and $j \leq 1 + k$ all contribute to the local part. Besides that, the calculation works in the same way as before.

		Now, let $E \subseteq (0,t] \times \R^n$ be compact. For $a \geq 1$ big enough, there holds $E\subseteq [t/(2a),t] \times \overline{\B}(x,\sqrt{at})$. Therefore,~\eqref{eq:convergens_tilde_S} yields in particular the absolute convergence of $\sum_{j,\ell} S(f_{j,\ell})$ in $\L^r(E)$. After passing to a subsequence, the convergence also holds pointwise almost everywhere on $E$. Recall that the functions $f_{j,\ell}$ are adapted to the Whitney box $\WB(t,x)$. However, we claim that the (pointwise almost everywhere) limit of $\sum_{j,\ell} S(f_{j,\ell})$ on $E$ is independent of the starting Whitney box $\WB(t,x)$. Taking this fact for granted, we can define $\widetilde{S}(f)(s,y)$ almost everywhere on $(0,t] \times \R^n$ by $\sum_{j,\ell} S(f_{j,\ell})$ for $f_{j,\ell}$ defined with respect to \emph{any} starting Whitney box of height $t' \in [t,T]$. In particular, we then obtain the $\Z^{\infty,q}_\beta \to \Z^{\infty,r}_{\beta + \kappa}$ boundedness of $\widetilde{S}$ from~\eqref{eq:convergens_tilde_S} with $a = 1$. That $\tilde S$ coincides with $S$ on $\L^q_\cc(T)$ is clear by definition. Finally, starting with $t = T$ and defining $f_m$ as a partial sum of $\sum_{j,\ell} f_{j,\ell}$ (using the usual diagonal argument), we can represent $\tilde S(f)$ as the limit of $S(f_m)$ in $\L^r(E)$ as $m\to \infty$, where $f_m \in \L^q_\cc(T)$.

		Let us now show independence of the limit. Put $W \coloneqq \WB(t,x)$ and let $W' \coloneqq \WB(\tau,y)$ be another Whitney box with $y\in \R^n$ and $\tau \in [t,T]$. Set $u_{j,W} \coloneqq \sum_\ell S(f_{j,\ell})$, where the functions $f_{j,\ell}$ are the ones constructed in~\eqref{eq:def_f_jl} with respect to the Whitney box $W$. Moreover, let $u_{j,W'}$ be the corresponding sequence where the $f_{j,\ell}$ are constructed with respect to the starting Whitney box $W'$. Let $z \in \R^n$ and $0 < s \leq t$. We show that $\sum_j u_{j,W}$ and $\sum_j u_{j,W'}$ have the same limit on $E \coloneqq \WB(s,z)$. By arbitrariness of $(z,s)$, this shows independence of the limit on $(0,t] \times \R^n$. To do so, write $u_W = U_{j,W} + R_{j,W}$ for $j \geq 1$, where $R_{j,W}$ is the tail of $\sum_j u_{j,W}$ starting at the index $j+1$. Similarly, decompose $u_{W'} = U_{j,W'} + R_{j,W'}$. Pick $k \geq 1$ big enough such that $\B(z,\sqrt{s}) \subseteq \B(x, \sqrt{2^k t}) \cap \B(y, \sqrt{2^k \tau})$. Write $$u_W - u_{W'} = (U_{j+k,W} - U_{j+k,W'}) + R_{j+k,W} - R_{j+k,W'}.$$ Both tails converge to zero in $\L^r(\WB(s,z))$ as $j\to \infty$ owing to~\eqref{eq:convergens_tilde_S} with $a = 2^k$. To treat $U_{j+k,W} - U_{j+k,W'}$, we claim that $U_{j+k,W} - U_{j+k,W'} = S(g_j)$ on $(0,s] \times \R^n$, where $$g_j = f(\ind_{\B(x,\sqrt{2^{j+k+1}t})} - \ind_{\B(y,\sqrt{2^{j+k+1}\tau})})$$ on $(0,s] \times \R^n$. Indeed, this is a consequence of linearity of $S$ on $\L^q_\beta$, which follows from Case~1 applied with $p = q$. Clearly, $g_j$ is uniformly bounded in $\Z^{\infty,q}_\beta(s)$, and an elementary geometric inspection reveals that $g_j$ vanishes on $(0,s) \times \B(z, \sqrt{2^{j+1} s})$. Hence, expanding $S(g_j)$ around the Whitney box $\WB(s,z)$ and employing the tail estimate another time yields convergence of $U_{j+k,W} - U_{j+k,W'}$ in $\L^r(\WB(s,z))$ to zero as $j \to \infty$. We tacitly use $g_j \in \L^q_\beta$ (Lemma~\ref{lem:Z_local_Lq}) and continuity of $S$ in this argument.

		Finally, we comment on the boundedness on $\Zv$-spaces. Write $C$ for the operator norm of $\tilde S$, which is independent of the time interval. Let $f \in \Zv^{\infty,q}_\beta(T)$ and $\eps > 0$. By definition, there is $t \in (0,T]$ such that $\| f \|_{\Z^{\infty,q}_\beta(t)} \leq \nicefrac{\eps}{C}$. Thus, $\| S(f) \|_{\Z^{\infty,r}_{\beta + \kappa}(t)} \leq C \| f \|_{\Z^{\infty,q}_\beta(t)} \leq \eps$. Hence, $S(f) \in \Zv^{\infty,r}_{\beta + \kappa}(T)$ as desired.
	\end{proof}

	\begin{remark}
		\label{rem:SIO_thm}
		If an operator $S$ is an SIO of type $(p_0,q,r,\kappa,M)$ and of type $(\tilde p_0,\tilde q,\tilde r,\kappa,\tilde M)$ at the same time, the extensions into both settings are consistent. In particular, if $\tilde q = \tilde r = 2$, then our extensions coincide with those obtained in~\cite{AH3}.
	\end{remark}

	Combining Theorem~\ref{thm:sio} with the weighted $\Z$-space embedding (Proposition~\ref{prop:Z_space_embedding} or Corollary~\ref{cor:embedding_into_Zv}), we readily find the following strengthening of property (a) of an SIO.

	\begin{corollary}[Hypercontractivity of SIOs on weighted Lebesgue spaces]
		\label{cor:SIO_weighted_Lq}
		Let $S$ be an SIO of type $(q,q,r,\kappa,M)$ with $q,r \in (1,\infty]$, $\kappa \in [0,1]$ and $M \in [0,\infty]$ satisfying $r \geq q$, $M > \nicefrac{n}{2q}$ and $M > \nicefrac{n}{2}(\nicefrac{1}{q} - \nicefrac{1}{r}) + \nicefrac{1}{q} - \kappa$.
		Also, let $\beta > -1$ and $T \in (0,\infty]$. Then, $S$ extends to a bounded operator $S \colon \L^q_\beta(T) \to \L^r_{\beta + \kappa - \nicefrac{n}{2}(\nicefrac{1}{q} - \nicefrac{1}{r})}(T)$.
		If one of $q$, $r$ is infinite, the corresponding $\L^\infty$-space can be replaced by an $\Lv^\infty$-space
	\end{corollary}

	We conclude this subsection by the postponed proof of Lemma~\ref{lem:fractional_integral_decay}.

	\begin{proof}[Proof of Lemma~\ref{lem:fractional_integral_decay}]
		Without loss of generality, assume that $f$ is positive. Using that $t \approx t-s$, calculate
		\begin{align}
			T_\lambda^k(f)(t) \approx 2^{-k(\gamma_0 + 1)} t^{\gamma_0 + \lambda} \fint_{2^{-k-1}t}^{2^{-k}t} s^{-\gamma_0} f(s) \d s.
		\end{align}
		Now we distinguish cases.

		\textbf{Case 1}: $r = \infty$. Using~\eqref{eq:fractional_integral_relation} and Jensen's inequality, we simply obtain
		\begin{align}
			\sup_{t>0} t^{-\gamma_1} T_\lambda^k(f)(t) \lesssim 2^{-k(\gamma_0 + 1)} \sup_{t > 0} \fint_{2^{-k-1}t}^{2^{-k}t} s^{-\gamma_0} f(s) \d s \lesssim 2^{-k(\gamma_0 + 1)} \| f \|_{\L^q_{\gamma_0}}.
		\end{align}

		\textbf{Case 2}: $r = q$. Similarly to the previous case, we first deduce
		\begin{align}
			\Bigl( \int_0^\infty \bigl( t^{-\gamma_1} T_\lambda^k(f)(t) \bigr)^q \ddt \Bigr)^\frac{1}{q} &\lesssim 2^{-k(\gamma_0 + 1)} \Bigl( \int_0^\infty \Bigl( \fint_{2^{-k-1}t}^{2^{-k}t} s^{-\gamma_0} f(s) \d s \Bigr)^q \ddt \Bigr)^\frac{1}{q} \\
			&\lesssim 2^{-k(\gamma_0 + 1)} \Bigl( \int_0^\infty \fint_{2^{-k-1}t}^{2^{-k}t} \bigl( s^{-\gamma_0} f(s) \bigr)^q \d s \ddt \Bigr)^\frac{1}{q}.
		\end{align}
		Hence, Fubini's theorem reveals
		\begin{align}
			\| T_\lambda^k(f) \|_{\L^q_{\gamma_1}} \lesssim 2^{-k(\gamma_0 + 1)} \| f \|_{\L^q_{\gamma_1}}.
		\end{align}

		\textbf{Case 3}: $r \in (q,\infty)$. This case follows by interpolation from the previous two cases.
	\end{proof}

	\section{Consequences for inhomogeneous linear Cauchy problems}
	\label{subsec:duhamel_cauchy}

	Using the results from the preceding section, we construct weak solutions to the linear problem
	\begin{align}
		\label{eq:LP}
		\tag{LP}
		\partial_t u - \Div(A\nabla u) = f + \Div(F), \quad u(0) = 0,
	\end{align}
	and show that they possess higher integrability than the source terms.
	For brevity, put $L = -\Div(A\nabla)$. We have reviewed the relevant properties of such elliptic operators in Section~\ref{subsec:prelim_elliptic} and we are going to use them freely.

	Given a scalar function $f$ and a vector field $F$, define solution operators (formally) by
	\begin{align}
		\Soll^L_1(f)(t) &= \int_0^t \e^{-(t-s)L} f(s) \d s, \\
		\Sol^L_\frac{1}{2}(F)(t) &= \int_0^t \e^{-(t-s)L} \Div(F(s)) \d s.
	\end{align}
	For $f,F \in \L^2(\R^{1+n}_+)$, the integrals converge absolutely in $\L^2(\R^n)$ and yield linear maps $\L^2(\R^{1+n}_+) \to \L^2_\loc(\R^{1+n}_+)$. If in addition $f,F$ are compactly supported, it is straight-forward to show that $u \coloneqq \Soll^L_1(f) + \Sol^L_{\nicefrac{1}{2}}(F)$ provides a weak solution to~\eqref{eq:LP}. The roadmap for this section is as follows:
	\begin{itemize}
		\item Show that $\Soll^L_1$ and $\Sol^L_{\nicefrac{1}{2}}$, as well as their gradients, are SIOs in the sense of Definition~\ref{def:SIO}. In particular, this includes establishing boundedness properties in the scale of weighted Lebesgue spaces for them (Corollary~\ref{cor:Duhamel_Lq}).
		\item Extend $\Soll^L_1$ and $\Sol^L_{\nicefrac{1}{2}}$ using the SIO-theory from Section~\ref{subsec:SIO} to hypercontractive operators in the scale of weighted $\Z$-spaces (Proposition~\ref{prop:Duhamel_mapping_Z}). We impose the conditions from Table~\ref{tab:Duhamel_Sobolev} for this.
		\item Show (by virtue of an approximation argument) that these extensions furnish (very) weak solutions to~\eqref{eq:LP} with forcing terms in weighted $\Z$-spaces (Proposition~\ref{prop:Duhamel_Z_weak}).
		\item Establish the vanishing trace condition in case that the solutions are locally square-integrable (Proposition~\ref{prop:Duhamel_Z_trace}).
	\end{itemize}

	We start out with the first bullet point from the roadmap. To check the mapping properties in Lebesgue spaces, we often use corresponding mapping properties for the scalar fractional integral. A proof for its mapping behavior is given in~\cite[Thm.~3.7, p.~72]{Samko}. We recall that result in our notation for the reader's convenience.

	\begin{lemma}[Fractional integral]
		\label{lem:fractional_integral}
		Let $\lambda \in (0,1)$, $q,r \in (1,\infty]$ satisfying $r \geq q$ and $\gamma_0, \gamma_1 \in \R$. Consider the fractional integral of order $\lambda$ given by
		\begin{align}
			T_\lambda(f)(t) \coloneqq \int_0^t (t-s)^{-1+\lambda} f(s) \d s, \qquad t>0.
		\end{align}
		If the parameters satisfy
		\begin{align}
			\gamma_1 = \gamma_0 + \lambda, \quad \lambda \geq \frac{1}{q} - \frac{1}{r}, \quad \gamma_0 > -1,
		\end{align}
		then $T_\lambda \colon \L^q_{\gamma_0} \to \L^r_{\gamma_1}$.
	\end{lemma}

	Now, we can continue with the SIO property of the Duhamel operators.

	\begin{table}
		\centering
		\begin{tabular}{ c | c |c c }
			& & $A$ rough & $A$ regular \\
			\hline
			$\Soll^L_1$ & $\kappa = 1$ & $r \in \SobSet^{**}(q)$ & $r \in \SobSet^{**}(q)$ \\
			$\nabla \Soll^L_1$ & $\kappa = \nicefrac{1}{2}$ & $r \in \SobSet^{*}(q)$, $r < q(A)$ & $r \in \SobSet^{*}(q), {r<\infty}$ \\
			$\Sol^L_{\nicefrac{1}{2}}$ & $\kappa = \nicefrac{1}{2}$ & $r \in \SobSet^*(q)$, $\min(p,q) > q(A^*)'$ & $r \in \SobSet^*(q), {q<\infty}$ \\
			$\nabla \Sol^L_{\nicefrac{1}{2}}$ & $\kappa = 0$ & $q = r \in (q(A^*)', q(A))$, $p > q(A^*)'$ & $q = r \in (1,\infty)$ \\
		\end{tabular}
		\vspace{1em}
		\caption{Parameter restriction on $(p, q, r, \kappa)$ with $p,q,r \in (1,\infty]$ and $\kappa \in [0,1]$ satisfying $r \geq q$ for Proposition~\ref{prop:Duhamel_mapping_Z} and Corollary~\ref{cor:Duhamel_Lq}. {The critical numbers $q(A)$ and $q(A^*)'$ have been introduced in Section~\ref{subsec:prelim_elliptic}.}}
		\label{tab:Duhamel_Sobolev}
	\end{table}

	\begin{lemma}[Duhamel operators are SIOs]
		\label{lem:Duhamel_is_SIO}
		Let $S \in \bigl\{ \Soll^L_1, \nabla \Soll^L_1, \Sol^L_{\nicefrac{1}{2}}, \nabla \Sol^L_{\nicefrac{1}{2}} \bigr\}$ and $p$, $q$, $r$, $\kappa$ be subject to the restrictions imposed in Table~\ref{tab:Duhamel_Sobolev} for $A$ rough. Then, $S$ is an SIO of type $(\min(p,q),q,r,\kappa,\infty)$. If $A$ is regular, the same conclusion holds true for the time-truncated version\footnote{Compare with Remark~\ref{rem:SIO_global_estimate}.} of $S$.
	\end{lemma}

	\begin{proof}
		Properties (b) and (c) of an SIO are readily verified for a given operator $S \in \bigl\{ \Soll^L_1, \nabla \Soll^L_1, \Sol^L_{\nicefrac{1}{2}}, \nabla \Sol^L_{\nicefrac{1}{2}} \bigr\}$. Indeed, the parameter restrictions in Table~\ref{tab:Duhamel_Sobolev} are designed to ensure the required off-diagonal estimates (with $M = \infty$) for the respective integral kernels, which yields (b). Property (c) was established in~\cite[Prop.~4.1]{AH2} for $\Soll^L_1$ and $\nabla \Soll^L_1$, and in~\cite[Prop.~5.8]{AH3} for $\Sol^L_{\nicefrac{1}{2}}$ and $\nabla \Sol^L_{\nicefrac{1}{2}}$.

		It remains to argue for (a). The strategy for all six non-singular cases when $\kappa >0$ is the same. We focus on the treatment for $\Sol^L_{\nicefrac{1}{2}}$ when $A$ is rough and only comment on the necessary changes for the remaining cases. This is Step~1 of the subsequent proof. The genuinely singular cases are discussed in Step~2.

		\textbf{Step 1}: non-singular cases $\kappa > 0$. As said, we just consider $S = \Sol^L_{\nicefrac{1}{2}}$ for $A$ rough in detail. Fix $q$ and $r$ as in Table~\ref{tab:Duhamel_Sobolev}.
		We distinguish the three cases appearing in the definition of the set $\SobSet^*(q)$ separately.
		First, in the case $q < n+2$, the claim follows from $\L^q_x \to \L^r_x$ boundedness of the family $\{ \sqrt{t} \e^{-tL} \Div \}_{t>0}$ in conjunction with the fractional integral. Indeed, note first that Table~\ref{tab:Duhamel_Sobolev} is designed to ensure the necessary $\L^q_x \to \L^r_x$ bounds. Then, for $t>0$, calculate
		\begin{align}
			\| \Sol^L_\frac{1}{2}(F)(t) \|_{\L^r_x} \leq \int_0^t \| \e^{-(t-s)L} \Div(F(s))\|_{\L^r_x} \d s \lesssim \int_0^t (t-s)^{-\nicefrac{n}{2}(\nicefrac{1}{q} - \nicefrac{1}{r}) - \nicefrac{1}{2}} \| F(s) \|_{\L^q_x} \d s.
		\end{align}
		Put $\lambda = \nicefrac{1}{2} - \nicefrac{n}{2}(\nicefrac{1}{q} - \nicefrac{1}{r})$, then the preceding bound can be rewritten as
		\begin{align}
			\| \Sol^L_\frac{1}{2}(F)(t) \|_{\L^r_x} \lesssim T_\lambda(\| F(\cdot) \|_{\L^q_x})(t).
		\end{align}
		An application of Lemma~\ref{lem:fractional_integral} with $\gamma_0 \coloneqq 0$ and $\gamma_1 \coloneqq \nicefrac{1}{2} -\nicefrac{n}{2}(\nicefrac{1}{q} - \nicefrac{1}{r})$ now gives the claim. Its application is justified since $r \in \SobSet^*(q)$ implies $\lambda \geq \nicefrac{1}{q} - \nicefrac{1}{r}$. In particular, $\lambda > 0$.

		Second, when $q > n+2$, we use $\L^q_x \to \L^\infty_x$ boundedness of $\{ \sqrt{t} \e^{-tL} \Div \}_{t>0}$ and Hölder's inequality to give, for $0 < t \leq \infty$,
		\begin{align}
			\| \Sol^L_\frac{1}{2}(F)(t) \|_{\L^{\vphantom{q}\infty}_x} &\lesssim \int_0^t (t-s)^{- \frac{1}{2} - \frac{n}{2q}} \| F(s) \|_{\L^{q}_x} \d s \\
			&\leq \Bigl( \int_0^t (t-s)^{(- \frac{1}{2} - \frac{n}{2q})q'} s^{q'} \dds \Bigr)^\frac{1}{q'} \| F \|_{\L^q_{\vphantom{x}}}.
		\end{align}
		The integral converges owing to $q > n+2$, and it evaluates to $t^{\nicefrac{1}{2} - \nicefrac{n}{2q}}$. Therefore, $\Sol^L_{\nicefrac{1}{2}}$ maps $\L^q$ to $\L^\infty_{\nicefrac{1}{2} - \nicefrac{n}{2q}}$. Moreover, by interpolation with $\L^q \to \L^q_{\nicefrac{1}{2}}$ boundedness from the previous case, we obtain any $r \in [q, \infty]$.

		Third, when $q = n+2$, let $r\in [q,\infty)$. Note $r > 2$. Define $\theta \coloneqq \nicefrac{2}{r} \in (0,1)$. Next, define $\tilde q$ through $\frac{1}{n+2} = \frac{\theta}{2} + \frac{1-\theta}{\tilde q}$. One has $\tilde q > n+2$, and $\tilde q \leq \infty$ is equivalent to the assumption $r \geq n+2$. Now, $\L^q = \L^{n+2} \to \L^r_{\nicefrac{1}{2} + \nicefrac{n}{2}(\nicefrac{1}{q} - \nicefrac{1}{r})}$ boundedness follows by interpolation between the bounds $\L^2 \to \L^2_{\nicefrac{1}{2}}$ and $\L^{\tilde q} \to \L^\infty_{\nicefrac{1}{2} - \nicefrac{n}{2\tilde q}}$, keeping in mind that the defining relation for $\tilde q$ can be rewritten to $\frac{1-\theta}{\tilde q} = \frac{1}{q} - \frac{1}{r}$.

		For the other cases listed in Table~\ref{tab:Duhamel_Sobolev}, observe three facts: first, $r$ can be improved by one Sobolev conjugate when $\kappa = \nicefrac{1}{2}$ and by two Sobolev conjugates if $\kappa = 1$. Second, the upper and/or lower bounds for $q$ and $r$ are dictated by the availability of $\L^q_x \to \L^r_x$ boundedness described in Section~\ref{subsec:prelim_elliptic}. Third, when the coefficients are regular, off-diagonal estimates are only available until a fixed but arbitrary time $T < \infty$, so we have to work with the time-truncated version of the operator, compare with Remark~\ref{rem:SIO_global_estimate}. Keeping them in mind, the above-presented proof is easily adapted to the remaining non-singular cases.

		\textbf{Step 2}: singular case $\kappa = 0$. We focus on the regular coefficients, as for rough coefficients it is similar and even simpler.
		Let $F\in \L^2_\cc(\R^{1+n}_+)$.
		Note that $(L+1)^{-\frac{1}{2}} \Sol^L_{\nicefrac{1}{2}}(F) = \Soll^L_1 \bigl((L+1)^{-\frac{1}{2}} \Div(F) \bigr)$, where $(L+1)^{-\frac{1}{2}} \Div$ is defined as the adjoint of the Riesz transformation $\nabla (L^*+1)^{-\frac{1}{2}}$ associated with $L^*+1$.
		Hence, we can write
		\begin{align}
			\nabla \Sol^L_\frac{1}{2}(F) = \nabla (L+1)^{-\frac{1}{2}} (L+1) \Soll^L_1((L+1)^{-\frac{1}{2}} \Div(F)).
		\end{align}
		The Riesz transform $\nabla (L+1)^{-\frac{1}{2}}$ is $\L^q$-bounded owing to the regular coefficients of $L$, see~\cite[Thm.~4.1]{Memoirs}. By duality, the same is true for $(L+1)^{-\frac{1}{2}} \Div$, where we use that the notion of regular coefficients is invariant under taking the adjoint. Hence, $\L^q$-boundedness of $\nabla \Sol^L_{\nicefrac{1}{2}}$ follows from $\L^q$-boundedness of $(L+1) \Soll^L_1$. To see that $(L+1) \Soll^L_1$ is $\L^q$-bounded, we consider the terms $L \Soll^L_1$ and $\Soll^L_1$ separately. For the second term, this follows from arguments similar (but easier) to Step~1. Note, however, that the operator norm for this term will now depend on $T$. Boundedness of the first term is equivalent to the question of \emph{$\L^q$-maximal regularity} of $L$. The latter was established %
		in~\cite{Coulhon_Duong} under the assumption that $L$ satisfies Gaussian bounds, see Section~\ref{subsec:prelim_elliptic}. Since $L$ has real coefficients, Gaussian bounds for $L$ indeed hold. This concludes Step~2.
	\end{proof}

	Combining the lemma with Theorem~\ref{thm:sio}, we obtain boundedness for the Duhamel operators on $\Z$-spaces.

	\begin{proposition}[Hypercontractivity of Duhamel operators on $\Z$-spaces]
		\label{prop:Duhamel_mapping_Z}
		Fix $T \in (0,\infty)$. Let $S \in \bigl\{ \Soll^L_1, \nabla \Soll^L_1, \Sol^L_{\nicefrac{1}{2}}, \nabla \Sol^L_{\nicefrac{1}{2}} \bigr\}$ and $p$, $q$, $r$, $\kappa$ be subject to the restrictions from Table~\ref{tab:Duhamel_Sobolev}.
		Let $\beta > -1$.
		Then, the operator $S$ admits a bounded extension\footnote{If $p = \infty$, an explicit representation of this extension is stated in Theorem~\ref{thm:sio}.} to $\Z^{p,q}_\beta(T) \to \Z^{p,r}_{\beta + \kappa}(T)$. The mapping property remains true for $\Zv$-spaces when $p=\infty$.
		In each case, the operator norm is non-decreasing in $T$, and if the coefficients are rough, then the operator norm is in fact independent of $T$, so that we can allow $T = \infty$.
	\end{proposition}

	If we employ Corollary~\ref{cor:SIO_weighted_Lq} instead of Theorem~\ref{thm:sio}, we obtain the subsequent result.

	\begin{corollary}[Hypercontractivity of Duhamel operators on Lebesgue spaces]
		\label{cor:Duhamel_Lq}
		Fix $T \in (0,\infty)$. Let $S \in \bigl\{ \Soll^L_1, \nabla \Soll^L_1, \Sol^L_{\nicefrac{1}{2}}, \nabla \Sol^L_{\nicefrac{1}{2}} \bigr\}$ and $\kappa$, $q$, $r$ be subject to the restrictions imposed in Table~\ref{tab:Duhamel_Sobolev} with $p \coloneqq q$. Let $\beta > -1$. Then, the operator $S$ admits a bounded extension $\L^q_\beta(T) \to \L^r_{\beta + \kappa - \nicefrac{n}{2}(\nicefrac{1}{q} - \nicefrac{1}{r})}(T)$, with operator norm non-decreasing in $T$.
		The mapping property remains true for $\Lv^\infty$-spaces when $q=\infty$ and/or $r = \infty$.
		In each case, the operator norm is non-decreasing in $T$, and if the coefficients are rough, then the operator norm is in fact independent of $T$, so that we can allow $T = \infty$.
	\end{corollary}

	From these preliminary results, we derive existence of (very) weak solutions to~\eqref{eq:LP}.

	\begin{proposition}[Duhamel solutions to~\eqref{eq:LP}]
		\label{prop:Duhamel_Z_weak}
		Fix $T \in (0,\infty)$. Let $\beta > -1$, $p,q,\tilde q \in (1,\infty]$, $f \in \Z^{p,q}_\beta(T)$ and $F \in \Z^{p,\tilde q}_\beta(T)$. %
		If $A$ is rough, assume in addition $\tilde q \geq 2$.
		Then, the function $u \coloneqq \Soll^L_1(f) + \Sol^L_{\nicefrac{1}{2}}(F)$ is a very weak solution to~\eqref{eq:LP}.
		Moreover, if $q \geq 2_*$ and $\tilde q \geq 2$, then $u$ is even a weak solution to~\eqref{eq:LP}
	\end{proposition}

	\begin{proof}
		We focus on the case where $A$ is rough and we assume that $F = 0$ for brevity of the presentation.
		The other cases are analogous.
		Let us also mention that, when $\tilde q \geq 2$, the existence of weak solutions for non-trivial $F$ was treated in more generality in~\cite{AH3}.

		Throughout the proof, we use the second inclusion in~\eqref{eq:Z_loc_embedding} freely.

		\textbf{Step 1}: the case $p = q \leq 2$.
		Let $T \in (0,\infty)$ and put $Q \coloneqq (0,T) \times \R^n$.
		By Corollary~\ref{cor:Duhamel_Lq}, $\Soll^L_1$ and $\nabla \Soll^L_1$ extend to bounded operators
		\begin{align}
			\Soll^L_1 &\colon \L^q_\beta(T) \to \L^q_{\beta + 1}(T) \subseteq \L^1_\loc(Q), \\
			\nabla \Soll^L_1 &\colon \L^q_\beta(T) \to \L^q_{\beta + \frac{1}{2}}(T) \subseteq \L^1_\loc(Q).
		\end{align}
		By density~\cite[Ex.~9.4.1]{GrafakosModern}, pick a sequence $f_n \in \L^q_\beta(T) \cap \L^2_\cc(Q)$ such that $f_n \to f$ in $\L^q_\beta(T) \subseteq \L^1_\loc(Q)$. By continuity of $\Soll^L_1$ we find $u_n \coloneqq \Soll^L_1(f_n) \to \Soll^L_1(f) \eqqcolon u$ in $\L^1_\loc(Q)$. Using this and the continuity of $\nabla \Soll^L_1$ in conjunction with the definition of the weak gradient, deduce $\nabla u_n \to \nabla u$ in $\L^1_\loc(Q)$. By the $\L^2$-theory presented at the beginning of this subsection, $u_n$ is a weak solution to~\eqref{eq:LP} with forcing term $f_n$, that is to say, for $\phi \in \Cont^\infty_\cc(Q)$ there holds the integral identity
		\begin{align}
			\label{eq:weak_approximate_identity}
			\iint_Q -u_n \partial_t \phi + A \nabla u_n \cdot \nabla \phi \d x \d t = \iint_Q f_n \phi \d x \d t.
		\end{align}
		Since $u_n \to u$, $\nabla u_n \to \nabla u$ and $f_n \to f$ in $\L^1_\loc(Q)$, we can take the limit $n\to \infty$ in~\eqref{eq:weak_approximate_identity} to find
		\begin{align}
			\iint_Q -u \partial_t \phi + A \nabla u \cdot \nabla \phi \d x \d t = \iint_Q f \phi \d x \d t, \qquad \phi \in \Cont^\infty_\cc(Q).
		\end{align}
		Since $u, \nabla u \in \L^1_\loc(Q)$, this shows that $u$ is a very weak solution to~\eqref{eq:LP}.

		\textbf{Step 2}: case $p \geq q$. By the nesting property of weighted $\Z$-spaces, we can suppose $q \leq 2$. The strategy is similar as in Step~1. Let again $T \in (0,\infty)$ and put $Q \coloneqq (0,T) \times \R^n$.
		By Proposition~\ref{prop:Duhamel_mapping_Z}, $\Soll^L_1$ and $\nabla \Soll^L_1$ extend to bounded operators
		\begin{align}
			\Soll^L_1 &\colon \Z^{p,q}_\beta(T) \to \Z^{p,q}_{\beta + 1}(T) \subseteq \L^1_\loc(Q), \\
			\nabla \Soll^L_1 &\colon \Z^{p,q}_\beta(T) \to \Z^{p,q}_{\beta + \frac{1}{2}}(T) \subseteq \L^1_\loc(Q).
		\end{align}
		We claim that there exists a sequence $f_n \in \L^q_\cc(Q)$ such that $f_n \to f$ in $\L^1_\loc(Q)$, $u_n \coloneqq \Soll^L_1(f_n) \to \Soll^L_1(f) \eqqcolon u$ in $\L^1_\loc(Q)$ and $\nabla u_n \to \nabla u$ in $\L^1_\loc(Q)$. Indeed, if $p < \infty$, then by density (see Section~\ref{subsec:prelim_function_spaces}) there exists a sequence $f_n \in \L^q_\cc(Q)$ that converges in $\Z^{p,q}_\beta(T)$ to $f$. Then, the claims for $u_n$ and $\nabla u_n$ follow by continuity of $\Soll^L_1$ and $\nabla \Soll^L_1$ as in Step~1. If $p = \infty$, we use the sequence $f_n$ provided by Theorem~\ref{thm:sio}. Note that we can take the same sequence for $\Soll^L_1$ and $\nabla \Soll^L_1$. The convergence of $\nabla u_n \to \nabla u$ uses continuity in conjunction with the definition of the weak gradient in an analogous way as before.

		To conclude Step~2, we use the same approximation argument as in Step~1, this time exploiting that $u_n$ is a weak solution to~\eqref{eq:LP} with forcing term $f_n$ owing to Step~1.

		\textbf{Step 3}: weak solutions. Now, suppose in addition that $q \geq 2_*$. Then, Proposition~\ref{prop:Duhamel_mapping_Z} provides bounded extensions %
		\begin{align}
			\Soll^L_1 &\colon \Z^{p,q}_\beta(T) \to \Z^{p,2}_{\beta + 1}(T) \subseteq \L^2_\loc(Q), \\
			\nabla \Soll^L_1 &\colon \Z^{p,q}_\beta(T) \to \Z^{p,2}_{\beta + \frac{1}{2}}(T) \subseteq \L^2_\loc(Q).
		\end{align}
		Therefore, $u, \nabla u \in \L^2_\loc(Q)$. Since $u$ is a very weak solution to~\eqref{eq:LP} by Step~2, this shows that $u$ is in fact a weak solution to~\eqref{eq:LP}.
	\end{proof}

	\begin{remark}
		We emphasize that it is possible in the previous proposition that the forcing term $f$ is \emph{not} itself locally square-integrable. Such forcing terms were not considered in~\cite{AH2}.
		{To the contrary, in the variational setting for parabolic equations in divergence form, forcing terms in $f\in \L^{2_*}$ can be considered, see for instance \cite{ABES19}.
		So, our condition $q\ge 2_{*}$ is natural.}
	\end{remark}

	To conclude this section, we discuss the vanishing trace of the Duhamel operators.

	\begin{proposition}[Vanishing trace for Duhamel solutions]
		\label{prop:Duhamel_Z_trace}
		Fix $T \in (0,\infty)$.
		Let $\beta > -1$, $q \in [2_*, \infty]$, $\tilde q \in [2,\infty]$ and $p \in (1,\infty]$. %
		Moreover, let
		\begin{itemize}
			\item $f \in \Z^{p,q}_{\beta - \nicefrac{1}{2}}$ if in addition $\beta > - \nicefrac{1}{2}$, otherwise put $f = 0$,
			\item $F \in \Z^{p,\tilde q}_{\beta}$.
		\end{itemize}
		Then, the function $u \coloneqq \Soll^L_1(f) + \Sol^L_{\nicefrac{1}{2}}(F)$ satisfies $u(0) = 0$ in the sense of Definition~\ref{def:initial_condition}.
	\end{proposition}

	\begin{proof}
		Keeping $\Z^{p,\tilde q}_\beta \subseteq \Z^{p,2}_\beta$ in mind, the term $\Sol^L_{\nicefrac{1}{2}}(F)$ has already been treated in~\cite[Thm.~5.1~(d)]{AH3}. Likewise, the term $\Soll^L_1(f)$ was treated in~\cite{AH2} under the assumption that $q \geq 2$. Their proof relies on the standard {localized energy estimate presented in~\cite[Prop.~3.6]{AMP19}}. If the more refined version in~\cite[Proof of Prop.~4.3]{ABES19} is used instead, then the proof immediately extends to our setting, keeping the inclusion {$\Z^{p,q}_{\beta-\nicefrac{1}{2}} \subseteq \Z^{p,2_*}_{\beta-\nicefrac{1}{2}}$} in mind.
	\end{proof}

	\section{The free evolution of Besov data}
	\label{sec:caloric}

	We investigate solvability of the initial value problem
	\begin{align}
		\label{eq:caloric}
		\tag{HC}
		\partial_t u -\Div(A\nabla u) = 0, \quad u(0) = u_0,
	\end{align}
	where $u_0$ is taken from a homogeneous Besov space $\dot \B^\alpha_{p,p}$.

	Throughout this section, put $L = -\Div(A\nabla)$. Let $p \in (1,\infty]$ and $\alpha \in (-1,0)$ satisfying $p \geq p(\alpha,A)$. Owing to Proposition~\ref{prop:initial_value_problem_hedong}, the mapping
	\begin{align}
		\dot \B^\alpha_{p,p} \ni u_0 \mapsto \Ext_L(u_0) \coloneqq u \in \Z^{p,2}_{\nicefrac{\alpha}{2}}
	\end{align}
	is well-defined, where $u$ is the unique weak solution to the initial value problem~\eqref{eq:caloric}. Call $\Ext_L(u_0)$ the \emph{free evolution} of $u_0$.
	If $u_0$ belongs to $\L^2_x$ in addition, then the free evolution can be expressed by virtue of the heat semigroup generated by $-L$.

	The following result is the main tool for the free evolution of Besov data.

	\begin{proposition}[Free evolution of Besov data]
		\label{prop:caloric_extension_Besov}
		Let $p \in (1,\infty]$ and $\alpha \in (-1,0)$ satisfying $p \geq p(\alpha,A)$. Set $\sigma \coloneqq \nicefrac{n}{p} - \alpha$. Let $u_0 \in \dot\B^\alpha_{p,p}$ if $p < \infty$ and $u_0 \in \vv\dot\B^\alpha_{\infty,\infty}$ if $p = \infty$. Then, for all $r \in [p,\infty]$ and $q\in (1,\infty]$, one has $\Ext_L(u_0) \in \Z^{r,q}_{\nicefrac{n}{2r} - \nicefrac{\sigma}{2}}$ and $\nabla\Ext_L(u_0) \in \Z^{r,2}_{\nicefrac{n}{2r} - \nicefrac{\sigma}{2}-\nicefrac{1}{2}}$ along with the estimates
		\begin{align}
			\label{eq:caloric_extension_Besov}
			\|\Ext_L(u_0)\|_{ {\Z^{r,q}_{\frac{n}{2r} - \frac{\sigma}{2}}}} {+ \|\nabla \Ext_L(u_0)\|_{ {\Z^{r,2}_{\frac{n}{2r} - \frac{\sigma}{2}-\frac{1}{2}}}}} \lesssim \|u_{0}\|_{\dot\B^\alpha_{p,p}}.
		\end{align}
		If $r = \infty$, we even have $\Ext_L(u_0) \in {\Zv^{\infty,q}_{-\nicefrac{\sigma}{2}}}$ {and $\nabla\Ext_L(u_0) \in {\Zv^{\infty,2}_{-\nicefrac{\sigma}{2}-\nicefrac{1}{2}}}$}.
		Put $u\coloneqq \Ext_L(u_0)$. Then $u$ is the unique weak solution to the initial value problem~\eqref{eq:caloric}.
	\end{proposition}

	\begin{remark}
		For non-linear problems, the parameter $\sigma$ is dictated by the scaling behavior of the equation, compare with Section~\ref{subsec:scaling}.
	\end{remark}

	For its proof, we need some preparatory results. One of them is the following self-improving property of the free evolution in weighted $\Z$-spaces.

	\begin{proposition}[Self-improving property]
		\label{prop:free_evolution_improve_Loo}
		Fix $T \in (0, \infty]$.
		Let $\beta \in \R$ and {$p,q,r \in (1,\infty]$ with $r > q$}. Then, $\Ext_L$ satisfies the self-improving estimate
		\begin{align}
			\| \Ext_L(u_0) \|_{\Z^{p,r}_\beta(T)} \lesssim \| \Ext_L(u_0) \|_{\Z^{p,q}_\beta(T)}, \qquad u_0 \in \Soo.
		\end{align}
		The implicit constant is independent of $T$.
	\end{proposition}

	Note that the converse inequality follows from the nesting property of weighted $\Z$-spaces. The argument is inspired by~\cite[Lem.~5.13]{Auscher_Frey}.

	\begin{proof}
		Let $u_0 \in \Soo \subseteq \L^2_x$.
		For convenience, put $G(t) \coloneqq \Ext_L(u_0)(t)= e^{-tL}u_{0}$. By the semigroup property, there holds for every $t>0$ the averaging identity
		\begin{align}
			G(t) = \fint_{\nicefrac{t}{4}}^{\nicefrac{t}{2}} \e^{-(t-s)L} G(s) \d s.
		\end{align}
		Fix $x \in \R^n$ and $\tau \in (0,T]$. Let $t\in (\nicefrac{\tau}{2}, \tau)$. Using Minkowski's inequality and $\L^q_x \to \L^r_x$ off-diagonal estimates of the heat semigroup, calculate
		\begin{align}
			\Bigl(\fint_{\B(x,\sqrt{\tau})} |G(t,y)|^r \d y \Bigr)^\frac{1}{r} &\leq \fint_{\nicefrac{t}{4}}^{\nicefrac{t}{2}} \tau^{-\frac{n}{2r}} \| \e^{-(t-s)L} G(s) \|_{\L^r(\B(x,\sqrt{\tau}))} \d s \\
			&\leq \sum_{j \geq 1} \fint_{\nicefrac{t}{4}}^{\nicefrac{t}{2}} \tau^{-\frac{n}{2r}} \| \e^{-(t-s)L} (G(s) \ind_{\ann_j(x,\tau)}) \|_{\L^r(\B(x,\sqrt{\tau}))} \d s \\
			&\lesssim \sum_{j \geq 1} \fint_{\nicefrac{t}{4}}^{\nicefrac{t}{2}} \tau^{-\frac{n}{2r}} (t-s)^{\frac{n}{2r} - \frac{n}{2q}} \e^{-c\frac{2^j \tau}{t-s}} \| G(s) \|_{\L^q(\B(x,\sqrt{2^{j+1}\tau}))} \d s.
		\end{align}
		Now, use $t-s\approx t \approx \tau$, $(\nicefrac{t}{4}, \nicefrac{t}{2}) \subseteq (\nicefrac{\tau}{8}, \nicefrac{\tau}{2})$ and Jensen's inequality, to find
		\begin{align}
			\Bigl(\fint_{\B(x,\sqrt{\tau})} |G(t,y)|^r \d y \Bigr)^\frac{1}{r} &\lesssim \sum_{j \geq 1} \e^{-c2^j} \fint_{\nicefrac{\tau}{8}}^{\nicefrac{\tau}{2}} \tau^{- \frac{n}{2q}} \| G(s) \|_{\L^q(\B(x,\sqrt{2^{j+1}\tau}))} \d s \\
			&\lesssim \sum_{j \geq 1} \e^{-c2^j} \Bigl( \fint_{\nicefrac{\tau}{8}}^{\nicefrac{\tau}{2}} \tau^{- \frac{n}{2}} \| G(s) \|_{\L^q(\B(x,\sqrt{2^{j+1}\tau}))}^q \d s \Bigr)^\frac{1}{q}.
		\end{align}
		The right-hand side of the last bound is independent of $t$, so applying an $\L^r$-average over $t$ yields
		\begin{align}
				\Bigl(\fint_{\nicefrac{\tau}{2}}^\tau \fint_{\B(x,\sqrt{\tau})} |G(t,y)|^r \d y \d t \Bigr)^\frac{1}{r} \lesssim \sum_{j \geq 1} \e^{-c2^j} \Bigl( \fint_{\nicefrac{\tau}{8}}^{\nicefrac{\tau}{2}} \tau^{- \frac{n}{2}} \int_{\B(x,\sqrt{2^{j+1}\tau})} |G(s,y)|^q \d y \d s \Bigr)^\frac{1}{q}.
		\end{align}
		Now, multiply the last bound by $\tau^{-\beta}$, use $t \approx \tau \approx s$, and apply the $\L^p$-norm in conjunction with a change of angle to deduce
		\begin{align}
			\| G \|_{\Z^{p,r}_\beta(T)} \lesssim \sum_{j \geq 1} \e^{-c2^j} 2^\frac{jn}{2\min(p,q)} \| G \|_{\Z^{p,q}_\beta(T)}.
		\end{align}
		The sum in $j$ converges. This completes the proof by definition of $G$.
	\end{proof}

	\begin{lemma}
		\label{lem:heat_extension_bounded}
		Let $u_0 \in \Soo$. Then, $\Ext_{L}(u_0) \in \Zv^{\infty,{q}}_\beta$ for all $\beta < 0$ {and $q\in (1,\infty]$}.
	\end{lemma}

	\begin{proof}
		We focus on the case $q < \infty$ for simplicity. The case $q = \infty$ is similar and even simpler.
		Let $u_0 \in \Soo \subseteq \L^2_x$.
		Fix $x\in \R^n$ and $T \in (0,\infty]$. Let $t \in (0,T]$. Using $\L^\infty_x$-boundedness of the heat semigroup, calculate
		\begin{align}
			\label{eq:heat_extension_bounded1}
			\begin{split}
				\fint_{t/2}^t \fint_{\B(x,\sqrt{t})} |s^{-\beta} \e^{-sL} u_0|^q \d y \d s \leq \fint_{t/2}^t s^{-q\beta} \| \e^{-sL} u_0 \|_\infty^q \d s &\lesssim \fint_{t/2}^t s^{-q\beta} \d s \| u_0 \|_\infty^q \\
				&\approx t^{-q\beta} \| u_0 \|_\infty^q.
			\end{split}
		\end{align}
		Recall that $\beta$ is negative.
		Hence, $t^{-q\beta} \leq T^{-q\beta}$, so that the right-hand side of~\eqref{eq:heat_extension_bounded1} can be made arbitrarily small by taking $T$ small, depending on $\| u_0 \|_\infty$ and the operator norm of the heat semigroup. Thus, taking the supremum over $x\in \R^n$ gives the claim.
	\end{proof}

	Up to some technicalities, the proof of Proposition~\ref{prop:caloric_extension_Besov} now follows readily from the self-improving property (Proposition~\ref{prop:free_evolution_improve_Loo}), the smallness property for test-functions (Corollary~\ref{lem:heat_extension_bounded}), and the preliminary facts from Section~\ref{subsec:prelim_function_spaces}.

	\begin{proof}[Proof of Proposition~\ref{prop:caloric_extension_Besov}]
		By definition, $\Ext_L(u_0)$ is the unique weak solution to~\eqref{eq:caloric}. The estimates for $\nabla \Ext_L(u_0)$ follow from Proposition~\ref{prop:initial_value_problem_hedong} and the weighted $\Z$-space embedding. %
		Hence, we concentrate on the stated estimates for $\Ext_L(u_0)$.

		Write $X \coloneqq \dot\B^\alpha_{p,p}$ if $p < \infty$ and $X \coloneqq \vv\dot\B^\alpha_{\infty,\infty}$ if $p=\infty$. In both cases, $\Soo$ is dense in $X$, see Remark~\ref{rem:def_vvB}. Let $u_0 \in \Soo$ first. Using the weighted $\Z$-space embedding with $r \in [p,\infty]$, the self-improving property in Proposition~\ref{prop:free_evolution_improve_Loo} and Proposition~\ref{prop:initial_value_problem_hedong}, estimate
		\begin{align}
			\| \Ext_L(u_0) \|_{\Z^{r,q}_{\frac{n}{2r} - \frac{\sigma}{2}}} \lesssim \| \Ext_L(u_0) \|_{\Z^{p,q}_{\frac{n}{2p} - \frac{\sigma}{2}}} \lesssim \| \Ext_L(u_0) \|_{\Z^{p,2}_{\frac{n}{2p} - \frac{\sigma}{2}}} \lesssim \| u_0 \|_{\X},
		\end{align}
		where we used the identity $\nicefrac{n}{2p} - \nicefrac{\sigma}{2} = \nicefrac{\alpha}{2}$ in the last step.
		By density, this bound extends to all of $\X$, which completes the proof of~\eqref{eq:caloric_extension_Besov}.

		If $r = \infty$, it remains to show $\Ext_L(u_0) \in \Zv^{\infty,q}_{%
		- \nicefrac{\sigma}{2}}$ and $\nabla \Ext_L(u_0) \in \Zv^{\infty,2}_{%
		- \nicefrac{\sigma}{2}-\nicefrac{1}{2}}$.
		On the one hand, Lemma~\ref{lem:heat_extension_bounded} provides the claim for $\Ext_L(u_0)$ when $u_0 \in \Soo$. On the other hand, we have just shown that $\Ext_L$ is $\X \to \Z^{\infty,q}_{%
		- \nicefrac{\sigma}{2}}$ bounded. Thus, since $\Zv^{\infty,q}_{%
		- \nicefrac{\sigma}{2}}$ is a closed subspace of $\Z^{\infty,q}_{%
		- \nicefrac{\sigma}{2}}$, the claim follows by density of $\Soo$ in $X$.
		Finally, from $\Ext_L(u_0) \in \Zv^{\infty,2}_{%
		- \nicefrac{\sigma}{2}}$ and Caccioppoli's inequality~\cite[Cor.~4.6]{ABES19} there follows immediately $\nabla \Ext_L(u_0) \in \Zv^{\infty,2}_{%
		- \nicefrac{\sigma}{2}-\nicefrac{1}{2}}$, which completes the proof.
	\end{proof}

	\section{Well-posedness of the linear equation}
	\label{sec:linear_wp}

	For completeness, we assemble the results of Sections~\ref{subsec:duhamel_cauchy} and~\ref{sec:caloric} to obtain well-posedness of weak solutions to the linear Cauchy problem with rough coefficients
	\begin{align}
		\label{eq:CP}
		\tag{CP}
		\partial_t u - \Div(A\nabla u) = f + \Div(F), \quad u(0) = u_0,
	\end{align}
	where $f$ and $F$ are taken from appropriate weighted $\Z$-spaces and $u_0$ is taken from a homogeneous Besov space. Solution and data spaces are not tied to the locally square-integrable regime, which is a novelty of our approach. As before, put $L = -\Div(A\nabla)$.

	\begin{theorem}[Well-posedness of~\eqref{eq:CP} in weighted $\Z$-spaces]
		\label{thm:linear_wp}
		Fix $T \in (0,\infty]$.
		Let $p \in (1, \infty]$ and $\alpha \in (-1,0)$ satisfying $p \geq p(\alpha,A)$.
		Set $\sigma \coloneqq \nicefrac{n}{p} - \alpha$.
		The data is supposed to be as follows:
		\begin{enumerate}[(i)]
			\item $u_0 \in \dot\B^\alpha_{p,p}$ if $p < \infty$ and $u_0 \in \vv\dot\B^\alpha_{\infty,\infty}$ if $p = \infty$,
			\item $f \in \Z^{\tilde r, \tilde q}_{\nicefrac{n}{2\tilde r} - \nicefrac{\sigma}{2} - 1}(T)$ with $2_{*}\le \tilde q$ and $\tilde r < \nicefrac{n}{\sigma}$,
			\item $F \in \Z^{\hat r, \hat q}_{\nicefrac{n}{2\hat r} - \nicefrac{\sigma}{2} - \nicefrac{1}{2}}(T)$ with $2\leq \hat q$ and $\hat r < \nicefrac{n}{(\sigma-1)}$.
		\end{enumerate}

		Then, the following hold:
		\begin{itemize}
			\item (Existence) The function $u = \Ext_L(u_0) + \Soll^L_1(f) + \Sol^L_{\nicefrac{1}{2}}(F)$ is a weak solution to~\eqref{eq:CP}.
			\item (Regularity) One has $u \in \Z^{r,q}_{\nicefrac{n}{2r} - \nicefrac{\sigma}{2}}(T)$ and $\nabla u \in \Z^{p,2}_{\nicefrac{n}{2p} - \nicefrac{\sigma}{2} - \nicefrac{1}{2}}(T)$ for all $r,q\in (1,\infty]$ satisfying $\max(p, \tilde r, \hat r) \leq r$ and $q \in \SobSet^{**}(\tilde q) \cap \SobSet^*(\hat q)$. All bounds are independent of $T$.
			\item (Uniqueness) The weak solution $u$ to~\eqref{eq:CP} is unique among all weak solutions satisfying $\nabla u \in \Z^{p,2}_{\nicefrac{n}{2p} - \nicefrac{\sigma}{2} - \nicefrac{1}{2}}(T)$ if $\max(\tilde r, \hat r) \leq p$.
		\end{itemize}
		As usual, if one of $r$, $\tilde r$, $\hat r$ is infinite, the corresponding $\Z^\infty$-space can be replaced by a $\Zv^\infty$-space.
	\end{theorem}

	\begin{proof}
		To show the theorem, we just have to assemble the results of Sections~\ref{subsec:duhamel_cauchy} and~\ref{sec:caloric}.

		\textbf{Step 1}: Existence. For $\Ext_L(u_0)$, the claim was established in Proposition~\ref{prop:caloric_extension_Besov}. For $\Soll^L_1(f)$ and $\Sol^L_{\nicefrac{1}{2}}(F)$, we aim to apply Proposition~\ref{prop:Duhamel_Z_weak}. All requirements are clear except the weight restrictions. But $\nicefrac{n}{2\tilde r} - \nicefrac{\sigma}{2} - 1 > -1$ is equivalent to $\tilde r < \nicefrac{n}{\sigma}$ and $\nicefrac{n}{2\hat r} - \nicefrac{\sigma}{2} - \nicefrac{1}{2} > -1$ is equivalent to $\hat r < \nicefrac{n}{(\sigma-1)}$, which justifies the application of Proposition~\ref{prop:Duhamel_Z_weak}.

		\textbf{Step 2}: Regularity. The claim for $\Ext_L(u_0)$ follows again from Proposition~\ref{prop:caloric_extension_Besov} using $r \geq p$. For the Duhamel terms, we rely on Proposition~\ref{prop:Duhamel_mapping_Z} in conjunction with the weighted $\Z$-space embedding. The restrictions
		on $(p,\tilde q,q,1) $ for $\Soll^L_1$,
		on $(p,\min(\tilde q,2),2, \nicefrac1 2) $ for $\nabla \Soll^L_1$,
		on $(p, \hat q, q, \nicefrac1 2)$ for $\Sol^L_{\nicefrac{1}{2}}$ and on $(p, 2, 2, 0)$ for $\nabla \Sol^L_{\nicefrac{1}{2}}$
		in Table~\ref{tab:Duhamel_Sobolev} are all satisfied. For the gradient terms, we potentially use the nesting property of weighted $\Z$-spaces. Note that the factors $-1$ and $-\nicefrac{1}{2}$ appearing in the weights for $f$ and $F$ are compensated by the different singularity parameters $\kappa$ in Table~\ref{tab:Duhamel_Sobolev}.

		\textbf{Step 3}: Uniqueness. By the nesting property of weigthed $\Z$-spaces, we deduce $f \in \Z^{\tilde r, 2_*}_{\nicefrac{n}{2\tilde r} - \nicefrac{\sigma}{2} - 1}(T)$ and $F \in \Z^{\hat r, 2}_{\nicefrac{n}{2\hat r} - \nicefrac{\sigma}{2} - \nicefrac{1}{2}}(T)$. Hence, we can apply the regularity part with $q = 2$ to obtain $\nabla u \in \Z^{p,2}_{\nicefrac{n}{2p} - \nicefrac{\sigma}{2} - \nicefrac{1}{2}}(T)$. %
		Moreover, $u(0) = u_0$ by virtue of Proposition~\ref{prop:Duhamel_Z_trace} and Proposition~\ref{prop:caloric_extension_Besov}.
		We conclude by invoking Proposition~\ref{prop:linear_uniqueness}. Indeed, put $\beta \coloneqq \nicefrac{n}{2p} - \nicefrac{\sigma}{2} - \nicefrac{1}{2}$, then $\beta > -1$ is equivalent to $\alpha > -1$ and $p \geq p(\alpha,A) = p(2\beta + 1, A)$, so that the application of Proposition~\ref{prop:linear_uniqueness} is justified.
	\end{proof}

	\section{Mild solutions for the reaction--diffusion problem}
	\label{sec:mild}

	For the rest of this article, we turn our attention to the well-posedness of~\eqref{eq:rd}.
	To simplify the formulation of the well-posedness result, introduce the parameter functions
	\begin{align}
		\RD_{-}(n,\rho) = (n+2) \nicefrac{\rho}{2}, \qquad
		\RD_+(n,\rho) = n(1+\rho) \nicefrac{\rho}{2}.
	\end{align}
	We are going to show the following main result. A precise definition of weak solutions to the equation will be given in Definition~\ref{def:weak_solution_rd}.

	\begin{figure}[t]
		\centering

		\begin{subfigure}{0.48\textwidth}
			\centering
			\begin{tikzpicture}
				\begin{axis}[
					width=\linewidth, height=6cm,
					axis x line=middle, axis y line=middle,
					xlabel={$1/p$}, ylabel={$\alpha$},
					every axis x label/.style={
						at={(ticklabel cs:1)}, %
						anchor=south west,
					},
					every axis y label/.style={
						at={(axis description cs:0,1)}, %
						anchor=south, %
					},
					xmin=0, xmax=1.15,
					ymin=-1.2, ymax=0.15, %
					xtick={0},
					ytick={-1,0},
					axis equal image, %
					enlargelimits=false
					]

					\def\xleft{0.2525}
					\def\xright{1}
					\fill[blue!20]
					(axis cs:\xleft,0) --
					(axis cs:\xright,0) --
					(axis cs:0.7,-1) --
					(axis cs:\xleft,-1) -- cycle;

					\draw[blue,thick]
					(axis cs:\xleft,0) --
					(axis cs:\xright,0) --
					(axis cs:0.7,-1) --
					(axis cs:\xleft,-1) -- cycle;

					\draw[line width=2.5pt,red] (axis cs:0.2525,-0.8687) -- (axis cs:0.5556,0);

					\draw[dashed] (axis cs:0,-1.6667) -- (axis cs:0.2525,-0.8687);

					\draw[dashed] (axis cs:0.5556,0) -- (axis cs:1,1.3333);

				\end{axis}
			\end{tikzpicture}
			\caption*{Case $\rho = \nicefrac{6}{5} \geq 1$}
		\end{subfigure}
		\hfill
		\begin{subfigure}{0.48\textwidth}
			\centering
			\begin{tikzpicture}
				\begin{axis}[
					width=\linewidth, height=6cm,
					axis x line=middle, axis y line=middle,
					xlabel={$1/p$}, ylabel={$\alpha$},
					every axis x label/.style={
						at={(ticklabel cs:1)}, %
						anchor=south west, %
					},
					every axis y label/.style={
						at={(axis description cs:0,1)}, %
						anchor=south, %
					},
					xmin=0, xmax=1.05,
					ymin=-1.2, ymax=0.15,
					xtick={0},
					ytick={-1,0},
					axis equal image, %
					enlargelimits=false
					]

					\def\xleft{0.508} %
					\def\xright{1}
					\fill[blue!20]
					(axis cs:\xleft,0) --
					(axis cs:\xright,0) --
					(axis cs:0.7,-1) --
					(axis cs:\xleft,-1) -- cycle;

					\draw[blue,thick]
					(axis cs:\xleft,0) --
					(axis cs:\xright,0) --
					(axis cs:0.7,-1) --
					(axis cs:\xleft,-1) -- cycle;

					\def\n{3}
					\def\rho{0.75}
					\def\m{\n} %
					\def\b{-2/\rho} %

					\def\xbottom{(-1-\b)/\m} %
					\def\xtop{(0-\b)/\m} %

					\def\xsegleft{\xbottom> \xleft ? \xbottom : \xleft}
					\def\xsegright{\xtop<\xright ? \xtop : \xright}

					\draw[line width=2.5pt,red]
					(axis cs:0.5556,-1) -- (axis cs:0.8889,0);

					\draw[dashed]
					(axis cs:\xleft,\m*\xleft+\b) -- (axis cs:0.5556,-1);

					\draw[dashed]
					(axis cs:0.8889,0) -- (axis cs:\xright,\m*\xright+\b);

				\end{axis}
			\end{tikzpicture}
			\caption*{Case $\rho = \nicefrac{3}{4} < 1$}
		\end{subfigure}%

		\caption{Set $n = 3$. The blue areas depict all pairs $(p,\alpha)$ with $p \in (1,\RD_{+}(n,\rho))$, $\alpha \in (-1,0)$ and $p \geq p(\alpha,A)$. The red line segment is the intersection of the critical line (that is, the points satisfying $\nicefrac{2}{\rho} = \nicefrac{n}{p} - \alpha$) with the blue area. The red line leaves the blue area on its left edge if and only if $\rho \geq 1$. If $\rho < 1$, it may happen that the red line leaves the blue area on its right-hand slope.}
		\label{fig:line-rect-cases}
	\end{figure}
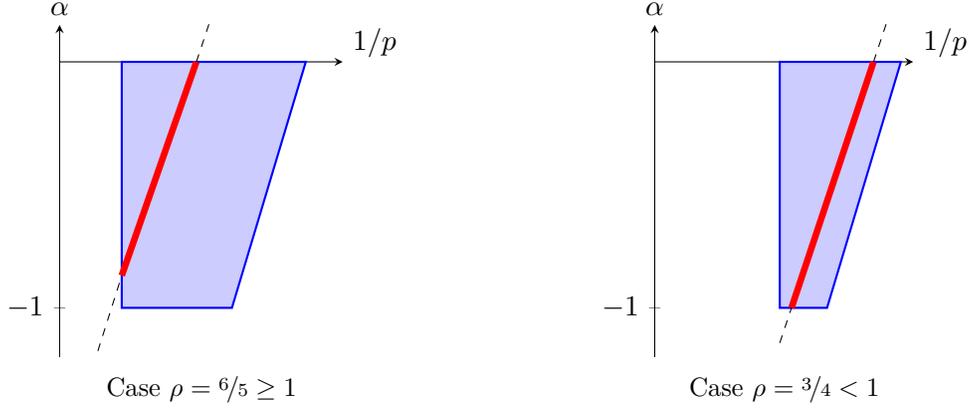

	\begin{theorem}[Well-posedness of~\eqref{eq:rd}]
		\label{thm:rd_new}
		Fix $\rho > \nicefrac{2}{n}$. Let $p \in (1, \RD_{+}(n,\rho))$ and $\alpha \in (-1,0)$ satisfying $\nicefrac{2}{\rho} = \nicefrac{n}{p} - \alpha$ and $p \geq p(\alpha,A)$. Fix $u_0 \in \dot\B^\alpha_{p,p}$.
		Then, the following hold:
		\begin{itemize}
			\item (Existence) There exists a weak solution $(u,\tau)$ of~\eqref{eq:rd} with $\tau > 0$.
			\item (Regularity) For every $r,q \in (1,\infty]$ satisfying $\RD_{-}(n,\rho) < r < \RD_{+}(n,\rho)$, $r \geq p$, $q > \RD_{-}(n,\rho)$ and $T \in (0,\tau)$, we have $u \in \Z^{r, q}_{\nicefrac{n}{2r} - \nicefrac{1}{\rho}}(T)$ and $\nabla u \in \Z^{p, 2}_{\nicefrac{n}{2p} - \nicefrac{1}{\rho} - \nicefrac{1}{2}}(T)$.
			\item (Uniqueness) For every $r,q \in (1,\infty]$ satisfying $\RD_{-}(n,\rho) < r < \RD_{+}(n,\rho)$, $r \geq p$, $q > \RD_{-}(n,\rho)$ and $T \in (0,\tau)$, the solution $(u,T)$ is unique among all weak solutions $(v,T)$ of~\eqref{eq:rd} satisfying $v \in \Z^{r, q}_{\nicefrac{n}{2r} - \nicefrac{1}{\rho}}(T)$ and $\nabla v \in \Z^{p,2}_{\nicefrac{n}{2p} - \nicefrac{1}{\rho} - \nicefrac{1}{2}}(T)$.
			\item (Maximality) The solution $(u,\tau)$ is maximal in any of the stated uniqueness classes.
		\end{itemize}
	\end{theorem}

	In this section, we first investigate mild solutions to~\eqref{eq:rd}. Afterwards, in Section~\ref{sec:weak}, we turn our attention to weak solutions. Eventually, we combine all these finding in Section~\ref{sec:lifespan} to complete the proof of Theorem~\ref{thm:rd_new}.

	\subsection{Existence of local mild solutions}
	\label{sec:existence_mild}

	We start with the existence of local mild solution to~\eqref{eq:rd}. In our context, a \emph{mild solution} is a fixed-point of the map $v \mapsto \Ext_L(u_0) + \Soll^L_1(\phi(v))$, where $v$ varies over a suitable solution class and $u_0$ is a fixed initial datum. The local existence of mild solutions will be provided by Banach's fixed-point theorem.

	Fix $\rho > \nicefrac{2}{n}$.
	Let $p \in (1, \RD_{+}(n,\rho))$ and $\alpha \in (-1, 0)$ such that $\nicefrac{2}{\rho} = \nicefrac{n}{p} - \alpha$ and $p \geq p(\alpha,A)$.
	Note that $p$ is finite, and fix $u_0 \in \dot\B^\alpha_{p,p}$.
	Now, fix some $r\in (\RD_{-}(n,\rho), \RD_{+}(n,\rho))$ satisfying $r\ge p$.
	{Such $r$ exists since the interval $(\RD_{-}(n,\rho), \RD_{+}(n,\rho))$ is non-empty owing to $\rho > \nicefrac{2}{n}$.}
	Note also that
	\begin{align}
		\label{eq:fujita_exp}
		r > \RD_{-}(n,\rho) > 1+\rho
	\end{align}
	by definition of $\RD_{-}(n,\rho)$ and the lower bound on $\rho$. In particular, $r > 1$ follows.

	For the given $u_0 \in \dot\B^\alpha_{p,p}$, we study the existence of local mild solutions to the non-linear reaction--diffusion equation
	\begin{align}
		\partial_t u -\Div(A\nabla u) = \phi(u), \quad u(0) = u_0,
	\end{align}
	where the non-linearity $\phi$ is a map $\R \to \R$ that satisfies for some constants $C_g, C_L >0$ the growth and Lipschitz estimates
	\begin{align}
		\label{eq:RD_growth_Lipschitz}
		|\phi(u)| \leq C_g |u|^{1+\rho}, \qquad
		|\phi(u) - \phi(v)| \leq C_L(|u|^\rho + |v|^\rho)|u-v|.
	\end{align}
	The coefficient function $A \colon \R^n \to \R^{n\times n}$ is measurable and satisfies~\eqref{eq:A_elliptic}.
	Now, for $\lambda, T \in (0,\infty)$ to be specified, define the complete metric set
	\begin{align}
		\X(\lambda, T) \coloneqq \Bigl\{ v \in \L^r_{\frac{n}{2r} - \frac{1}{\rho}}(T) \colon \| v \|_{\L^r_{\frac{n}{2r} - \frac{1}{\rho}}(T)} \leq \lambda \Bigr\},
	\end{align}
	where the distance on $\X(\lambda,T)$ is given by $\d(v,w) \coloneqq \| v - w \|_{\L^r_{\nicefrac{n}{2r} - \nicefrac{1}{\rho}}(T)}$.
	Put $L = -\Div(A\nabla)$.
	Define the fixed-point map
	\begin{align}
		\Theta(v) \coloneqq \Ext_L(u_0) + \Soll^L_1(\phi(v)), \qquad v\in \X(\lambda,T).
	\end{align}
	We are going to show below that, for suitable values of $\lambda$ and $T$, $\Theta$ is well-defined, preserves $\X(\lambda,T)$ and is a strict contraction. Thus, it admits a fixed-point, which is by definition a mild solution to~\eqref{eq:rd}.

	We start with some parameter relations which will simplify our life.

	\begin{lemma}
		\label{lem:p_A_rd}
		Fix $\rho > \nicefrac{2}{n}$. Let $p \in (1,\RD_{+}(n,\rho))$ and $\alpha \in (-1,0)$ satisfying $\nicefrac{2}{\rho} = \nicefrac{n}{p} - \alpha$ and $p \geq p(\alpha, A)$. Moreover, let $r \in (1,\RD_{+}(n,\rho))$ satisfying $r \geq p$. Put $\beta \coloneqq \tfrac{n(1+\rho)}{2r} - \tfrac{1}{\rho} - \tfrac{1}{2}$. Then, the following hold:
		\begin{align}
			\mathrm{(a)} \quad \beta > -\tfrac{1}{2}, \qquad \mathrm{(b)} \quad p \geq p(2\beta + 1, A), \qquad \mathrm{(c)} \quad p > \tfrac{r}{1+\rho}.
		\end{align}
	\end{lemma}

	\begin{proof}
		The first item is equivalent to $r < \RD_{+}(n,\rho)$, which we assume. In particular, $2\beta + 1 > 0 > \alpha$. Now, using that $p(\cdot,A)$ is decreasing, we find $p \geq p(\alpha, A) \geq p(2\beta + 1, A)$, which is the second item.

		For the third item, consider the critical line $L$ in the $(\nicefrac{1}{p},\alpha)$-plane passing through the points $(0, -\nicefrac{2}{\rho})$ and $(\nicefrac{2}{n\rho}, 0)$. By definition, $(\nicefrac{1}{p}, \alpha)$ belongs to $L$, and the same is true for $(\tfrac{1+\rho}{r}, 2\beta + 1)$. Since $2\beta + 1 > \alpha$, we must have $\tfrac{1+\rho}{r} > \tfrac{1}{p}$ since $L$ has positive slope, which concludes the proof.
	\end{proof}

	By the lemma, all restrictions in the linear theory involving the number $p(\cdot, A)$ will automatically be fulfilled in applications to~\eqref{eq:rd}. We are going to use this freely.

	Let us rephrase the mapping properties for $\Soll^L_1$ from Corollary~\ref{cor:Duhamel_Lq} in a convenient way.
	\begin{corollary}
		\label{cor:Duhamel_RD}
		The Duhamel operator $\Soll^L_1$ maps $\L^{\nicefrac{r}{1+\rho}}_{(\nicefrac{n}{2r} - \nicefrac{1}{\rho})(1 + \rho)}(T)$ into $\L^r_{\nicefrac{n}{2r} - \nicefrac{1}{\rho}}(T)$ with operator norm independent of $T$.
	\end{corollary}

	\begin{proof}
		Corollary~\ref{cor:Duhamel_Lq} applies to $\Soll^L_1$ when $\kappa=1$ provided
		the parameters $\beta^\sharp \coloneqq (\nicefrac{n}{2r} - \nicefrac{1}{\rho})(1 + \rho)$, $q^\sharp \coloneqq \nicefrac{r}{(1+\rho)}$ and $r^\sharp \coloneqq r$ satisfy the requirements there. %
		First, $\beta^\sharp > -1$ is equivalent to $r <\RD_{+}(n,\rho)$.
		Second, $\nicefrac{n}{2r} - \nicefrac{1}{\rho} =\beta^\sharp+\kappa {-\nicefrac{n}{2}(\nicefrac{1}{q^\sharp} - \nicefrac{1}{r^\sharp})}$.
		Third, $q^\sharp > 1$ by~\eqref{eq:fujita_exp}.
		Fourth, $r^\sharp \in \SobSet^{**}(q^\sharp)$ is implied by $r \geq \RD_{-}(n,\rho)$.
	\end{proof}

	For brevity, put $\Y(T)\coloneqq \L^{\nicefrac{r}{1+\rho}}_{(\nicefrac{n}{2r} - \nicefrac{1}{\rho})(1 + \rho)}(T)$. We proceed with the self-mapping property for $\Theta$.

	\begin{lemma}[Self-mapping property]
		There is $\lambda^* > 0$ such that, whenever $\lambda \leq \lambda^*$ and $T \leq T^*$ for some $T^*$ depending on $\lambda$, then $\Theta$ maps $\X(\lambda,T)$ into itself.
	\end{lemma}

	\begin{proof}
		We consider the free evolution and the Duhamel operator in the definition of $\Theta$ separately.

		\textbf{Step 1}: smallness property for the Duhamel operator. By Corollary~\ref{cor:Duhamel_RD} and the growth estimate in~\eqref{eq:RD_growth_Lipschitz}, calculate
		\begin{align}
			\| \Soll^L_1(\phi(v)) \|_{\X(T)} \leq C \| \phi(v) \|_{\Y(T)} \leq C C_g \| v \|_{\X(T)}^{1+\rho} \leq C C_g \lambda^{1+\rho},
		\end{align}
		where $C$ is a bound for the operator norm of $\Soll^L_1$ and is independent of $T$. Choose $\lambda \leq \lambda^*$ with $\lambda^*$ small enough so that $CC_g \lambda^\rho \leq \nicefrac{1}{2}$. Then, we obtain
		\begin{align}
			\| \Soll^L_1(\phi(v)) \|_{\X(T)} \leq \frac{1}{2} \lambda
		\end{align}
		independently of $T$, which completes Step~1. For the second step, consider $\lambda$ to be fixed.

		\textbf{Step 2}: smallness property for the free evolution. By Proposition~\ref{prop:caloric_extension_Besov}, since $r \geq p$, there is $T^* > 0$ depending on $\lambda$ and $u_0$ such that, for $T \leq T^*$, we have the estimate
		\begin{align}
			\| \Ext_L(u_0) \|_{\X(T)} \leq \frac{1}{2} \lambda.
		\end{align}
		Combining both steps gives the claim.
	\end{proof}

	\begin{lemma}[Contraction property]
		\label{lem:rd_contraction}
		There is $\lambda^* > 0$ such that, whenever $\lambda \leq \lambda^*$, then for all $T\in { (0,\infty)}$ the map $\Theta$ is a strict contraction on $\X(\lambda, T)$.
	\end{lemma}

	\begin{proof}
		Let $u,v \in \X(\lambda,T)$. Observe that $\Theta(u) - \Theta(v) = \Soll^L_1(\phi(u) - \phi(v))$. By virtue of Corollary~\ref{cor:Duhamel_RD}, the operator $\Soll^L_1$ is $\Y(T) \to \X(T)$ bounded. Hence, it suffices to control the $\Y(T)$ norm of $\phi(u) - \phi(v)$ accordingly. To this end, use the local Lipschitz estimate in~\eqref{eq:RD_growth_Lipschitz} and H\"older's inequality
		to give
		\begin{align}
			\| \phi(u) - \phi(v) \|_{\Y(T)} \leq C_L \Bigl( \| u \|_{\X(T)}^\rho + \| v \|_{\X(T)}^\rho \Bigr) \| u-v \|_{\X(T)}
			\leq C_L 2\lambda^\rho \| u-v \|_{\X(T)}.
		\end{align}
		Taking $\lambda$ sufficiently small depending on $C_L$ and the operator norm of $\Soll^L_1$ concludes the proof.
	\end{proof}

	Combining the last two lemmas gives the following result.

	\begin{proposition}[Local mild existence for~\eqref{eq:rd}]
		\label{prop:RD_existence_mild}
		Fix $\rho > \nicefrac{2}{n}$. Let $p \in (1, \RD_{+}(n,\rho))$ and $\alpha \in (-1, 0)$ such that $\nicefrac{2}{\rho} = \nicefrac{n}{p} - \alpha$ and $p \geq p(\alpha,A)$.
		Fix $u_0 \in \dot\B^\alpha_{p,p}$.
		Then, there exist {$r \in (\RD_{-}(n,\rho), \RD_{+}(n,\rho))$ satisfying $r \geq p$} and $T \in (0,\infty)$ such that~\eqref{eq:rd} admits a mild solution $u\in\L^r_{\nicefrac{n}{2r}-\nicefrac{1}{\rho}}(T)$.
	\end{proposition}

	\subsection{Bootstrapping regularity -- mild solutions}
	\label{sec:regularity_mild}

	Given a mild solution to~\eqref{eq:rd}, we establish further regularity properties of it. To do so, we employ an iteration argument using the equation. In particular, the artificial choice of the fixed-point set in the last section is justified.

	Fix $\rho > \nicefrac{2}{n}$ and $T \in (0,\infty]$. Let $p \in (1, \RD_{+}(n,\rho))$ and $\alpha \in (-1, 0)$ such that $\nicefrac{2}{\rho} = \nicefrac{n}{p} - \alpha$ and $p \geq p(\alpha,A)$. Fix $u_0 \in \dot\B^\alpha_{p,p}$.
	Assume that $u \in \Z^{\tilde r,\tilde q}_{\nicefrac{n}{2\tilde r} - \nicefrac{1}{\rho}}(T)$ is a mild solution to~\eqref{eq:rd}, that is to say, the identity
	\begin{align}
		\label{eq:bootstrapping_mild_solution_rd}
		u = \Ext_L(u_0) + \Soll^L_1(\phi(u))
	\end{align}
	holds in $\Z^{\tilde r, \tilde q}_{\nicefrac{n}{2\tilde r} - \nicefrac{1}{\rho}}(T)$. For the parameters we assume $$\tilde r, \tilde q \in (1, \infty], \quad \tilde r \geq p, \quad \RD_{-}(n,\rho) < \tilde r < \RD_{+}(n,\rho), \quad \tilde q > \RD_{-}(n,\rho).$$
	We are going to show the following result.

	\begin{proposition}[Bootstrapping regularity for~\eqref{eq:rd}]
		\label{prop:RD_bootstrapping}
		In the situation just described, one has $u \in \Z^{r,\infty}_{\nicefrac{n}{2r} - \nicefrac{1}{\rho}}(T)$ for all $r \in (1,\infty]$ satisfying $r \geq p$ and $\RD_{-}(n,\rho) < r < \RD_{+}(n,\rho)$. In particular, $u \in \Z^{r,q}_{\nicefrac{n}{2r} - \nicefrac{1}{\rho}}(T)$ for any $q \in (1,\infty]$.
	\end{proposition}

	The proof is quite technical. However, its idea is fairly simple: using the fact that~\eqref{eq:bootstrapping_mild_solution_rd} is an \emph{implicit} representation of $u$, self-improvement can be achieved by an iteration scheme using hypercontractivity of the Duhamel operator $\Soll^L_1$ combined with the weighted $\Z$-space embeddings.

	\begin{proof}
		Recall that we assume $u \in \Z^{\tilde r, \tilde q}_{\nicefrac{n}{2\tilde r} - \nicefrac{1}{\rho}}(T)$. We proceed in two steps.

		\textbf{Step 1}: We show that $u \in \Z^{r, \tilde q}_{\nicefrac{n}{2r} - \nicefrac{1}{\rho}}(T)$.
		If $r\ge \tilde r$, the claim of this step follows from to the weighted $\Z$-space embedding.
		Otherwise, assume $r<\tilde r$. We want to lower $\tilde r$ to $r$ by an induction procedure. Define $(r_j)_{j=0}^m$ as the minimal sequence of numbers satisfying
		\begin{align}
			r_0 \leq r, \qquad r_m = \tilde r, \qquad
			r_{j-1} = \frac{r_j}{1+\rho} \quad \text{for} \quad j = 1, \dots, m.
		\end{align}
		We know that $u\in \Z^{r_m, \tilde q}_{\nicefrac{n}{2r_m} - \nicefrac{1}{\rho}}(T)$. We claim that, for {$j = 1,\dots,m$}, one has
		\begin{align}
			\label{eq:bootstrapping_RD_induction}
			v\in \Z^{r_j, \tilde q}_{\frac{n}{2r_j} - \frac{1}{\rho}}(T) \qquad \Longrightarrow \qquad \Soll^L_1(\phi(v)) \in \Z^{r_{j-1}, \tilde q}_{\frac{n}{2r_{j-1}} - \frac{1}{\rho}}(T).
		\end{align}
		Indeed, $$v\in \Z^{r_j, \tilde q}_{\frac{n}{2r_j} - \frac{1}{\rho}}(T) \qquad \Longrightarrow \qquad \phi(v) \in \Z^{\frac{r_{j}}{1+\rho}, \frac{\tilde q}{1+\rho}}_{(\frac{n}{2r_j} - \frac{1}{\rho})(1+\rho)}(T) = \Z^{r_{j-1}, \frac{\tilde q}{1+\rho}}_{\frac{n}{2r_{j-1}} - \frac{1}{\rho} - 1}(T),$$
		and~\eqref{eq:bootstrapping_RD_induction} follows on applying Proposition~\ref{prop:Duhamel_mapping_Z} with $\Soll^L_1$ in the case $\kappa \coloneqq 1$, provided we check that $\beta\coloneqq(\nicefrac{n}{2r_j} - \nicefrac{1}{\rho})(1+\rho) > -1$
		and $\tilde q \in \SobSet^{**}(\nicefrac{\tilde q}{1+\rho})$. The verification is similar to Corollary~\ref{cor:Duhamel_RD}. First,
		 $(\nicefrac{n}{2r_j} - \nicefrac{1}{\rho})(1+\rho) > -1$ is equivalent to $r_j < \RD_{+}(n,\rho)$, which follows from $r_j \leq \tilde r < \RD_{+}(n,\rho)$. %
		 Second, $\tilde q \in \SobSet^{**}(\nicefrac{\tilde q}{1+\rho})$
		follows from $\tilde q \geq \RD_{-}(n, \rho)$.

		Concerning the free evolution, we have $\Ext_L(u_0) \in \Z^{r, \tilde q}_{\nicefrac{n}{2r} - \nicefrac{1}{\rho}}(T)$ by Proposition~\ref{prop:caloric_extension_Besov}, taking $r \geq p$ into account.
		Combined with the weighted $\Z$-space embedding of Proposition~\ref{prop:Z_space_embedding}, this yields in particular that $\Ext_L(u_0) \in \Z^{r_j, \tilde q}_{\nicefrac{n}{2r_j} - \nicefrac{1}{\rho}}(T)$ for $j = 1, \dots, m$, keeping $r_j \geq r$ in mind.

		Now, {by virtue of~\eqref{eq:bootstrapping_RD_induction} applied with $v= u$,} we can start a descending induction using the representation~\eqref{eq:bootstrapping_mild_solution_rd}, to arrive at
		$$\Soll^L_1(\phi(u)) \in \Z^{r_0, \tilde q}_{\frac{n}{2r_0} - \frac{1}{\rho}}(T) \subseteq \Z^{r, \tilde q}_{\frac{n}{2r} - \frac{1}{\rho}}(T),$$
		where we used the weighted $\Z$-spaces embedding in the last inclusion, taking $r \geq r_0$ into account.
		By the representation~\eqref{eq:bootstrapping_mild_solution_rd}, this concludes Step~1.

		\textbf{Step 2}: We
		show that $u \in \Z^{r,\infty}_{\nicefrac{n}{2r} - \nicefrac{1}{\rho}}(T)$.
		{We know $u \in \Z^{r, \tilde q}_{\nicefrac{n}{2r} - \nicefrac{1}{\rho}}(T)$ from Step~1.}
		Using a similar iteration scheme, we want to improve $\tilde q$ to $\infty$. We claim that it is possible to pick a finite, increasing sequence $(q_j)_{j=0}^m$ satisfying
		\begin{align}
			q_0 = \tilde q, \qquad q_m = \infty, \qquad
			q_{j+1} \in \SobSet^{**}\Bigl(\frac{q_j}{1+\rho}\Bigr)
			\quad \text{for} \quad j = 0, \dots, m-1.
		\end{align}
		This is possible using $\tilde q >\RD_{-}(n, \rho)$, since $ q^\flat \in [\tilde q,\infty)$ implies $\tfrac{1+\rho}{q^\flat} - \tfrac{2}{n+2} \leq \tfrac{1}{q^\flat} - \delta$ for a fixed \enquote{step size} $\delta > 0$.
		As we want $q_m = \infty$, the sequence has to be constructed in such a way that $\tfrac{q_{m-1}} {1+\rho} >\tfrac{n}{2} + 1$.
		We prove by induction
		that, for $j = 0,\dots, m$, one has
		$u\in \Z^{r, q_j}_{\nicefrac{n}{2r} - \nicefrac{1}{\rho}}(T)$.
		Step~1 furnishes the initial input $u \in \Z^{r, \tilde q}_{\nicefrac{n}{2r} - \nicefrac{1}{\rho}}(T) = \Z^{r, q_0}_{\nicefrac{n}{2r} - \nicefrac{1}{\rho}}(T)$.
		Next, assume the induction hypothesis holds up to $j\in \{0, \ldots, m-1\}$ and show it for $j+1$. %
		This and the initial input gives us
		\begin{align}
			\label{eq:phi_induction}
			\phi(u) \in \Z^{\frac{r}{1+\rho}, \frac{q_j}{1+\rho}}_{(\frac{n}{2r} - \frac{1}{\rho})(1+\rho)}(T).
		\end{align}
		Then, Proposition~\ref{prop:Duhamel_mapping_Z} applied to $\Soll^L_1$ yields
		\begin{align}
			\label{eq:bootstrapping_RD_step2}
			\Soll^L_1(\phi(u)) \in \Z^{\frac{r}{1+\rho}, q_{j+1}}_{\frac{n(1+\rho)}{2r} - \frac{1}{\rho}}(T),
		\end{align}
		provided we
	check that $\beta \coloneqq (\nicefrac{n}{2r} - \nicefrac{1}{\rho})(1+\rho) > -1$ and $q_{j+1} \in \SobSet^{**}(\nicefrac{q_j}{1+\rho})$.
	Indeed, $(\nicefrac{n}{2r} - \nicefrac{1}{\rho})(1+\rho) > -1$ was checked in Step~1 and
	$q_{j+1} \in \SobSet^{**}(\nicefrac{q_j}{1+\rho})$
	holds by choice of the sequence $(q_j)_j$.
	Eventually, using the weighted $\Z$-space embedding, we find $\Soll^L_1(\phi(u)) \in \Z^{r, q_{j+1}}_{\nicefrac{n}{2r} - \nicefrac{1}{\rho}}(T)$.

	Next, using that $r \geq p$, we find that $\Ext_L(u_0)$ belongs to the same space by Proposition~\ref{prop:caloric_extension_Besov}.
	This completes the proof of the induction. %
	As $q_{m}=\infty$, we deduce $u \in \Z^{r,\infty}_{\nicefrac{n}{2r} - \nicefrac{1}{\rho}}(T)$ as desired.
\end{proof}

	\subsection{Uniqueness of mild solutions}
	\label{sec:uniqueness_mild}

	In this section, we show uniqueness of mild solutions. In contrast to Section~\ref{sec:existence_mild}, mild solutions are not supposed to belong to a small ball in the contraction space, so that uniqueness is \emph{not} a consequence of Banach's contraction principle. Instead, we employ a contraction argument to show short-time coincidence of two given mild solutions, whereas their long-time coincidence relies on \emph{instantaneous regularization} and smoothing properties of Duhamel operators.

	In the following proposition, we show that there exists \emph{at most} one mild solution to the reaction--diffusion equation.

	\begin{proposition}[Uniqueness of mild solutions]
		\label{prop:mild_uniqneness_rd}
		Fix $\rho > \nicefrac{2}{n}$ and $T \in (0,\infty)$. Let $p \in (1, \RD_{+}(n,\rho))$ and let $\alpha \in (-1,0)$ satisfying $\nicefrac{2}{\rho} = \nicefrac{n}{p} - \alpha$ and $p \geq p(\alpha,A)$. Fix $u_0 \in \dot\B^\alpha_{p,p}$. Then, there exists at most one mild solution $u \in \L^r_{\nicefrac{n}{2r} - \nicefrac{1}{\rho}}(T)$ to~\eqref{eq:rd} with $r\in (1,\infty]$ satisfying $r \geq p$ and $\RD_{-}(n,\rho) < r < \RD_{+}(n,\rho)$.
	\end{proposition}

	Using the bootstrapping result for~\eqref{eq:rd} in Proposition~\ref{prop:RD_bootstrapping}, we immediately obtain the following corollary.

	\begin{corollary}
		\label{cor:mild_uniqueness_rd}
		The conclusion of Proposition~\ref{prop:mild_uniqneness_rd} stays true for $u \in \Z^{r,q}_{\nicefrac{n}{2r} - \nicefrac{1}{\rho}}(T)$ when $r,q\in (1,\infty]$ satisfy $r \geq p$, $\RD_{-}(n,\rho) < r < \RD_{+}(n,\rho)$ and $q > \RD_{-}(n,\rho)$.
	\end{corollary}

	\begin{proof}
		Let $u \in \Z^{r,q}_{\nicefrac{n}{2r} - \nicefrac{1}{\rho}}(T)$ be a mild solution to~\eqref{eq:rd}. By Proposition~\ref{prop:RD_bootstrapping}, we find $u\in \L^r_{\nicefrac{n}{2r} - \nicefrac{1}{\rho}}(T)$. Hence, Proposition~\ref{prop:mild_uniqneness_rd} entails that there is at most one such $u$.
	\end{proof}

	The strategy of proof is classical and was employed in other settings and for various types of equations, see for instance~\cite{Miura}.

	\begin{proof}[Proof of Proposition~\ref{prop:mild_uniqneness_rd}]
		Suppose that $u,v \in \L^r_{\nicefrac{n}{2\rho} - \nicefrac{1}{\rho}}(T)$ are mild solutions to~\eqref{eq:rd}. To ease notation, put $\X(t) \coloneqq \L^r_{\nicefrac{n}{2r} - \nicefrac{1}{\rho}}(t)$ for any $t>0$. We are going to show that $u = v$.

		\textbf{Step 1}: Coincidence for a short time. Subtract the two solutions to find $u - v = \Soll^L_1(\phi(u) - \phi(v))$. Arguing as in Lemma~\ref{lem:rd_contraction}, we find for $0 < \tau \leq T$ that
		\begin{align}
			\| u-v \|_{\X(\tau)} \leq C \Bigl( \| u \|_{\X(\tau)}^\rho + \| v \|_{\X(\tau)}^\rho \Bigr) \| u-v \|_{\X(\tau)},
		\end{align}
		where $C$ is a constant independent of $\tau$, $u$ and $v$. Take $\tau$ small enough (depending on $u$ and $v$) so that $C \bigl( \| u \|_{\X(\tau)}^\rho + \| v \|_{\X(\tau)}^\rho \bigr) \leq \frac{1}{2}$. It follows $u = v$ on $[0,\tau]$.

		\textbf{Step 2}: Coincidence for large times. It remains to show $u = v$ on $[\tau,T]$. To do so, we use the following instantaneous regularization property: there exists $M > 0$ such that
		\begin{align}
			\label{eq:RD_uniqueness_M}
			\| u \|_{\L^r(\tau,T)} + \| v \|_{\L^r(\tau,T)} \leq M.
		\end{align}
		Let $\eps > 0$ such that $\tau + \eps \leq T$.
		We are going to show that $u = v$ on $[\tau, \tau + \eps]$ when $\eps$ is small enough.
		To this end, define $q$ via $\nicefrac{1}{q} \coloneqq \nicefrac{1}{r} + \nicefrac{2}{n+2}$. From $r > \RD_{-}(n,\rho)$ and $\rho > \nicefrac{2}{n}$ it follows $q > 1$, and $r > \RD_{-}(n,\rho)$ yields $q < \nicefrac{r}{1+\rho}$. Hence, for a suitable $p^\sharp\in (1,\infty)$, we can write
		\begin{align}
			\label{eq:hoelder_rel}
			\frac{1}{q} = \frac{1 + \rho}{r} + \frac{1}{p^\sharp}.
		\end{align}
		From Corollary~\ref{cor:Duhamel_Lq} it follows that $\Soll^L_1 \colon \L^q_{-\nicefrac{2}{n+2}} \to \L^r$ with operator norm bounded by a constant $C$ independent of $T$. Hence, using $\phi(u) - \phi(v) = 0$ on $[0,\tau]$, Hölder's inequality with~\eqref{eq:hoelder_rel} and the local Lipschitz condition~\eqref{eq:RD_growth_Lipschitz}, we find for some $\nu > 0$ independent of $\eps$ that
		\begin{align}
			\| u-v \|_{\L^r(\tau, \tau + \eps)} &\leq C \| \phi(u) - \phi(v) \|_{\L^q_{-\frac{2}{n+2}}(\tau, \tau + \eps)} \\
			&\leq \eps^\nu C(T) \| \phi(u) - \phi(v) \|_{\L^\frac{r}{1+\rho}(\tau, \tau + \eps)} \\
			&\leq \eps^\nu C_L C(T) \Bigl( \| u \|_{\L^r(\tau, T)}^\rho + \| v \|_{\L^r(\tau, T)}^\rho \Bigr) \| u-v \|_{\L^r(\tau, \tau + \eps)},
		\end{align}
		where $C(T)$ is a constant depending on $T$.
		Hence, taking~\eqref{eq:RD_uniqueness_M} into account, deduce
		\begin{align}
			\| u-v \|_{\L^r(\tau, \tau + \eps)} \leq \eps^\nu C_L C(T) 2 M^\rho \| u-v \|_{\L^r(\tau, \tau + \eps)}.
		\end{align}
		Thus, we can take $\eps$ small to obtain
		$\| u-v \|_{\L^r(\tau, \tau + \eps)} \leq \frac{1}{2} \| u-v \|_{\L^r(\tau, \tau + \eps)}$ and hence $u = v$ on $[\tau, \tau + \eps]$. Since~\eqref{eq:RD_uniqueness_M} is a bound on $[\tau, T]$, we can iterate the argument with a uniform step size $\eps$ to eventually obtain $u = v$ on $[\tau, T]$ as desired.
	\end{proof}

	\section{Weak solutions for the reaction--diffusion problem}
	\label{sec:weak}

	So far, we have only treated mild solutions to~\eqref{eq:rd}. This is going to change now. In the current section, we will establish existence, uniqueness and regularity of weak solutions for it.

	\subsection{Existence of local weak solutions}
	\label{sec:weak_existence}

	The starting point in the construction of local weak solutions is, of course, the existence of local mild solutions, which we already have at hand. A central challenge is that we have to investigate their gradients now in order to identify them as weak solutions. To overcome this challenge, we employ the bootstrapping result for mild solutions to bring us back into a locally square-integrable setting.

	\begin{definition}[Weak solution]
		\label{def:weak_solution_rd} Let $\tau\in (0,\infty]$.
		Call $(u,\tau)$ a weak solution to~\eqref{eq:rd} if $u, \nabla u \in \L^2_\loc((0,\tau) \times \R^n)$, {$\phi(u) \in \L^{2_*}_\loc((0,\tau) \times \R^n)$} and $u$ is a weak solution to the equation $\partial_t u - \Div(A\nabla u) = \phi(u)$ on $(0,\tau) \times \R^n$ in the sense of Definition~\ref{def:weak_solution} which attains $u(0) = u_0$ in the sense of Definition~\ref{def:initial_condition}.

		If $\tau$ is clear from the context, we simply say that $u$ is a weak solution to~\eqref{eq:rd}.
	\end{definition}

	The lower Sobolev exponent $2_*$ of $2$ for source terms is a natural condition from scaling considerations.

	With a precise definition of a weak solution at hand, we establish existence of local weak solutions in the following proposition.

	\begin{proposition}[Existence of local weak solutions]
		\label{prop:existence_weak_rd}
		Fix $\rho > \nicefrac{2}{n}$, $p \in (1, \RD_{+}(n,\rho))$ and let $\alpha \in (-1,0)$ satisfying $\nicefrac{2}{\rho} = \nicefrac{n}{p} - \alpha$ and $p \geq p(\alpha,A)$. Fix $u_0 \in \dot\B^\alpha_{p,p}$. Let $r \in (1,\infty]$ satisfying $r \geq p$ and $\RD_{-}(n,\rho) < r < \RD_{+}(n,\rho)$.
		Then, there exists a weak solution $(u,T)$ of~\eqref{eq:rd} satisfying $u\in \Z^{r,\infty}_{\nicefrac{n}{2r} - \nicefrac{1}{\rho}}(T)$ and $\nabla u \in \Z^{p,2}_{\nicefrac{n}{2p} - \nicefrac{1}{\rho} - \nicefrac{1}{2}}(T)$.
		In particular, $u\in \Z^{r,q}_{\nicefrac{n}{2r} - \nicefrac{1}{\rho}}(T)$ for any $q\in (1,\infty]$.
	\end{proposition}

	\begin{proof}
		By Proposition~\ref{prop:RD_existence_mild}, there exist $T\in (0,\infty)$, $\tilde r \in (1,\infty]$ satisfying $\tilde r \geq p$ and $\RD_{-}(n,\rho) < \tilde r < \RD_{+}(n,\rho)$, and $u \in \L^{\tilde r}_{\nicefrac{n}{2\tilde{r}} - \nicefrac{1}{\rho}}(T)$ so that $u$ is a mild solution to~\eqref{eq:rd}. Using bootstrapping (Proposition~\ref{prop:RD_bootstrapping}), the regularity $u\in \Z^{r,\infty}_{\nicefrac{n}{2r} - \nicefrac{1}{\rho}}(T)$ follows.
		It remains to show that $u$ is a weak solution to~\eqref{eq:rd} satisfying $\nabla u \in \Z^{p,2}_{\nicefrac{n}{2p} - \nicefrac{1}{\rho} - \nicefrac{1}{2}}(T)$.

		Recall that by definition, $u = \Ext_L(u_0) + \Soll^L_1(\phi(u)) \eqqcolon v + w$. Owing to Proposition~\ref{prop:caloric_extension_Besov}, $v$ is a weak solution to
		\begin{align}
			\partial_t v - \Div(A\nabla v) = 0, \qquad v(0) = u_0,
		\end{align}
		satisfying $\nabla v\in \Z^{p,2}_{\nicefrac{n}{2p} - \nicefrac{1}{\rho} - \nicefrac{1}{2}}(T)$. Therefore, we are left to show that $w = \Soll^L_1(\phi(u))$ is a weak solution to
		\begin{align}
			\label{eq:existence_weak_rd1}
			\partial_t w - \Div(A\nabla w) = \phi(u), \qquad w(0) = 0,
		\end{align}
		satisfying $\nabla w \in \Z^{p,2}_{\nicefrac{n}{2p} - \nicefrac{1}{\rho} - \nicefrac{1}{2}}(T)$.
		To this end, we aim to invoke Theorem~\ref{thm:linear_wp}.
		By the nesting property of weighted $\Z$-spaces, there holds $u \in \Z^{r, 2_*(1+\rho)}_{\nicefrac{n}{2r} - \nicefrac{1}{\rho}}(T)$. Consequently
		\begin{align}
		\label{eq:bootstrapphiu}
\phi(u) \in \Z^{\frac{r}{1+\rho},2_*}_{(\frac{n}{2r} - \frac{1}{\rho})(1+\rho)}(T).
\end{align}
		Recall that $\nicefrac{r}{1+\rho} > 1$ by~\eqref{eq:fujita_exp}.
		Thus, with $f \coloneqq \phi(u)$, $\sigma \coloneqq \nicefrac{2}{\rho}$ and $\tilde r \coloneqq \nicefrac{r}{(1+\rho)}$, we can indeed apply Theorem~\ref{thm:linear_wp} as $\tilde r < \nicefrac{n}{\sigma}$ is equivalent to $r < \RD_{+}(n,\rho)$.
		Hence, $w$ is a weak solution to~\eqref{eq:existence_weak_rd1} satisfying $$\nabla w \in \Z^{\frac{r}{1+\rho},2}_{\frac{n(1+\rho)}{2r} - \frac{1}{\rho} - \frac{1}{2}}(T) \subseteq {\Z^{p,2}_{\frac{n}{2p} - \frac{1}{\rho} - \frac{1}{2}}(T),}$$
		where we used the weighted $\Z$-space embedding in conjunction with Lemma~\ref{lem:p_A_rd} in the last inclusion.
		Therefore, $u = v + w$ is by construction a weak solution to~\eqref{eq:rd} satisfying $\nabla u \in {\Z^{p,2}_{\nicefrac{n}{2p} - \nicefrac{1}{\rho} - \nicefrac{1}{2}}(T)}$. This completes the proof.
	\end{proof}

	\subsection{Bootstrapping regularity -- weak solutions}
	\label{sec:regularity_weak}

	On the side of mild solutions, we have developed a thorough regularity theory using bootstrapping methods in Section~\ref{sec:regularity_mild}. They were crucial in the passage from mild to weak solutions in Section~\ref{sec:weak_existence}. For instance, we have used bootstrapping for a mild solution $u$ to argue that $\phi(u) \in\Z^{\nicefrac{r}{1+\rho},2}_{(\nicefrac{n}{2r} - \nicefrac{1}{\rho})(1+\rho)}$ in \eqref{eq:bootstrapphiu}, which was an essential step to identify $u$ as a weak solution. In the next section, we are going in the opposite direction: we will show that a weak solution $u$ is also a mild solution, so that we can infer its uniqueness from Section~\ref{sec:uniqueness_mild}. To do so, the question whether $\phi(u) \in \Z^{\nicefrac{r}{1+\rho},2_*}_{(\nicefrac{n}{2r} - \nicefrac{1}{\rho})(1+\rho)}$ will pop up. A priori, this information is not always available. Using a bootstrapping procedure for weak solutions, we are going to close this gap in the present subsection.

	The tools to achieve improved regularity for weak solutions are super-linear reverse Hölder inequalities (we simply write RH-inequalities) and their self-improvement. To the best of our knowledge, this is a new principle that has not yet been employed in the literature.

	\subsubsection{Self-improvement of super-linear RH-inequalities}
	\label{subsec:rh_self_improvement}

	We formulate this section in the scope of a doubling metric measure space $(M, \dist, \mu)$. Here, \emph{doubling} means that there exists a constant $C>0$ such that $\mu(\B(x,2r)) \leq C \mu(\B(x,r))$ for every $x\in M$ and $r \in (0, \diam(M))$, where $\B(x,r)$ is a ball in $M$ of radius $r$ around $x$ with respect to the distance~$\dist$.
	\begin{example}
		\label{ex:parabolic_distance}
		The Euclidean half-space $\R^{1+n}_+$ equipped with the (dilated) parabolic distance $\dist((t,x),(s,y)) \coloneqq \max(2\sqrt{|t-s|}, |x-y|)$ and the usual Lebesgue measure is an example of a doubling metric measure space. Note that Whitney boxes are balls with respect to that distance.
	\end{example}
	In the Euclidean setting, the linear case of the following result goes back to~\cite{RH}. Our proof builds on the ideas presented in~\cite{BCF}, in which the linear case was generalized to doubling metric measure spaces.

	\begin{theorem}[Self-improvement of super-linear RH-inequalities]
		\label{thm:rh_improvement}
		Let $s^\sharp, p^\sharp \in (0,\infty]$ satisfying $s^\sharp < p^\sharp$, $q^\sharp, r^\sharp \in (s^\sharp, p^\sharp)$ and $\alpha^\sharp \geq 1$. Assume that there are $\beta_1, \beta_2 \in (0,1)$ such that
		\begin{align}
			\label{eq:rh_improvement_relations}
			\frac{1}{r^\sharp} = \frac{\beta_1}{p^\sharp} + \frac{1-\beta_1}{s^\sharp}, \qquad \frac{1}{q^\sharp} = \frac{\beta_2}{p^\sharp} + \frac{1-\beta_2}{s^\sharp}, \qquad \beta_2 \alpha^\sharp = \beta_1.
		\end{align}
		Let $B$ be a ball. Assume that there are $\lambda > 1$ and $v\in \L^{p^\sharp}(\lambda B)$ non-negative such that
		\begin{align}
			\label{eq:RH_assumption}
			\tag{RH}
			\left( \tiltfiint_{\wt B} v^{p^\sharp} \right)^\frac{1}{p^\sharp} \lesssim {\left( \tiltfiint_{2\wt B} v^{r^\sharp} \right)^\frac{1}{r^\sharp} + \left( \tiltfiint_{2\wt B} v^{q^\sharp} \right)^\frac{\alpha^\sharp}{q^\sharp}}
		\end{align}
		holds for all balls $\wt B$ satisfying {$2 \wt B \subseteq \lambda B$}.
		Then, there also holds the improved super-linear RH-inequality
		\begin{align}
			\left( \tiltfiint_{B} v^{p^\sharp} \right)^\frac{1}{p^\sharp} \lesssim \left( \tiltfiint_{\lambda B} v^{s^\sharp} \right)^\frac{1}{s^\sharp} + \left( \tiltfiint_{\lambda B} v^{s^\sharp} \right)^\frac{\theta}{s^\sharp},
		\end{align}
		where $$\theta \coloneqq \frac{\alpha^\sharp (1-\beta_2)}{1-\beta_1} \in [1,\infty).$$
	\end{theorem}

	\begin{proof}
		Let the ball $B$ and $\lambda > 1$ be as in the statement.
		For $\eps > 0$ define the functional
		\begin{align}
			K(\eps, s^\sharp) \coloneqq \sup_{B'} \left[ \frac{A(B')}{\wt A(\lambda B') + \eps A(\lambda B')}\right],
		\end{align}
		where the supremum is taken over all balls $B'$ satisfying $\lambda B' \subseteq \lambda B$
		and the quantities $A$ and $\tilde A$ are given by
		\begin{align}
			A(\hat B) &\coloneqq \left( \tiltfiint_{\hat B} v^{r^\sharp} \right)^\frac{1}{r^\sharp} + \left( \tiltfiint_{\hat B} v^{q^\sharp} \right)^\frac{\alpha^\sharp}{q^\sharp}, \\
			\wt A(\hat B) &\coloneqq \left( \tiltfiint_{\hat B} v^{s^\sharp} \right)^\frac{1}{s^\sharp} + \left( \tiltfiint_{\hat B} v^{s^\sharp} \right)^\frac{\theta}{s^\sharp},
		\end{align}
		{where $\hat B$ is any ball satisfying $\hat B \subseteq \lambda B$. Note that $B' \coloneqq B$ is an admissible choice in the definition of $K$, and that the quantities $A$ and $\tilde A$ are finite owing to the assumption $v\in \L^{p^\sharp}(\lambda B)$.}
		By definition, $K(\eps, s^\sharp)$ is bounded by a constant depending on $\eps$. The central task in the proof is to show that $K(\eps, s^\sharp)$ is bounded independently of $\eps$.

		{Fix $B'$ satisfying $\lambda B' \subseteq \lambda B$.}
		Cover {$B'$} by a family of balls $(B_i)_i$ satisfying%
		\begin{align}
			\label{eq:RH_covering}
			2\lambda B_i \subseteq \lambda B', \quad {\r(B_i) \approx \r(B')}, \quad \#_i \lesssim 1,
		\end{align}
		where implicit constants depend on $\lambda$ and the dimension.
		Note that, in particular, $2\lambda B_i \subseteq \lambda B' \subseteq \lambda B$.
		Using~\eqref{eq:rh_improvement_relations} in conjunction with Hölder's inequality and the definition of $\theta$, we obtain for every $i$ the estimate
		\begin{align}
			A(B_i) &\leq \left[ \Bigl( \tiltfiint_{B_i} v^{s^\sharp} \Bigr)^\frac{1-\beta_1}{s^\sharp} + \Bigl( \tiltfiint_{B_i} v^{s^\sharp} \Bigr)^\frac{\alpha^\sharp (1-\beta_2)}{s^\sharp} \right] \Bigl( \tiltfiint_{B_i} v^{p^\sharp} \Bigr)^\frac{\beta_1}{p^\sharp}
			\approx \wt A(B_i)^{1-\beta_1} \Bigl( \tiltfiint_{B_i} v^{p^\sharp} \Bigr)^\frac{\beta_1}{p^\sharp}.
		\end{align}
		Summing this bound in $i$ and using~\eqref{eq:RH_covering}, we find
		\begin{align}
			\label{eq:rh_improvement_1}
			A({B'}) \lesssim \sum_i A(B_i) \lesssim \sum_i \wt A(B_i)^{{1-\beta_1}} \Bigl( \tiltfiint_{B_i} v^{p^\sharp} \Bigr)^\frac{\beta_1}{p^\sharp}.
		\end{align}
		For fixed $i$, use~\eqref{eq:RH_assumption} and the definition of $K$ to estimate the last $p^\sharp$ average further by
		\begin{align}
			\Bigl( \tiltfiint_{B_i} v^{p^\sharp} \Bigr)^\frac{1}{p^\sharp} \lesssim A(2B_i) &\leq K(\eps, s^\sharp) \bigl[ \wt A(2\lambda B_i) + \eps A(2\lambda B_i) \bigr] \\
			&\lesssim K(\eps, s^\sharp) \bigl[ \wt A(\lambda {B'}) + \eps A(\lambda {B'}) \bigr].
		\end{align}
		Plugging this bound back into~\eqref{eq:rh_improvement_1}, deduce
		\begin{align}
			A({B'}) &\lesssim \sum_i \wt A(B_i)^{1-\beta_1} K(\eps, s^\sharp)^{\beta_1} \bigl[ \wt A(\lambda {B'}) + \eps A(\lambda {B'}) \bigr]^{\beta_1} \\
			&\lesssim \wt A(\lambda {B'})^{1-\beta_1} K(\eps, s^\sharp)^{\beta_1} \bigl[ \wt A(\lambda {B'}) + \eps A(\lambda {B'}) \bigr]^{\beta_1}.
		\end{align}
		Since $\wt A(\lambda {B'})^{1-\beta_1} [ \wt A(\lambda {B'}) + \eps A(\lambda {B'}) ]^{\beta_1} \leq \wt A(\lambda {B'}) + \beta_1 \eps A(\lambda {B'})$ by Young's inequality, we find by definition of $K$ {(using that $B'$ was arbitrary)} that
		\begin{align}
			K(\eps, s^\sharp) \leq K(\beta_1 \eps, s^\sharp) \lesssim K(\eps, s^\sharp)^{\beta_1},
		\end{align}
		where we used $\beta_1 \leq 1$ in the first step. Since even $\beta_1 < 1$, such a bound can only hold if $K(\eps, s^\sharp)$ is uniformly bounded in $\eps$. In particular, {specializing to $B' = B$,} this entails
		\begin{align}
			A(B) \leq C \left[ \wt A(\lambda B) + \eps A(\lambda B) \right],
		\end{align}
		where the constant $C$ is independent of $\eps$. Hence, taking $\eps \to 0$, we derive
		\begin{align}
			A(B) \leq C \wt A(\lambda B).
		\end{align}
		By definition of the quantities $A$ and $\wt A$, this gives the claim.
	\end{proof}

	{
	\begin{remark}
		The choice $\alpha^\sharp = 1$, and hence $r^\sharp = q^\sharp$, recovers the linear case from~\cite{BCF}.
	\end{remark}
	}

	\subsubsection{Consequences for reaction--diffusion equations}
	\label{subsec:regularity_weak_rd}

	Suppose that $(u,T)$ is a weak solution to~\eqref{eq:rd}. Fix $T' \in (0,T)$.
	Then, pick $\lambda = \sqrt{1 + \alpha} \in (1,3)$, where $0 < \alpha < \tfrac{4(T-T')}{T'}$.
	Recall from Example~\ref{ex:parabolic_distance} that Whitney boxes are balls with respect to the dilated parabolic distance.
	By construction, for $x\in \R^n$ and $0 < t \leq T'$, we have $\lambda \WB(t,x) \subseteq (0,T) \times \R^n$ compactly, which we are going to use frequently. In particular, for $(s,y) \in \lambda \W(t,x)$ we have $s \approx t$ with implicit constants depending on $\lambda$.

	In this setup, we obtain the following RH-inequality.

	\begin{lemma}[RH-inequality for~\eqref{eq:rd}]
		\label{lem:RH_rd}
		Fix $x\in \R^n$ and $t \in (0, T']$. Let $B'$ be an open ball with respect to the parabolic distance such that $2B' \subseteq \lambda \WB(t,x)$.
		Then
		\begin{align}
			\left( \tiltfiint_{B'} |s^{\nicefrac{1}{\rho}} u|^{2^*} \right)^\frac{1}{2^*}
			\lesssim \tiltfiint_{2B'} |s^{\nicefrac{1}{\rho}} u| + \left( \tiltfiint_{2B'} |s^{\nicefrac{1}{\rho}} u|^{2_*(1+\rho)} \right)^\frac{1}{2_*}.
		\end{align}
	\end{lemma}

	\begin{proof}
		Since $2B' \subseteq \lambda \WB(t,x) \subseteq (0,T) \times \R^n$ by choice of $\lambda$, $u$ is a weak solution to the linear problem $\partial_t u -\Div(A\nabla u) = \phi(u)$ on $2B'$. Hence, we can use the reverse Hölder inequality from~\cite[Prop.~4.4]{ABES19} with $f = \phi(u)$ and $\gamma = 2$, followed by the growth bound~\eqref{eq:RD_growth_Lipschitz}, to obtain
		\begin{align}
			\label{eq:RH1_rd}
			\left( \tiltfiint_{B'} |u|^{2^*} \right)^\frac{1}{2^*}
			&\lesssim \tiltfiint_{2B'} |u| + \r(B')^2 \left( \tiltfiint_{2B'} |\phi(u)|^{2_*} \right)^\frac{1}{2_*} \\
			&\lesssim{} \tiltfiint_{2B'} |u| + t \left( \tiltfiint_{2B'} |u|^{2_*(1+\rho)} \right)^\frac{1}{2_*},
		\end{align}
		where we also used $\r(B') \leq \r(\lambda \WB(t,x)) = \sqrt{\lambda t}$ in the last step. As $2B' \subseteq \lambda \WB(t,x)$, we find $s \approx t$ on $2B'$. Thus, multiplying~\eqref{eq:RH1_rd} by $t^{\nicefrac{1}{\rho}}$ and keeping $1 + \tfrac{1}{\rho} = \tfrac{1+\rho}{\rho}$ in mind, deduce
		\begin{align}
			\label{eq:RH2_rd}
			\left( \tiltfiint_{B'} |s^{\nicefrac{1}{\rho}} u|^{2^*} \right)^\frac{1}{2^*}
			\lesssim \tiltfiint_{2B'} |s^{\nicefrac{1}{\rho}} u| + \left( \tiltfiint_{2B'} |s^{\nicefrac{1}{\rho}} u|^{2(1+\rho)} \right)^\frac{1}{2}.
		\end{align}
		This completes the proof.
	\end{proof}

	\begin{remark}
		\label{rem:RH_rd}
		Since $u$ is a weak solution to~\eqref{eq:rd}, $u$ is locally square-integrable and $\phi(u)$ is locally $2_*$ integrable by definition. In particular, $u$ is locally integrable. Hence,~\eqref{eq:RH1_rd} entails local $\L^{2^*}$-integrability on $B'$.
		Eventually, {by varying $\lambda$ slightly in conjunction with a covering argument, we can show $\L^{2^*}$-integrability on $\lambda\WB(t,x)$.}
	\end{remark}

	Now, the following enhanced RH-inequality for~\eqref{eq:rd} can readily be deduced from our abstract result on self-improvement of super-linear RH-inequalities (Proposition~\ref{thm:rh_improvement}).

	\begin{proposition}[Improved RH-inequality for~\eqref{eq:rd}]
		\label{prop:rh_rd_improved}
		Let $\rho \in [0, \nicefrac{4}{n})$. Suppose that $(u,T)$ is a weak solution to~\eqref{eq:rd}. Fix $T' \in (0,T)$ and let $q \in (\RD_{-}(n,\rho), 2_*(1+\rho))$.
		Define $$\theta \coloneqq \frac{q(\tfrac{4}{n} - \rho)}{\tfrac{4q}{n} - 2^* \rho} > 1.$$
		Then, there exist $\lambda \in (1,3)$ and a constant $C = C(n,\rho,q,\lambda) > 0$ such that, for all $x\in \R^n$ and $t \in (0, T']$, the following improved RH-inequality holds:
		\begin{align}
			\left( \tiltfiint_{\W(t,x)} |s^{\nicefrac{1}{\rho}} u|^{2^*} \right)^\frac{1}{2^*}
			\leq{} C \Biggl[ \Bigl( \tiltfiint_{\lambda \WB(t,x)} |s^{\nicefrac{1}{\rho}} u|^q \Bigr)^\frac{\theta}{q} + \Bigl( \tiltfiint_{\lambda \WB(t,x)} |s^{\nicefrac{1}{\rho}} u|^q \Bigr)^\frac{1}{q} \Biggr].
		\end{align}
	\end{proposition}

	\begin{proof}
		We subdivide the proof into two steps.

		\textbf{Step 1}: setup parameters. In this tedious preparatory step, we set up the parameters for Proposition~\ref{thm:rh_improvement}. Define
		\begin{align}
			p^\sharp \coloneqq 2^*, \qquad q^\sharp \coloneqq 2_*(1+\rho), \qquad r^\sharp \coloneqq \frac{2_* q}{q - 2_* \rho}, \qquad s^\sharp \coloneqq q.
		\end{align}
		By $\rho < \nicefrac{4}{n}$, we obtain $p^\sharp > q^\sharp$, and we have $q^\sharp > s^\sharp$ by assumption on $q$.
		{Next, $r^\sharp > s^\sharp$ is equivalent to $q < 2_* (1+\rho)$.
		Lastly, $p^\sharp > r^\sharp$ is equivalent to $q > \RD_{-}(n,\rho)$.}
		In summary,
		\begin{align}
			{r^\sharp, q^\sharp \in (s^\sharp, p^\sharp).}
		\end{align}
		Thus, if we define
		\begin{align}
			\beta_1 \coloneqq \frac{p^\sharp(r^\sharp - s^\sharp)}{r^\sharp(p^\sharp - s^\sharp)} \in (0,1), \qquad \beta_2 \coloneqq \frac{p^\sharp(q^\sharp - s^\sharp)}{q^\sharp(p^\sharp - s^\sharp)} \in (0,1),
		\end{align}
		we obtain the Hölder relations
		\begin{align}
			\frac{1}{r^\sharp} = \frac{\beta_1}{p^\sharp} + \frac{1-\beta_1}{s^\sharp}, \qquad \frac{1}{q^\sharp} = \frac{\beta_2}{p^\sharp} + \frac{1-\beta_2}{s^\sharp}.
		\end{align}
		Put $\alpha^\sharp \coloneqq 1 + \rho$. The definition of $s^\sharp$ is designed in such a way that the crucial identity
		\begin{align}
			\label{eq:beta_relation_rd}
			\beta_2 \alpha^\sharp = \beta_1
		\end{align}
		holds. Lastly, we comment on the exponent $\theta$. Its numerator is positive since $\rho < \nicefrac{4}{n}$, and its denominator is positive since $\nicefrac{4q}{n} > \nicefrac{4}{n} \RD_{-}(n,\rho) = 2^* \rho$. Hence, $\theta > 1$ is equivalent to $q < 2^*$, which we have established above.

		\textbf{Step 2}: conclusion by Proposition~\ref{thm:rh_improvement}. Let $\lambda \in (1,3)$ be as in Lemma~\ref{lem:RH_rd}. Fix $x \in \R^n$ and $t \in (0, T']$, and set $B \coloneqq \WB(t,x)$. Define the function $v(s,y) \coloneqq s^{\nicefrac{1}{\rho}} |u(s,y)|$. By Remark~\ref{rem:RH_rd}, we have $v \in \L^{p^\sharp}(\lambda\WB(t,x))$. Moreover, using Jensen's inequality, Lemma~\ref{lem:RH_rd} provides~\eqref{eq:RH_assumption} for the parameters introduced in Step~1. All remaining assumptions of Proposition~\ref{thm:rh_improvement} were likewise checked in Step~1. Hence, we can apply the proposition to conclude.
	\end{proof}

	From the previous proposition, we easily obtain self-improvement on the level of weighted $\Z$-spaces for weak solutions to~\eqref{eq:rd}.

	\begin{corollary}[Bootstrapping of weak solutions to~\eqref{eq:rd}]
		\label{cor:bootstrapping_weak_rd}
		Let $\rho \in [0, \nicefrac{4}{n})$. Suppose that $(u,
		T)$ is a weak solution to~\eqref{eq:rd} such that $u\in \Z^{r,q}_{\nicefrac{n}{2r} - \nicefrac{1}{\rho}}(T)$ for some $r \in (1, \infty]$ and $q\in (\RD_{-}(n,\rho), 2_*(1+\rho))$. Then, $u \in \Z^{r,2^*}_{\nicefrac{n}{2r} - \nicefrac{1}{\rho}}(T')$ for all $T' \in (0,T)$.
		In particular, $u \in \Z^{r,2_*(1+\rho)}_{\nicefrac{n}{2r} - \nicefrac{1}{\rho}}(T')$ for all $T' \in (0,T)$.
	\end{corollary}

	\begin{proof}
		The definition of weighted $\Z$-spaces is invariant under dilation of the Whitney box, see~\cite[Rem.~3.6]{Luca}. Hence, multiplying the improved RH-inequality from Proposition~\ref{prop:rh_rd_improved} by $t^{-\nicefrac{n}{2r}}$ and applying the $\L^r$-norm to it, yields
		\begin{align}
			\| u \|_{\Z^{r,2^*}_{\frac{n}{2r} - \frac{1}{\rho}}(T')} \lesssim \| u \|_{\Z^{\theta r, q}_{\frac{n}{2\theta r} - \frac{1}{\rho}}(T)}^\theta + \| u \|_{\Z^{r, q}_{\frac{n}{2r} - \frac{1}{\rho}}(T)}.
		\end{align}
		By the weighted $\Z$-space embedding (using that $\theta > 1$) and the assumption, we have
		$$\| u \|_{\Z^{\theta r, q}_{\frac{n}{2\theta r} - \frac{1}{\rho}}(T)} \lesssim \| u \|_{\Z^{r, q}_{\frac{n}{2r} - \frac{1}{\rho}}(T)} \; \to \; 0 \qquad \text{as } T \to 0.$$
		This gives the main claim of the corollary.
		The \enquote{in particular} part follows from the nesting property of weighted $\Z$-spaces as $2_*(1+\rho) \leq 2^*$ is equivalent to $\rho \leq \nicefrac{4}{n}$.
	\end{proof}

	\begin{remark}
		In contrast to Proposition~\ref{prop:RD_bootstrapping}, note that bootstrapping for weak solutions allows to increase the local integrability, but \emph{not} to decrease the global integrability.
	\end{remark}

	\subsection{Uniqueness of weak solutions}
	\label{sec:weak_uniqueness}

	In Propositions~\ref{prop:existence_weak_rd}, we proved existence of local weak solutions for the reaction--diffusion equation for some time $T \in (0,\infty)$ depending on the initial value $u_0$ and the non-linearity. They belong to some regularity classes. Now, we study uniqueness of weak solutions in such classes. More precisely, given any time $T$, we wish to show that any such class contain \emph{at most} one weak solution.

	We reduce the uniqueness of weak solutions to the uniqueness of mild solutions. The latter was established in Corollary~\ref{cor:mild_uniqueness_rd}. We proceed in two steps.

	\begin{proposition}[Weak solutions to~\eqref{eq:rd} are mild solutions]
		\label{prop:weak_mild_rd}
		Fix $\rho > \nicefrac{2}{n}$ and $T\in (0,\infty)$. Let $p \in (1, \RD_{+}(n,\rho))$ and $\alpha \in (-1,0)$ satisfying $\nicefrac{2}{\rho} = \nicefrac{n}{p} - \alpha$ and $p \geq p(\alpha,A)$. Let $u_0 \in \dot\B^\alpha_{p,p}$.
		Suppose that $(u,T)$ is a weak solution to~\eqref{eq:rd} satisfying $u\in \Z^{r,q}_{\nicefrac{n}{2r} - \nicefrac{1}{\rho}}(T)$ and $\nabla u \in \Z^{p, 2}_{\nicefrac{n}{2p} - \nicefrac{1}{\rho} - \nicefrac{1}{2}}(T)$, where $r,q \in(1,\infty]$ satisfy $r \geq p$, $\RD_{-}(n,\rho) < r < \RD_{+}(n,\rho)$ and $q \geq \max(2_*(1+\rho),2^*)$.
		Then, $u$ is a mild solution to~\eqref{eq:rd}.
	\end{proposition}

	\begin{proof}
		Assume that $u$ is a weak solution to~\eqref{eq:rd} satisfying $u\in \Z^{r,q}_{\nicefrac{n}{2r} - \nicefrac{1}{\rho}}(T)$ and $\nabla u \in \Z^{p, 2}_{\nicefrac{n}{2p} - \nicefrac{1}{\rho} - \nicefrac{1}{2}}(T)$.
		Define $v \coloneqq \Ext_L(u_0) + \Soll^L_1(\phi(u))$. Arguing as in Proposition~\ref{prop:existence_weak_rd}, we have that $v$ is a weak solution to
		$$\partial_t v - \Div(A\nabla v) = \phi(u), \quad v(0) = u_0,$$
		satisfying $\nabla v \in \Z^{p,2}_{\nicefrac{n}{2p} - \nicefrac{1}{\rho} - \nicefrac{1}{2}}(T)$.
		Define $w \coloneqq u - v$. By construction, $w$ is a weak solution to $\partial_t w - \Div(A\nabla w) = 0$ with $w(0) = 0$. Moreover, $\nabla w \in \Z^{p,2}_{\nicefrac{n}{2p} - \nicefrac{1}{\rho} - \nicefrac{1}{2}}(T)$, using also the assumption on $\nabla u$. Therefore, by uniqueness of the linear Cauchy problem stated in Proposition~\ref{prop:linear_uniqueness}, $w = 0$.
		Thus, $u = v$. Plugging this back into the definition of $v$ shows that $u$ is a mild solution to~\eqref{eq:rd} as desired.
	\end{proof}

	From Proposition~\ref{prop:weak_mild_rd}, uniqueness of weak solutions follows rapidly.

	\begin{corollary}[Uniqueness of weak solutions to~\eqref{eq:rd}]
		\label{cor:uniqueness_weak_rd}
		Fix $\rho > \nicefrac{2}{n}$. Let $p \in (1, \RD_{+}(n,\rho))$ and $\alpha \in (-1,0)$ satisfying $\nicefrac{2}{\rho} = \nicefrac{n}{p} - \alpha$ and $p \geq p(\alpha,A)$. Fix $u_0 \in \dot\B^\alpha_{p,p}$.
		Then, for $T\in (0,\infty) $,
		there exists at most one weak solution $(u,T)$ to~\eqref{eq:rd} satisfying $u\in \Z^{r,q}_{\nicefrac{n}{2r} - \nicefrac{1}{\rho}}(T)$ and $\nabla u\in \Z^{p,2}_{\nicefrac{n}{2p} - \nicefrac{1}{\rho} - \nicefrac{1}{2}}(T)$, where $r,q \in(1,\infty]$ satisfy $r \geq p$, $\RD_{-}(n,\rho) < r < \RD_{+}(n,\rho)$ and $q > \RD_{-}(n,\rho)$.
	\end{corollary}

	\begin{proof}
		Assume that $(u,T)$ is a weak solution to~\eqref{eq:rd} satisfying $u\in \Z^{r,q}_{\nicefrac{n}{2r} - \nicefrac{1}{\rho}}(T)$ and $\nabla u\in \Z^{p,2}_{\nicefrac{n}{2p} - \nicefrac{1}{\rho} - \nicefrac{1}{2}}(T)$.
		Let $T' \in (0,T)$ arbitrary.
		Using Corollary~\ref{cor:bootstrapping_weak_rd}, we know moreover that $u\in \Z^{r,\max(q,2^*)}_{\nicefrac{n}{2r} - \nicefrac{1}{\rho}}(T')$.
		We claim that $\max(q,2^*) \geq \max(2_*(1+\rho),2^*)$. Indeed, the maximum on the right-hand side is attained by $2_*(1+\rho)$ if and only if $\rho \geq \nicefrac{4}{n}$. In this case, $q > \RD_{-}(n,\rho) \geq 2_* (1+\rho)$.
		Hence, by Proposition~\ref{prop:weak_mild_rd}, $u$ is a mild solution to~\eqref{eq:rd} on $[0,T']$. As $u \in \Z^{r,q}_{\nicefrac{n}{2r} - \nicefrac{1}{\rho}}(T')$ with $q>\RD_{-}(n,\rho)$, Corollary~\ref{cor:mild_uniqueness_rd} states that there is \emph{at most} one such mild solution. Hence, $u$ is unique on $[0,T']$. Since $T'$ was arbitrary, this proves the claim.
	\end{proof}

	\section{Conclusion of Theorem~\ref{thm:rd_new}}
	\label{sec:lifespan}

	We construct a weak solution $(u,\tau)$ to~\eqref{eq:rd} which is maximal in all uniqueness classes from Section~\ref{sec:weak_uniqueness} simultaneously. In particular, the lifespan is independent of the chosen uniqueness class. In conjunction with the results from Section~\ref{sec:weak_uniqueness}, this enables us to conclude Theorem~\ref{thm:rd_new}.

	\begin{proposition}[Maximal weak solutions]
		\label{prop:weak_maximal}
		Fix $\rho > \nicefrac{2}{n}$. Let $p \in (1, \RD_{+}(n,\rho))$ and $\alpha \in (-1,0)$ satisfying $\nicefrac{2}{\rho} = \nicefrac{n}{p} - \alpha$ and $p \geq p(\alpha,A)$. Fix $u_0 \in \dot\B^\alpha_{p,p}$.
		Then, there exists a weak solution $(u,\tau)$ to~\eqref{eq:rd} such that, for any $r,q \in(1,\infty]$ satisfying $r \geq p$, $\RD_{-}(n,\rho) < r < \RD_{+}(n,\rho)$ and $q > \RD_{-}(n,\rho)$, the following property holds:
		\begin{align}
			\label{eq:reg_maximal}
			\tag{R}
			\forall T \in (0, \tau) \colon \quad u\in \Z^{r,q}_{\frac{n}{2r} - \frac{1}{\rho}}(T) \; \; \& \; \; \nabla u\in \Z^{p,2}_{\frac{n}{2p} - \frac{1}{\rho} - \frac{1}{2}}(T).
		\end{align}
		Moreover, $(u,\tau)$ is maximal among all solutions verifying~\eqref{eq:reg_maximal}.
	\end{proposition}

	\begin{proof}
		Fix $\rho$, $p$, $\alpha$ and $u_0$ as in the statement. Let $(r,q)$ be some fixed pair of parameters as in the statement.

		\textbf{Step~1}: construction of a maximal solution $(u,\tau)$ for the parameter pair $(r,q)$. Consider the set
		\begin{align}
			\cM \coloneqq \bigl\{ (v,T) \text{ weak solution to } \eqref{eq:rd} \colon v \in \Z^{r,q}_{\frac{n}{2r} - \frac{1}{\rho}}(T), \; \nabla v \in \Z^{p,2}_{\frac{n}{2p} - \frac{1}{\rho} - \frac{1}{2}}(T) \bigr\}.
		\end{align}
		By Proposition~\ref{prop:existence_weak_rd}, we know $\cM \neq \emptyset$. Define the maximal existence time
		\begin{align}
			\tau \coloneqq \sup \, \bigl\{ T \colon (v,T) \in \cM \bigr\} > 0.
		\end{align}
		If $t \in (0,\tau)$, define $u(t) \coloneqq v(t)$, where $(v,T) \in \cM$ with $T \geq t$. By the uniqueness result in Corollary~\ref{cor:uniqueness_weak_rd}, this definition is unambiguous. By construction, $u$ is a weak solution on $(0,\tau)$. Moreover, $(u, \tau)$ is maximal among all weak solutions verifying~\eqref{eq:reg_maximal}. Indeed, otherwise there exists $(\tilde v, \tilde T) \in \cM$ with $\tilde T>\tau$, contradicting the definition of $\tau$.

		\textbf{Step~2}: solution $(u,\tau)$ is maximal for any other parameter pair $(\tilde r, \tilde q)$. Let $(\tilde r, \tilde q)$ be another pair of parameters as in the statement of the proposition.
		First, we show that $(u,\tau)$ verifies~\eqref{eq:reg_maximal} for the pair $(\tilde r, \tilde q)$. Let $T \in (0,\tau)$. Repeating the proof of Corollary~\ref{cor:uniqueness_weak_rd}, we find that $u$ is a mild solution to~\eqref{eq:rd} on $[0,T]$. Then, the bootstrapping result in Proposition~\ref{prop:RD_bootstrapping} entails $u \in \Z^{\tilde r,\tilde q}_{\nicefrac{n}{2\tilde r} - \nicefrac{1}{\rho}}(T)$. Eventually, arguing as in Proposition~\ref{prop:existence_weak_rd}, we also find $\nabla u \in \Z^{p,2}_{\nicefrac{n}{2p} - \nicefrac{1}{\rho} - \nicefrac{1}{2}}(T)$, which completes the proof of~\eqref{eq:reg_maximal}.

		Second, we show that $(u,\tau)$ is maximal for the pair $(\tilde r, \tilde q)$. For the sake of contradiction, suppose that $(v,T)$ is a weak solution to~\eqref{eq:rd} satisfying~\eqref{eq:reg_maximal} for the pair $(\tilde r, \tilde q)$ and such that $T > \tau$. Fix $T' \in (\tau, T)$. Replacing $u$ by $v$ in the argument at the beginning of Step~2, we deduce that $(v, T')$ is a weak solution to~\eqref{eq:rd} satisfying $v\in \Z^{r,q}_{\nicefrac{n}{2r} - \nicefrac{1}{\rho}}(T')$ and $\nabla v \in \Z^{p,2}_{\nicefrac{n}{2p} - \nicefrac{1}{\rho} - \nicefrac{1}{2}}(T')$. But this contradicts the maximality of $u$ for the pair $(r,q)$ obtained in Step~1. This concludes the proof.
	\end{proof}

	Eventually, we can give the proof to our non-linear main result.

	\begin{proof}[Proof of Theorem~\ref{thm:rd_new}]
		The existence, regularity and maximality assertions of Theorem~\ref{thm:rd_new} are provided by Proposition~\ref{prop:weak_maximal}.
		Moreover, uniqueness was investigated in Corollary~\ref{cor:uniqueness_weak_rd}. This completes the proof.
	\end{proof}

\end{document}